\documentclass[opre,sglanonrev]{informs4}
\usepackage{eqndefns-left} 
\usepackage{import}

\usepackage{import}
\usepackage{amsmath, amssymb, bbold, mathtools}
\usepackage{footnote}
\usepackage{acronym}
\usepackage{enumerate}
\usepackage{hyperref}
\usepackage[capitalise]{cleveref}
\usepackage{multirow}
\usepackage{subcaption}

\usepackage{graphicx}

\usepackage{natbib}
 \bibpunct[, ]{(}{)}{,}{a}{}{,}
 \def\bibfont{\small}

\usepackage{tikz}
\usetikzlibrary{shapes.geometric}
\usetikzlibrary{matrix,positioning}
\usepackage{pgfplots}
\pgfplotsset{compat=1.15}

\usepackage{mathrsfs}
\usetikzlibrary{arrows}
\usepackage{array}

\usepackage[noline,ruled]{algorithm2e}

\crefname{theorem}{Theorem}{Theorems}
\crefname{lemma}{Lemma}{Lemmas}
\crefname{proposition}{Proposition}{Propositions}
\crefname{corollary}{Corollary}{Corollaries}
\crefname{definition}{Definition}{Definitions}

\usepackage{algorithm}
\usepackage{algpseudocode}

\newcommand{\Eb}{\mathbb{E}}

\newcommand{\Prob}{{\rm{Prob}}}

\newcommand{\cX}{\mathcal{X}}
\newcommand{\cC}{\mathcal{C}}
\newcommand{\cD}{\mathcal{D}}
\newcommand{\cP}{\mathcal{P}}

\renewcommand{\Re}{\mathbb{R}}
\newcommand{\integ}{\mathbb{N}}

\newcommand{\abs}[1]{\left|#1\right|}

\newcommand{\dir}[1]{\mathrm{dir}\left(#1\right)}

\newcommand{\relint}[1]{\text{relint}\left(#1\right)}
\newcommand{\aff}[1]{\text{Aff}\left(#1\right)}

\newcommand{\R}[0]{\mathbb R}
\newcommand{\Ker}[0]{\mathrm{Ker}\;}

\renewcommand{\dim}[0]{\mathrm{dim}\;}
\newcommand{\Vect}{\mathrm{span}~}
\newcommand{\proj}[2]{\mathrm{proj}_{#1}\left(#2\right)}

\newcommand{\norm}[1]{\left\lVert#1\right\rVert}

\newcommand{\conv}{\mathrm{conv}}
\newcommand{\dataset}[0]{\mathcal D}
\newcommand{\sufficient}{sufficient decision dataset}
\newcommand{\Sufficient}{Sufficient Decision Dataset}
\newcommand{\suff}{sufficient}
\newcommand{\Suff}{Sufficient}

\newcommand{\cQ}{\mathcal Q}
\newcommand{\dualpoly}[2]{#1^\star\left(#2\right)}

\newcommand{\FD}{\mathrm{FD}}

\newcommand{\spc}[1]{\mathrm{dir}\left(#1\right)}
\newcommand{\ctrue}[0]{c_{\text{true}}}
\newcommand{\chat}[0]{\hat{c}}
\newcommand{\diam}[1]{\mathrm{diam}\left(#1\right)}
\newcommand{\xtrue}[0]{x_{\text{true}}}
\newcommand{\xhat}[0]{\hat{x}}

\renewcommand{\relint}[0]{\mathrm{relint}~}

\RequirePackage{tgtermes}
\RequirePackage{newtxtext}
\RequirePackage{newtxmath}
\RequirePackage{bm}
\RequirePackage{endnotes}

\OneAndAHalfSpacedXII 

\usepackage{algorithm}
\usepackage{algpseudocode}
\usepackage{tikz}

\usepackage{cleveref}
\crefname{appsec}{Appendix}{Appendices}
\Crefname{appsec}{Appendix}{Appendices}

\usepackage{natbib}
 \bibpunct[, ]{(}{)}{,}{a}{}{,}%
 \def\bibfont{\small}%
\EquationsNumberedThrough    

\TheoremsNumberedThrough     
\ECRepeatTheorems  %

\MANUSCRIPTNO{}

\begin{document}




\TITLE{Data Informativeness in Linear Optimization under Uncertainty}
\ARTICLEAUTHORS{%
\AUTHOR{Omar Bennouna}
\AFF{Laboratory of Information and Decision Systems,
Massachusetts Institute of Technology, \EMAIL{omarben@mit.edu}}

\AUTHOR{Amine Bennouna}
\AFF{Kellogg School of Management,
Northwestern University, \EMAIL{amine.bennouna@kellogg.northwestern.edu}}

\AUTHOR{Saurabh Amin}
\AFF{Laboratory of Information and Decision Systems,
Massachusetts Institute of Technology, \EMAIL{amins@mit.edu}}

\AUTHOR{Asuman Ozdaglar}
\AFF{Laboratory of Information and Decision Systems,
Massachusetts Institute of Technology, \EMAIL{asuman@mit.edu}}
} 

\ABSTRACT{We study the problem of determining what data is required to solve a decision-making task when only partial information about the state of the world is available. Focusing on linear programs, we introduce a decision-focused notion of data informativeness that formalizes when a data set is sufficient to recover the optimal decision. Our notion abstracts away the notion of estimators (how data is used): it depends solely on the structure of the optimization task and the uncertainty.
Our main result provides a geometric characterization of data sufficiency: a data set is sufficient if and only if, together with prior knowledge, it captures all cost directions that can change the optimal solution, given the task structure and the uncertainty set.
Building on our characterization, we develop a tractable algorithm to determine minimal sufficient data sets under general data collection constraints. Taken together, our work introduces a principled framework for task-aware data collection.
We demonstrate the approach in two applications: selecting where to conduct field experiments to inform infrastructure design and choosing which candidates to interview in order to make an optimal hiring decision. Our results illustrate that small, carefully selected data sets often suffice to determine the optimal decisions.}




\KEYWORDS{Data-Driven Decision-Making, Optimization Under Uncertainty, Experimentation, Data Collection, Data Informativeness} 

\maketitle


\section{Introduction}
\label{sec:intro}

Decision-making problems are often performed under incomplete knowledge of the state of nature---that is, they rely on parameters that must be learned or estimated. In practice, experts draw on a combination of domain knowledge and experience from previously solved tasks. With the recent surge in data availability, data-driven decision-making has become a dominant paradigm: data now plays a central role in complementing contextual knowledge to guide decisions.
This paper seeks to understand the fundamental informational value of a given data set with respect to a given decision-making task.
More precisely, the question here is to what extent does a given data set contribute to solving a particular decision-making problem?

The fundamental question of data informativeness has several important implications, including for data collection: when faced with a given task, what data should be collected to best inform this task? This question is related to several extensively studied topics
in economics, statistics, computer science, and operations research literature.  In what follows, we discuss three key aspects in which our approach differs from the state of the art, then provide an example motivating our perspective.

\textit{Estimation-focused vs. decision-focused.} The extensive literature on experimentation primarily focuses on estimating differences in outcomes between alternatives—namely, treatment effects. This body of work is well established and offers deep insights into which data is most informative for parameter estimation. However, in complex decision-making settings, estimation is typically not the ultimate objective. Instead, the central task is the decision itself. Our interest lies in understanding how data directly informs \textit{optimal decisions} rather than how accurately it enables \textit{parameter estimation}. Ultimately, the goal is to determine which data should be collected to most effectively guide decision-making. In other words, we seek to study task-specific experimentation and to understand how the structure of a particular decision task shapes the informativeness of data.

\textit{Adaptive vs. non-adaptive data collection.}
In many data-driven settings, informativeness is approached via \textit{adaptive}, sequential data collection. Active learning \citep{settlesActiveLearningLiterature2009,hannekeAgnosticActive2025,zheng2017active} seeks to sequentially select data points that improve a classifier by minimizing predictive loss, while bandit algorithms \citep{lattimoreBanditAlgorithms2020,carlsson2024pure} aim to optimize decisions through sequential exploration. Adaptive experimental design \citep{zhaoExperimentalDesign2024a} similarly selects experiments to maximize information gain about unknown parameters, often guided by Bayesian criteria such as posterior variance reduction. 
These approaches rely on real-time feedback to guide data acquisition. However, in many practical applications—such as large-scale surveys or field trials—queries must often be selected in advance, and outcomes are revealed only afterward. In
such settings, adaptivity is infeasible. Even in settings where adaptivity is theoretically possible, it can be operationally burdensome, requiring continuous monitoring and frequent adjustments that disrupt standard workflows. In this work, we seek to study informativeness in a \textit{non-adaptive} regime: we investigate how to strategically select queries up-front, before observing any outcomes, so as to most efficiently inform a complex decision-making task.

\textit{Universal vs. task-specific informativeness.} One of the earliest and most celebrated frameworks for comparing data sets is Blackwell's theory of informativeness \citep{blackwellEquivalentComparisonsExperiments1953}. In this framework, a data set is abstracted as an experiment, which generates a signal $s \in S$ drawn from a distribution \(P(s|\theta)\), informing on $\theta \in \Theta$, the unknown state of nature. Two data sets (experiments) are compared by whether one enables better decision-making across \emph{all loss functions} and priors. Formally, an experiment \(P\) is more informative than experiment \(Q\) if
$$
\inf_{\delta:S \to \cX} \Eb_{\theta \sim \pi,\; s \sim P(\cdot|\theta)}[L(\delta(s),\theta)] 
\leq 
\inf_{\delta:S \to \cX} \Eb_{\theta \sim \pi,\; s \sim Q(\cdot|\theta)}[L(\delta(s),\theta)], \quad \text{for all loss $L$ and prior $\pi$}
$$
Blackwell's seminal result shows that this criterion is equivalent to several elegant characterizations, notably through the notion of garbling \citep{deoliveiraBlackwellsinformativeness2018}.

Blackwell's informativeness criterion imposes a strict requirement: it compares data sets by whether they enable better decisions across \textit{all possible tasks}.
In practice, only a specific task is of interest. 
In this work, we seek to develop a \textit{task-specific} notion of informativeness grounded in the structure of the decision task itself. 
This aligns better with practical applications but also makes the informativeness question more delicate: as \citet{lecamComparisonExperimentsShort1996} observed, such questions may become ``complex or impossible depending on the statistician’s goal''.

In this paper, we seek to develop a \textit{non-adaptive decision-focused} notion of the informativeness of data. Ultimately, our goal is to address the following question: given a decision-making task and uncertainty on the parameters of that task, what data enables finding the optimal decision? 

\paragraph{An Illustrating Example.}
To illustrate the importance of this question, consider the following practical problem. Suppose a city is planning to build a new subway line from point A to point B. The objective is to identify the least costly route across the city. Prior data and contextual information provide preliminary estimates of construction costs across different areas, but these estimates are too uncertain to reliably determine the optimal route.
Building costs depend on numerous factors, including ground and geological conditions, land acquisition costs, ecological impact...
More accurate cost assessments for each potential “edge’’ in the city’s graph can be obtained through detailed engineering field studies of these factors. However, it is neither feasible nor desirable to conduct such studies for every possible edge.

The city must therefore strategically select which edges to investigate to determine the optimal subway line route. Because each field study may take several months, \textit{adaptivity is clearly infeasible}. Moreover, not all parts of the city are equally relevant. For instance, if both the origin and destination lie in the northern region, edges far to the south are unlikely to matter for the final route. Thus, \textit{the structure of the specific decision-making task} itself plays a crucial role in determining which data is most informative.

To the best of our knowledge, even within this specific application, no established algorithm provides a non-trivial solution to the problem. As an application of our work, we will formalize this problem in \cref{sec: subway example} and show how to rigorously determine the smallest set of edges that must be investigated to guarantee recovery of the optimal route.

\paragraph{Contributions.}

To derive precise theoretical insights, we focus on decision-making tasks formulated as linear programs with cost uncertainty modeled through a set. Within this class, we study informativeness in terms of recovering the optimal solution. Importantly, we abstract away the notion of estimators (how data is used) and instead define a data set as \emph{sufficient} if it contains the fundamental information required to recover the optimal decision. This problem is formally introduced in Section~\ref{sec:problem-formulation}.  Our contributions are as follows:


\begin{itemize}
    \item 
    \textbf{Complete Geometric Characterization.} 
We prove a necessary and sufficient condition (\cref{thm: relatively open characterization}) under which a data set is \textit{sufficient} to recover the optimal solution of a linear program under cost uncertainty. A data set is sufficient if and only if it spans a subspace of task-relevant directions in the parameter space, which we construct explicitly from the relative geometry of the task structure and the uncertainty set. This yields a direct characterization of how problem structure and uncertainty together shape the fundamental information required for decision-making.
We further show that the space of task-relevant directions can equivalently be expressed as the span of differences between optimal solutions under different cost vectors in the uncertainty set. This characterization (\cref{thm:span delta is dir x}) provides an algorithmically accessible formulation for evaluating and constructing sufficient data sets.

%


\item \textbf{Algorithms and Computational Complexity:} We turn to the data selection problem (Section~\ref{sec:algorithm}): Given a task and uncertainty, what is the smallest data set that suffices to recover the optimal decision? We show that the problem can be addressed in two-stages.
First, we introduce a tractable algorithm that computes the task-relevant subspace. When uncertainty is described by a polyhedral set, the algorithm terminates after a number of iterations equal to the size of the required data set, with each iteration involving the solution of a tractable mixed-integer program (Theorem~\ref{thm:alg_termination}). 
Second, we address the problem of constructing minimal data sets under given \textit{query constraints} (restrictions on which observations can be collected). 
From our characterization, this is a data set that spans the task-relevant subspace while verifying the query constraints.
We prove that this data selection problem is NP-hard in general (Propositions~\ref{hardness result}, \ref{prop: hardness Q basis}). However, when query sets have geometric structure (vector spaces, relatively open sets, convex sets, extreme points), we provide polynomial-time algorithms achieving the lower bound or exceeding it by at most one (Section~\ref{sec:query constraints}, Table~\ref{tab:hardness results}). This reveals that continuous structure enables efficient data set construction despite the combinatorial nature of the selection problem.



\item \textbf{Applications:}
We demonstrate the scope of our framework through two applications (Section~\ref{sec:applications}). First, we study the subway design problem, modeled as a shortest-path problem where the planner must decide which edges to inspect in order to determine the optimal route under uncertain construction costs. Second, we study a hiring application modeled as a non-adaptive secretary problem, in which a decision-maker must select which candidates to interview among a large pool so as to reveal enough information to make an optimal hiring decision.

\item \textbf{Extensions:} 
Our framework suggests several natural extensions towards decision-focused experimentation.
We discuss these extensions and open questions in Section \ref{sec:conclusion}. 
These include more general decision-making problems (beyond LPs), observational structures (potentially stochastic), uncertainty structures (in the constraints), and more. Furthermore, we prove that our characterization extends to combinatorial problems (\cref{cor:MIP-sufficiency}) and noisy observations (\cref{prop:bound on noisy case}).


\end{itemize}

\subsection{Related Work}



In addition to experimentation, bandits, active learning, and Blackwell's informativeness, we now discuss further conceptually related work.

\paragraph{Data Attribution, Influence Functions and Robust Statistics.}
Influence functions, originating in robust statistics \citep{huberRobustEstimation1992, hampelRobustStatisticsApproach1986}, quantify the local impact of individual data points on estimators and have recently received renewed interest \citep{broderickAutomaticFiniteSampleRobustness2023}.
Similar approaches include DataShapely \citep{ghorbaniDataShapley2019,kwonBetaShapley2022, jiangOpenDataValUnified2023, jiaEfficientData2023}, and Datamodels \citep{ilyasDatamodelsPredicting2022, dassDataMILSelecting2025, ilyasMAGICNearOptimal2025}.
These methods typically analyze how \textit{small perturbations} to a data set affect the output of a \textit{fixed estimator}. However, a key limitation of this approach is that data value is generally ``non-additive'': the informativeness of an individual data point is not intrinsic, but rather related to the data set as a whole. Our focus is on the joint informativeness of the full data set---characterizing when a collection of observations, as a whole, suffices to recover the task-optimal decision. Joint informativeness, combinatorial in nature, is a more challenging problem \citep{freundPracticalRobustnessAuditing2023,rubinsteinRobustnessAuditingLinear2024}. 
Furthermore, while influence functions assess sensitivity in \textit{estimation} problems under \textit{fixed} inference procedures, our framework evaluates data informativeness with respect to a \textit{decision task}, at a data set-level \textit{independently} of any specific inference or optimization procedure. That is, in contrast to estimator-specific influence methods, our work characterizes data set-level informativeness based solely on the structure of the decision task. 

\paragraph{Data Value in Operations Problems.} Recent works in the operations management literature aim to provide insights on the value of data in their respective settings. In pricing, for example, \cite{bahamouFastRevenue2024} computes the worst-case optimality ratio of decisions informed by a given data set and provides a procedure to identify the single most informative additional data point. Meanwhile, \cite{ahnOnlineDecisions2025} studies how (potentially biased) offline data can still improve online decision-making. In assortment optimization, \cite{liuMarginalValue2023} measures the marginal revenue gained from acquiring an additional data point. In inventory management, several studies evaluate the quality of data-driven decision performance as a function of the sample size of randomly generated demand data \cite{zhangMoreData2024,besbesHowBig2023,leviDataDrivenNewsvendor2015a}. 
Our work studies the smallest data set that can be strategically selected to inform the optimal decision. In particular, it exhibits how the nature of the problem, and uncertainty, determine what data matters.

\paragraph{Contextual Optimization.} In contextual optimization, as in our setting, a decision-maker aims to choose a decision $x \in \cX$ minimizing the loss $L(x,\theta)$, where $\theta$ is unknown \citep{sadana2023survey,hu2022fast,bertsimas2020predictive}. The decision maker also has access to side information $\phi \in \R^p$ that is correlated with $\theta$. Given empirical samples from the joint distribution of $\phi$ and $\theta$, the decision-maker needs to learn a policy $\pi$ that maps side information $\phi$ to an optimal decision $x \in \cX$.
Within this literature, much of the work—particularly in linear optimization—focuses on constructing data-driven surrogates of the unknown loss function, with the goal of improving decision quality rather than merely predicting losses. A prominent line of research in this direction is the \textit{predict-and-optimize} framework \citep{elmachtoub2022smart}. This paradigm is similar to ours in the sense that the aim is to directly focus on optimal decisions rather than predictions. However, the fundamental difference between contextual optimization and our setting is that our main concern is to understand \textit{how to select the most informative data set}, whereas in contextual optimization, the \textit{data is already given} and one must determine how to use it to produce an optimal decision policy. More recent work in contextual optimization has also considered \textit{adaptive} data-selection strategies \citep{liu2023active}.

\paragraph{Parametric Programming and Sensitivity Analysis.} This stream of work aims to understand how the optimal decision and value change when the underlying problem parameters are perturbed. Sensitivity analysis focuses on small, local perturbations, asking how far one can move in a given direction while preserving optimality (\cite{ward1990approaches,xuRobustSensitivity2017}). Multiparametric programming, by contrast, considers larger, structured changes in the parameters and aims to characterize the full mapping from parameters to optimal solutions, often by partitioning the parameter space into regions where the set of minimizers remains constant (\cite{gal1972multiparametric,saaty1954parametric}). The connection to our work lies in the shared goal of studying the interplay between problem parameters and optimal solutions. However, while sensitivity analysis and parametric programming aim to describe how solutions evolve as parameters vary, our focus is on identifying which data sets—i.e., which task parameters—are sufficient to recover the optimal solution.


\paragraph{Robust Optimization and Set-based vs. Distribution-based Modeling of Uncertainty.} 
In our problem, we chose to model uncertainty through a set (unknown parameter belonging to a known set) following the robust optimization paradigm. This is in contrast to modeling uncertainty as the unknown parameter following some known distribution, as in Bayesian optimization, for example. This modeling choice of uncertainty has been widely discussed in the robust optimization literature \citep{ben2009robust,bertsimas2011theory,delage2010distributionally}.

Set-based uncertainty has several practical advantages in our context compared to distribution-based uncertainty. For instance, set-based approaches typically rely on milder and often more realistic assumptions, as they do not require a fully specified probabilistic model of uncertainty. Instead, uncertainty is captured through bounds or confidence sets that are valid for a general class of distributions (such as with a given finite moments, or a given support). This is particularly appealing in settings where the true distribution is unknown, partially observed, or difficult to estimate reliably. Moreover, set-based formulations often lead to more tractable optimization problems and are less sensitive to model misspecification (see \cite{bertsimas2018data,ben2002robust}). For instance, in Bayesian Experimental Design, standard approaches require expensive computations to evaluate expected information gains in high-dimensional spaces and are highly sensitive to model misspecification, which can lead to suboptimal results \citep{rainforth2024modern}.

\paragraph{Inverse Optimization.} 
This line of work \citep{ahuja2001inverse, iyengar2005inverse, chan2019inverse, besbesContextualInverse2025} studies a related but distinct question: given a set of realized decisions, the goal is to identify an objective function (e.g., a cost vector in linear programs) under which the observed decisions are (near-)optimal. In our work, by contrast, the aim is to \emph{strategically select} observations that enable recovery of the optimal decision.

The distinction is twofold. First, in inverse optimization, the observations are \emph{given}, and the goal is to use them as effectively as possible; while in our setting, the goal is specifically to \textit{choose} the most informative observations. Second, inverse optimization seeks to infer the latent parameters of the model whereas our goal is to recover the optimal decision itself.
The focus is therefore on model recovery: explaining behavior by reconstructing a latent objective consistent with the data. While both these approaches connect data and optimality, they address fundamentally different problems.

Data informativeness also intersects with stochastic probing \citep{weitzmanOptimalSearch1979, singlaPriceInformationCombinatorial2018, gallegoConstructiveProphet2022}, optimal experimental design \citep{chalonerBayesianExperimental1995, singhApproximationAlgorithms2020}, and algorithms with predictions \citep{mitzenmacherAlgorithmsPredictions2020}, though a detailed comparison is beyond our scope.

\section{Problem Formulation}\label{sec:problem-formulation}
We study decision-making tasks modeled as linear programs (LPs). LPs model a wide range of problems, such as resource allocation, portfolio optimization, network flows, scheduling, and logistics problems. Additionally, they provide constant-factor approximations to broader classes of combinatorial problems. The decision-maker's task is to solve the LP 
{\setlength{\abovedisplayskip}{3pt}
 \setlength{\belowdisplayskip}{3pt}
 \setlength{\abovedisplayshortskip}{3pt}
 \setlength{\belowdisplayshortskip}{3pt}
\begin{align}
    \min_{x\in \cX}c^\top x,\label{mainproblem}
\end{align}
}
where $\cX=\{x\in \R^d,\; Ax=b,\; x\geq 0\}$ is a bounded polyhedron, with $A\in \R^{m\times d},\; b\in \R^m$.

\textbf{Uncertainty Model.} The unknown parameter---or state of nature---here is the cost vector $c$. The decision-maker only knows it to be in some given uncertainty set $\cC\subset \R^d$, which captures prior information on $c$. Naturally, this is equivalent to having uncertainty in the constraint right-hand side $b \in \Re^m$ by studying the dual problem instead.

\textbf{Data Collection Model.} To solve the linear program, the decision-maker can complement their knowledge $c \in \cC$ by data on the task.
A data set $\cD \subset \R^d$ consists of a set of queries to evaluate the objective function. That is, a data set gives access to the observations $c^\top q$ for $q\in \cD$. We focus on the noiseless setting, where each observation $c^\top q$ is exact. This simplification enables a sharp characterization of informativeness. We then show how the core insights naturally extend to noisy observations in \cref{prop:bound on noisy case}.

The fundamental question we seek to address is which data sets are \textit{sufficient} to solve the linear program. We formalize such a property next. 
Here $\cP(\cX)$ denotes subsets of $\cX$.

\begin{definition}[\sufficient] \label{def:sufficient}
    A data set $\cD:=\{q_1,\dots,q_N\}$ is a \sufficient\ for an uncertainty set $\cC\subset \R^d$ and decision set $\cX$ if there exists a mapping $\hat X:\R^{N} \longrightarrow \cal P(\cX)$ such that
    {\setlength{\abovedisplayskip}{3pt}
 \setlength{\belowdisplayskip}{3pt}
 \setlength{\abovedisplayshortskip}{3pt}
 \setlength{\belowdisplayshortskip}{3pt}
    \begin{align*}
        \forall c \in \cC, \quad \hat X\left(c^\top q_1,\dots,c^\top q_N\right)=\arg\min_{x\in \cX}c^\top x.
    \end{align*}
    }
When there is no ambiguity on $\cC$ and $\cX$, we simply say that $\mathcal D$ is a \suff.
\end{definition}

\cref{def:sufficient} states that a data set is \suff\ if there exists a mapping that can recover the optimal solution of the decision-making task using \textit{only} the data set’s observations $\{c^\top q,\; q\in \cD\}$ and prior information ($c \in \cC$). 
 Importantly, data collection here is \textit{non-adaptive}: the decision-maker has to commit to a set of queries in advance (rather than sequentially) and then decides based on all the revealed observations.

It is important to note that in \cref{def:sufficient}, we abstract the notion of algorithm $\hat{X}$ that specifies how the data is used. A data set is sufficient if \textit{there exists some algorithm} $\hat{X}$ that can extract from the data the necessary information to solve the optimization problem. Hence, sufficiency quantifies the intrinsic value of the data, independent of algorithmic/estimator choices. While this definition captures a fundamental notion of data value, it is hard to determine. Verifying whether a data set is sufficient requires (at least naively) testing \textit{every} possible algorithm $\hat{X}$, to check whether it solves the LP for \textit{every} instance $c \in \cC$. Such a procedure is clearly impossible. A key challenge we address in this paper is developing tractable characterizations of sufficiency.

Naturally, $\cD = \{e_1,\ldots,e_d\}$, where $(e_i)_{i\in [d]}$ are canonical basis vectors is a \sufficient. In fact, observing $c^\top e_i = c_i$ for all $i \in [d]$ amounts to fully observing $c$, and solving the linear program with complete information with $\hat{X}((c^\top q)_{q \in \cD}) = \hat{X}(c) := \argmin_{x \in \cX} c^\top x$. The question is then whether there exist other, potentially smaller \suff\ data sets. That is, what is the least amount of information required for the decision-making task? As we will show, whether a data set is \suff\ depends critically on the uncertainty set $\cC$ and the feasible region $\cX$. From a data collection perspective, the smallest such data sets $\cD$ are of interest. We refer to such data sets as minimal. In \cref{sec:algorithm}, we will discuss the problem of finding minimal data sets $\cD$, under query constraints $\cD \subset \cQ$. 

If the goal is to solve the linear program \eqref{mainproblem}, 
a natural relaxation of \cref{def:sufficient} is to only require that a data set permits recovery of \textit{some} optimal solution, rather than the \textit{entire set} of optimal solutions.
We show in the next proposition that, under mild structural assumptions, this property is equivalent to the property of Definition \ref{def:sufficient}. This means that any data set that recovers one solution also recovers all solutions. The proof of this equivalence is nontrivial and relies on several structural results we develop later in the paper.

\begin{proposition}[One vs All Optimal Solutions] \label{prop:equivalence-argmin}
        Let $\cC$ be an open convex set and $\cD:=\{q_1,\dots,q_N\}$ a data set. The following are equivalent:
        \begin{enumerate}
            \item \label{item:fullargmin} There exists a mapping $\hat X:\R^{N} \longrightarrow\cal P(\cX)$ such that
            \(
                \forall c \in \cC, \; \hat X\left(c^\top q_1,\dots,c^\top q_N\right)=\arg\min_{x\in \cX}c^\top x.
            \)

            \item \label{item:singlesol} There exists a mapping $\hat{x}:\R^N\longrightarrow \cX$ such that
                \(
                    \forall c\in \cC,\; \hat{x}\left(c^\top q_1,\dots,c^\top q_N\right)\in \arg\min_{x\in \cX}c^\top x. \label{condition for suff}
                \)
        \end{enumerate}
    
    \end{proposition}

\begin{remark}[Extensions and Open Problems]
    At this stage, several possible natural extensions of this model appear. For instance, one can consider more general optimization problems beyond linear optimization. One can also consider approximate optimality rather than exact optimality in \cref{def:sufficient}. Uncertainty can be in the task constraints (in $A$ for e.g.) rather than the cost. We discuss and formalize several such possible exciting directions and open problems in our concluding section, Section \ref{sec:conclusion}. For now, we will show that within this setting, we already obtain interesting insights for a wide class of decision-making problems: little, well-chosen data can significantly generalize decision-making knowledge.
\end{remark}

 Notice that observing $c^\top q$ for all $q\in \cD$ is equivalent to observing the projection of $c$ in the span of $\cD$. This implies that \cref{def:sufficient} is equivalent to the following characterization, which gives a valuable perspective. For any subspace $F\subset \R^d$ and $u\in \R^d$, we denote $u_F$ the projection of $u$ in $F$.
\begin{proposition}\label{prop:sufficient:projections}
    $\dataset :=\{q_1,\dots,q_N\}$ is a \sufficient\ for uncertainty set $\cC$ and decision set $\cX$ if and only if for all $c,c'\in \cC$, if $c_{\Vect \cD}=c'_{\Vect \cD}$, then $\arg\min_{x\in \cX}c^\top x = \arg\min_{x\in \cX}c'^\top x$.
\end{proposition}

In words, \cref{prop:sufficient:projections} states that a data set $\cD$ is \suff\ if any two cost vectors that are equivalent from the perspective of the information provided by $\cD$ (and $\cC$) lead to the same optimal solutions in the decision-making problem.

This characterization suggests a natural algorithm for solving the LP \eqref{mainproblem} when given a sufficient data set \(\mathcal{D} = \{q_1, \ldots, q_N\}\).
Suppose we observe values \(o_i = c^\top q_i, i\in [N]\) for an unknown cost vector \(c \in \mathcal{C}\). We then compute $\hat{c} \in \argmin \{ \sum_{i=1}^{N} (c'^\top q_i - o_i)^2 \; : \; c' \in \mathcal{C} \}$ and use \(\hat{c}\) to solve the decision problem \(\min_{x \in \mathcal{X}} \hat{c}^\top x\). This procedure recovers the projection of \(c\) onto \(\Vect \mathcal{D}\) while respecting the prior of \(\mathcal{C}\). This ensures \( \hat{c}_{\Vect \mathcal{D}} = c_{\Vect \mathcal{D}} \) as $c \in \cC$, and since the data set is sufficient, guarantees that the resulting decision is task-optimal (\cref{prop:sufficient:projections}).

\section{Characterizing \Suff\ data sets} \label{sec:characterizing-sufficient-data sets}
Given an uncertainty set $\cC$ and a decision set $\cX$, we would like to characterize \sufficient{}s and eventually construct such data sets. As in Blackwell's theory, the difficulty of such characterizations depends on the richness of the uncertainty set $\cC$. In fact, the first results by \cite{blackwellComparisonreconnaissances1949, blackwellComparisonExperiments1951} and 
\cite{shermanTheoremHardy1951} were for a set $\cC$ with only two elements. That is, the data needs to distinguish only two alternative states of nature. The result was later extended to the finite sets by \cite{blackwellEquivalentComparisonsExperiments1953} and then to infinite sets with regularity conditions by \cite{boll_comparison_1955}.

\subsection{Characterization for Vector Space Uncertainty Sets}

We begin with the case where $\cC$ is a subspace of $\R^d$, which already gives significant insight into how sufficient data sets depend on the uncertainty set $\cC$ and the decision set $\cX$ and in particular the no prior knowledge case, i.e., $\cC=\R^d$, which isolates how the structure of the decision set \(\mathcal{X}\) alone determines what information is necessary to recover the optimal solution. We will then study the case of convex sets. To formulate our result, define $F_0=\Vect\{e_i,\;i\in [d],\; \exists x\in \cX,\; x_i\neq 0\}$ where $e_i$ is the $i-$th element of the canonical basis. $F_0$ captures the coordinates that can take non-zero values in feasible solutions of $\cX$. That is $F_0^\perp$ captures coordinates that are identically zero in all feasible solutions: $e_i \in F_0^\perp \implies \forall x \in \cX, \; x_i=0$.

\begin{proposition} \label{prop:suff:vectorspace}
    Let $\cC$ be a subspace of $\R^d$ and $\dataset:=\{q_1,\dots,q_N\}\subset \R^d$. Let $F_0=\Vect\{e_i,\;i\in [d],\; \exists x\in \cX,\; x_i\neq 0\}$, and $\Ker A:=\{x\in \R^d, \: Ax=0\}$. The following statements are equivalent. 
    \begin{enumerate}
        \item $\dataset$ is a \sufficient.
        \label{item:prop:Cvector-space:suff}
        \item $F_0 \cap \Ker A \subset \cC^\perp +{\Vect \cD}$. \label{item:prop:Cvector-space:KerA}
        \item There exists $c \in \cC$ such that for any $c'\in \cC$ satisfying $c_{\Vect \cD}=c'_{\Vect \cD}$, we have $\arg\min_{x\in \cX}c^\top x$ and  $\arg\min_{x\in \cX}c'^\top x$ are equal. 
        \label{item:prop:Cvector-space:one-solution}
    \end{enumerate}
   Furthermore, when the condition $F_0 \cap \Ker A \subset \cC^\perp +{\Vect \cD}$ is not satisfied, for any mapping $\hat{x}:\R^N\longrightarrow \cX^\angle$, and any $K>0$, there exists $c\in \cC$ such that $c^\top \hat{x}\left(c^\top q_1,\dots,c^\top q_N\right)\geq K + \min_{x\in \cX}c^\top x$. 
\end{proposition}

Condition \ref{item:prop:Cvector-space:KerA} reveals three main components: 
\begin{enumerate}
    \item \emph{Task-relevant subspace} ($F_0 \cap \Ker A$): what we need to know. The objective decomposes as $c^\top x = c_{\Ker A}^\top x + c_{(\Ker A)^\perp}^\top x$. Since $\cX$ lies in an affine space parallel to $\Ker A$, any feasible direction $\delta$ satisfies $\delta \in \Ker A$, making $c_{(\Ker A)^\perp}^\top x$ constant across feasible solutions which implies only $c_{\Ker A}$ affects optimization. Furthermore, coordinates in $F_0^\perp$ are identically zero throughout $\cX$, so $c_{F_0^\perp}$ does not contribute to objective values. Hence, only $c_{F_0 \cap \Ker A}$ matters for optimization.  
    \item \emph{Prior knowledge} ($\cC^\perp$): what we already know. When $\cC$ is a vector space, $c \in \cC$ implies $c \perp \cC^\perp$. Any component along $\cC^\perp$ is already known to be zero before any observation.
    \item \emph{Data} ($\Vect \cD$): what we measure. The data set reveals $c_{\Vect \cD}$, the projection onto span$(\cD)$. 
\end{enumerate}

Combining these, the task-relevant information ($F_0 \cap \Ker A$) must be captured by what we know ($\cC^\perp$) and what we measure ($\Vect \cD$):
{\setlength{\abovedisplayskip}{2pt}
 \setlength{\belowdisplayskip}{2pt}
\begin{align*}       
\underbrace{F_0 \cap \Ker A}_{\text{what we need to know}} \subset \underbrace{\cC^\perp}_{\text{what we know}} +\underbrace{\Vect \cD}_{\text{what we measure}}
\end{align*}
}

\subsection{Characterization for Convex Uncertainty Sets}
The goal now is to characterize \suff\ data sets for more general uncertainty sets $\cC$. We start by introducing some geometric notions that are useful to understand the sufficiency of a data set.
\begin{definition}[Extreme Points]
An element $x\in \cX$ is an extreme point if and only if there are no $\lambda \in (0,1)$ and $y,z\in \cX$ such that $y\neq z$ and $x=\lambda y + (1-\lambda)z$.
\end{definition}

From every extreme point, there is a set of \textit{feasible directions} that allow changing the solution while remaining in the polyhedron $\cX$---the feasible region. Out of these feasible directions, \textit{extreme directions} allow moving to ``neighboring'' extreme points.

\begin{proposition}[Feasible and Extreme Directions] \label{feasible directions polyhedral cone}
     For every $x^\star\in \cX^\angle$, we denote 
     {\setlength{\abovedisplayskip}{1pt}
 \setlength{\belowdisplayskip}{1pt}
 \begin{align*}
     \FD(x^\star)=\{\delta \in \R^d,\; \exists \varepsilon>0,\;x^\star +\varepsilon \delta \in \cX\}
 \end{align*}
 }
     the set of feasible directions from $x^\star$ in $\cX$. $\FD(x^\star)$ is a polyhedral cone and $\FD(x^\star)\subset F_0 \cap \Ker A$. We denote $D(x^\star)$ the set of extreme directions of $\FD(x^\star)$.
     
\end{proposition}
The extreme directions $D(x^\star)$ represent the ``edges’’ emanating from $x^\star$ in the polyhedral structure of $\cX$, connecting $x^\star$ to its neighboring extreme points. In linear programs, optimal solutions are attained at extreme points $\cX^\angle$.
Every extreme point is associated with a set of cost vectors $c$ for which it is optimal. This set forms a cone, as illustrated in Figure \ref{fig: optimality cones}.
\begin{proposition}[Optimality Cones]\label{optimality cone def}
    For every $x^\star\in \cX^\angle$, we denote $\Lambda(x^\star)=\{c\in \R^d \; :\; x^\star \in \arg\min_{x\in \cX}c^\top x\}$. We have $\Lambda(x^\star)=\{c\in \R^d,\; \forall \delta \in D(x^\star),\ c^\top \delta\geq 0\}.$ For every $\delta \in D(x^\star),$ we denote $F(x^\star,\delta):=\Lambda(x^\star)\cap \{\delta\}^\perp$ the face of the cone $\Lambda(x^\star)$ that is perpendicular to $\delta$. Furthermore, $\Lambda(x^\star)$ is the dual cone of $\FD(x^\star)$.
\end{proposition}
Notice that since $\cX$ is bounded, for any $c\in \R^d$, there exists $x^\star \in \cX^\angle$ such that $c\in \Lambda (x^\star)$, and consequently $\R^d=\bigcup_{x^\star \in \cX^\angle}\Lambda(x^\star)$ as illustrated in \cref{fig: optimality cones}. Neighboring cones share boundaries corresponding to their faces (\cref{fig: optimality cones}, middle), where multiple solutions can be optimal.

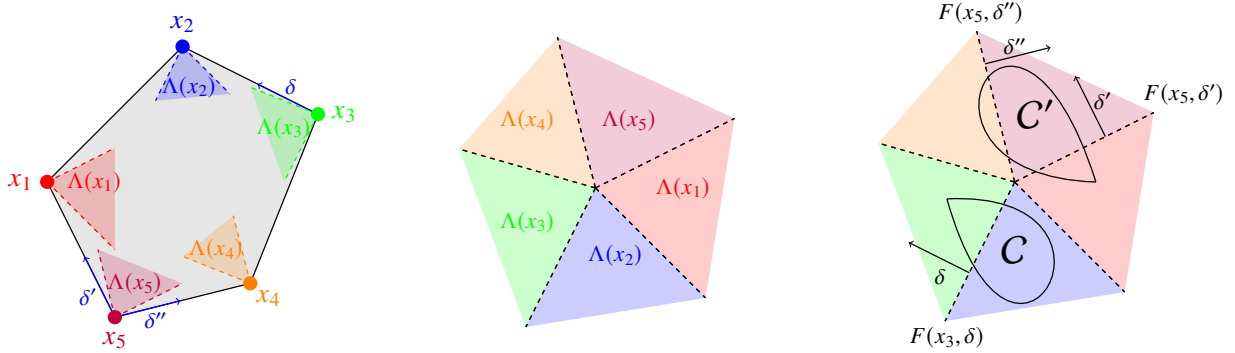
\begin{figure}[h] 
    \centering
    \begin{minipage}{0.33\textwidth} 
        \centering
        \scalebox{0.6}{              \begin{tikzpicture}[scale=1.5]

\coordinate (A) at (0,0);
\coordinate (A1) at (1,-2/2);
\coordinate (A2) at (1,1/2);
\coordinate (B) at (2,2);
\coordinate (B1) at (2 -1*0.4,2-2*0.4);
\coordinate (B2) at (2 +1*0.7,2 -1*0.7);
\coordinate (C) at (4,1);
\coordinate (C1) at (4 -1*0.5,1-2*0.5);
\coordinate (C2) at (4 -2.5*0.4,1 +1*0.4);
\coordinate (D) at (3,-1.5);
\coordinate (D1) at (3 -2.5*0.4,-1.5 +1*0.4);
\coordinate (D2) at (3 -1/2 * 0.5,-1.5 +2 *0.5);
\coordinate (E) at (1,-2);
\coordinate (E1) at (1 +2 * 0.5,-2+1 * 0.5);
\coordinate (E2) at (1-1/2 * 0.5,-2+2 *0.5);

\node[anchor=east, text=red, xshift=-5pt] at (A) {\Large $x_1$};
\node[anchor=south, text=blue, yshift=5pt] at (B) {\Large $x_2$};
\node[anchor=west, text=green,xshift=5pt] at (C) {\Large $x_3$};
\node[anchor=north west, text=orange] at (D) {\Large $x_4$};
\node[anchor=north, text=purple, yshift=-5pt] at (E) {\Large $x_5$};

\draw[thick, fill=gray!20] (A) -- (B) -- (C) -- (D) -- (E) -- cycle;

\draw[ thick, blue, ->] (1,-2) -- (1 - 0.447,-2 + 0.894);
\draw[ thick, blue, ->] ( 4,1 ) -- ( 4 -0.894 ,1 +0.447 );

\draw[ thick,blue, ->] ( 1,-2)-- ( 1 +0.970,-2 + 0.244);

\node[anchor=center,blue, font=\large] at (3.6,1.4) {$\delta$};
\node[anchor=center,blue, font=\large] at (0.6,-1.7) {$\delta'$};
\node[anchor=center,blue, font=\large] at (1.6,-2.05) {$\delta''$};

\fill[red] (A) circle (3pt);
\fill[blue] (B) circle (3pt);
\fill[green] (C) circle (3pt);
\fill[orange] (D) circle (3pt);
\fill[purple] (E) circle (3pt);

\draw[red, thick, dashed] (A) -- (A1);
\draw[red, thick, dashed] (A) -- (A2);
\fill[red, opacity=0.2] (A) -- (A1) -- (A2) -- cycle;

\draw[blue, thick, dashed] (B) -- (B1);
\draw[blue, thick, dashed] (B) -- (B2);
\fill[blue, opacity=0.2] (B) -- (B1) -- (B2) -- cycle;

\draw[green, thick, dashed] (C) -- (C1);
\draw[green, thick, dashed] (C) -- (C2);
\fill[green, opacity=0.2] (C) -- (C1) -- (C2) -- cycle;

\draw[orange, thick, dashed] (D) -- (D1);
\draw[orange, thick, dashed] (D) -- (D2);
\fill[orange, opacity=0.2] (D) -- (D1) -- (D2) -- cycle;

\draw[purple, thick, dashed] (E) -- (E1);
\draw[purple, thick, dashed] (E) -- (E2);
\fill[purple, opacity=0.2] (E) -- (E1) -- (E2) -- cycle;

\node[anchor=center, text=red, font=\large] at (0.7, 0) {$\Lambda(x_1)$};
\node[anchor=center, text=blue, font=\large] at (2.1, 1.4) {$\Lambda(x_2)$};
\node[anchor=center, text=green, font=\large] at (3.5, 0.8) {$\Lambda(x_3)$};
\node[anchor=center, text=orange, font=\large] at (2.5, -1) {$\Lambda(x_4)$};
\node[anchor=center, text=purple, font=\large] at (1.3,-1.5) {$\Lambda(x_5)$};

\end{tikzpicture}}
        \end{minipage}\hfill
    \begin{minipage}{0.33\textwidth} 
        \centering
        \scalebox{0.6}{        \begin{tikzpicture}[scale=1.5]

\coordinate (O) at (0,0);
\coordinate (A1) at (0.707*2.3,-0.707*2.3);
\coordinate (A2) at (0.894*2.3,0.447*2.3);

\coordinate (B1) at ( -0.447*2.3,-0.894*2.3);
\coordinate (B2) at (A1);

\coordinate (C1) at (B1);
\coordinate (C2) at ( -0.875*2.3, 0.25*2.3);

\coordinate (D1) at (C2);
\coordinate (D2) at ( -0.244*2.3,0.970*2.3);

\coordinate (E1) at (A2);
\coordinate (E2) at (D2);



\draw[ thick, dashed] (O) -- (A1);
\draw[ thick, dashed] (O) -- (A2);
\fill[red, opacity=0.2] (O) -- (A1) -- (A2) -- cycle;

\draw[ thick, dashed] (O) -- (B1);
\draw[ thick, dashed] (O) -- (B2);
\fill[blue, opacity=0.2] (O) -- (B1) -- (B2) -- cycle;

\draw[thick, dashed] (O) -- (C1);
\draw[ thick, dashed] (O) -- (C2);
\fill[green, opacity=0.2] (O) -- (C1) -- (C2) -- cycle;

\draw[ thick, dashed] (O) -- (D1);
\draw[ thick, dashed] (O) -- (D2);
\fill[orange,opacity=0.2] (O) -- (D1) -- (D2) -- cycle;

\draw[ thick, dashed] (O) -- (E1);
\draw[ thick, dashed] (O) -- (E2);
\fill[purple, opacity=0.2] (O) -- (E1) -- (E2) -- cycle;

\node[anchor=center, text=red, font=\large] at (1.3, 0) {$\Lambda(x_1)$};
\node[anchor=center, text=blue, font=\large] at (0.3, -1) {$\Lambda(x_2)$};
\node[anchor=center, text=green, font=\large] at (-1, -0.5) {$\Lambda(x_3)$};
\node[anchor=center, text=orange, font=\large] at (-1, 1) {$\Lambda(x_4)$};
\node[anchor=center, text=purple, font=\large] at (0.5,1) {$\Lambda(x_5)$};
\end{tikzpicture}}
        \end{minipage}
        \begin{minipage}{0.33\textwidth} 
        \centering
        \scalebox{0.6}{\begin{tikzpicture}[overlay,scale=1.5]

\coordinate (O) at (0,0);
\coordinate (A1) at (0.707*2.3,-0.707*2.3);
\coordinate (A2) at (0.894*2.3,0.447*2.3);

\coordinate (B1) at ( -0.447*2.3,-0.894*2.3);
\coordinate (B2) at (A1);

\coordinate (C1) at (B1);
\coordinate (C2) at ( -0.875*2.3, 0.25*2.3);

\coordinate (D1) at (C2);
\coordinate (D2) at ( -0.244*2.3,0.970*2.3);

\coordinate (E1) at (A2);
\coordinate (E2) at (D2);



\draw[ thick, dashed] (O) -- (A1);
\draw[ thick, dashed] (O) -- (A2);
\fill[red, opacity=0.2] (O) -- (A1) -- (A2) -- cycle;

\draw[ thick, dashed] (O) -- (B1);
\draw[ thick, dashed] (O) -- (B2);
\fill[blue, opacity=0.2] (O) -- (B1) -- (B2) -- cycle;

\draw[thick, dashed] (O) -- (C1);
\draw[ thick, dashed] (O) -- (C2);
\fill[green, opacity=0.2] (O) -- (C1) -- (C2) -- cycle;

\draw[ thick, dashed] (O) -- (D1);
\draw[ thick, dashed] (O) -- (D2);
\fill[orange,opacity=0.2] (O) -- (D1) -- (D2) -- cycle;

\draw[ thick, dashed] (O) -- (E1);
\draw[ thick, dashed] (O) -- (E2);
\fill[purple, opacity=0.2] (O) -- (E1) -- (E2) -- cycle;

\node[anchor=center, font=\Huge] at (0,-1) {$\cC$};
\node[anchor=center, font=\Huge] at (0.3,1) {$\cC'$};

\coordinate (O1) at ( -1,-0.25);
\coordinate (O2) at (2,-0.1);
\coordinate (O3) at (0,-3.8);
\coordinate (O4) at (-1,-1.6);

\coordinate (V1) at ( 1.2,0);
\coordinate (V2) at (-2,0.1);
\coordinate (V3) at (-0,3.8);
\coordinate (V4) at (1,1.6);
\draw[thick] 
    (V1) .. controls (V2) and (V3) ..  (V1) ;

\draw[thick] 
    (O1) .. controls (O2) and (O3) ..  (O1) ;

\draw[ thick, ->] ( -0.447*1.5,-0.894*1.5 ) -- ( -0.447*1.5 -0.894 ,-0.894*1.5 +0.447 );

\draw[ thick, ->] (0.894*1.5,0.447*1.5) -- (0.894*1.5 - 0.447,0.447*1.5 + 0.894);
\draw[ thick, ->] ( -0.244*1.8,0.970*1.8)-- ( -0.244*1.8 +0.970,0.970*1.8 + 0.244);
\node[anchor=center, font=\large] at (-1,-2.3) {$F(x_3,\delta)$};
\node[anchor=center, font=\large] at (-1.1,-1.4) {$\delta$};
\node[anchor=center, font=\large] at (1.3,1.2) {$\delta'$};
\node[anchor=center, font=\large] at (0,2) {$\delta''$};

\node[anchor=center, font=\large] at (-0.5,2.5) {$F(x_5,\delta'')$};
\node[anchor=center, font=\large] at (2.5,1.3) {$F(x_5,\delta')$};

\end{tikzpicture}}
        \end{minipage}\hfill
  \caption{\emph{Left}: Optimality cones relative to $\cX$: each extreme point has an associated cone of cost vectors making it optimal. \emph{Middle}: The same cones at the origin, illustrating their dual relationship with feasible directions. \emph{Right}: Uncertainty sets $\cC$ and $\cC'$ intersecting different cone collections. Data must distinguish between cones that overlap with the uncertainty set.}
    \label{fig: optimality cones}
\end{figure}


With the notion of optimality cones, solving a linear program for a given cost vector $c$ amounts to finding to which optimality cone it belongs. A data set is therefore sufficient if it enables to determine the optimality cone of each possible \(c \in \mathcal{C}\). As $\cC$ already restricts the location of its cost vectors, our data only needs to discriminate between cones overlapping with $\cC$ (see \cref{fig: optimality cones}).


To provide further intuition, consider the example sketched in the right of Figure \ref{fig: optimality cones}. The set $\cC$ intersects only the cones $\Lambda(x_2)$ (blue) and $\Lambda(x_3)$ (green), hence, the cost vectors can only be in these two cones. Clearly, observing their projection on the span of the extreme direction $\delta$ is sufficient to determine which of the two cones they belong to. The set $\cC'$, however, intersecting $\Lambda(x_4),\Lambda(x_5)$ and $\Lambda(x_1)$, requires projections on the span of both $\delta'$ and $\delta''$. These vectors are not arbitrary; these are extreme directions that move from one cone to another, inducing a face where both cones intersect. The illustration highlights that such vectors are necessary to capture by our data when the face they induce intersects the uncertainty set $\cC$. 
Hence, it is natural to introduce the following set of \textit{relevant extreme directions}.
\begin{definition}[Relevant Extreme Directions]
    Given $\cC\subset \R^d$, we define 
    {\setlength{\abovedisplayskip}{2pt}
 \setlength{\belowdisplayskip}{2pt}
 \setlength{\abovedisplayshortskip}{2pt}
 \setlength{\belowdisplayshortskip}{2pt}
   \begin{align*}
        \Delta(\cX,\cC)=\{\delta \in \R^d,\; \exists x^\star\in \cX^\angle,\; \delta \in D(x^\star) \text{ and }F(x,\delta)\cap \cC\neq \varnothing\}.
    \end{align*}
    }
\end{definition}
In the rightmost illustration of \cref{fig: optimality cones}, we have $\Vect \Delta(\cX,\cC)=\Vect \{\delta\}$ and $\Vect \Delta(\cX,\cC')=\Vect \{\delta',\delta''\}$, and it is necessary to observe the projections on $\Delta(\cX,\cC)$ and $\Delta(\cX,\cC')$ to recover optimal solutions for uncertainty sets $\cC$ and $\cC'$ respectively. This leads to our first main theorem. In order to state it, we first introduce two common topological notions.
\begin{definition}[Linear and Affine Hulls]
    For any nonempty set $\mathcal C\subset \R^d$ and any $c_0\in \mathcal C$, the linear hull of $\cC$ is $\dir{\cC}=\Vect(\cC - c_0) = \Vect \{x_1-x_2 \; : \; x_1,x_2 \in \cC\}$ and its affine hull is $\aff{\cC}=c_0 + \Vect (\cC - c_0)$.
\end{definition}

\begin{definition}[Relative Interior]
     For any nonempty set $\cC\subset \R^d$, we define
     {\setlength{\abovedisplayskip}{2pt}
 \setlength{\belowdisplayskip}{2pt}
 \setlength{\abovedisplayshortskip}{2pt}
 \setlength{\belowdisplayshortskip}{2pt}
     \begin{align*}
         \relint{\cC}=\{ x \in \cC, \text{ there exists } \varepsilon > 0 \text{ such that } B(x,\varepsilon) \cap \aff{\cC} \subset \cC \},
     \end{align*}}
where $B(x,\varepsilon):=\{y\in \R^d,\; \norm{x-y}<\varepsilon\}$. When $\cC = \relint{\cC}$, we say that $\cC$ is relatively open.
\end{definition}
We now formulate our first main theorem characterizing sufficient data sets for all relatively open convex uncertainty sets.
\begin{theorem} \label{thm: relatively open characterization}
             Let $\cC$ be a relatively open convex set. $\dataset$ is a \sufficient \;for $\cC$ if and only if $\Delta(\cX,\cC) \subset \dir{\cC}^\perp + \Vect \mathcal D$.
        \end{theorem}

        \cref{thm: relatively open characterization} is a fundamental characterization of sufficiency, by what information the data set needs to capture relative to the previously known $\cC$ and the problem structure $\cX$. The result also indicates that such a minimal data set is in general not unique. A careful reader might remark that \cref{thm: relatively open characterization} should imply \cref{prop:suff:vectorspace} when $\cC = \Re^d$. In fact, $\Delta(\cX,\R^d)$ is the set of all extreme directions of the polyhedron, which indeed precisely spans $\Ker A \cap F_0$. Furthermore, a similar intuition holds for this Theorem compared to Proposition \ref{prop:suff:vectorspace}:
{\setlength{\abovedisplayskip}{3pt}
 \setlength{\belowdisplayskip}{1pt}
 \setlength{\abovedisplayshortskip}{3pt}
 \setlength{\belowdisplayshortskip}{1pt}
\begin{align*}       
\underbrace{\Delta(\cX,\cC)}_{\text{what we need to know}} \subset \underbrace{\spc{\cC}^\perp}_{\text{what we already know}} +\underbrace{\Vect \cD}_{\text{what we measure}}.
\end{align*}

        \begin{remark}
        In Theorem \ref{thm: relatively open characterization}, only convexity is needed for the condition $ \Delta(\cX,\cC) \subset \dir{\cC}^\perp + \Vect \mathcal D$ to be sufficient, and only relative openness is needed for the condition to be necessary. In fact, a more general formulation of \cref{thm: relatively open characterization} would be
        \begin{enumerate}
            \item If $\cC$ is relatively open then: $\cD$ is \suff\ $\implies$ $\Delta(\cX,\cC) \subset \dir{\cC}^\perp + \Vect \mathcal D$.
            \item If $\cC$ is convex then: $\Delta(\cX,\cC) \subset \dir{\cC}^\perp + \Vect \mathcal D$ $\implies$ $\cD$ is \suff.
        \end{enumerate}
        \end{remark}
}
    \subsection{Algorithmically Tractable Characterization via Reachable Solutions}

    We now develop a second characterization of data set sufficiency that is particularly well-suited to algorithmic construction.

   The set $\Delta(\cX,\cC)$ of relevant extreme directions of \cref{thm: relatively open characterization} can be seen intuitively as the set of differences $x_1-x_2$ of \textit{neighboring} extreme points $x_1,x_2\in \cX^\angle$, that are optimal for some $c \in \cC$. By relaxing the ``neighboring’’ condition and optimality for a common $c\in \cC$, we arrive at a broader set of directions induced by all pairs of optimal extreme points---which we call \textit{reachable solutions}. This alternative perspective replaces the notion of relevant extreme directions with that of reachable solutions, leading to a representation of sufficiency in terms of differences between optimal solutions under cost vectors in \(\mathcal{C}\).

\begin{definition}[Reachable Solutions]
    Given $\cC\subset \R^d$, we define 
    {\setlength{\abovedisplayskip}{2pt}
 \setlength{\belowdisplayskip}{2pt}
 \setlength{\abovedisplayshortskip}{2pt}
 \setlength{\belowdisplayshortskip}{2pt}
    \begin{align*}
        \dualpoly{\cX}{\cC}:=\left\{x^\star\in \cX^\angle,\;\exists c \in \cC,\; x^\star\in \arg\min_{x\in \cX}c^\top x \right\}=\bigcup_{c\in \cC}^{}\arg\min_{x\in \cX}c^\top x.
    \end{align*} 
    }
\end{definition}

The set $\dir{\dualpoly{\cX}{\cC}}$ is equal to the span of the set of differences between \textit{any} two elements $x_1,x_2\in \cX$ such that there exists $c_1,c_2 \in \cC$ such that $x_1\in \arg\min_{x\in \cX}c_1^\top x$ and $x_2\in \arg\min_{x\in \cX}c_2^\top x$. By construction, we have $\Vect \Delta(\cX,\cC) \subset \dir{\dualpoly{\cX}{\cC}}  = \Vect \{x_1-x_2 \; : \; x_1,x_2 \in \dualpoly{\cX}{\cC} \}$ since each relevant extreme direction corresponds to a direction between optimal solutions. The following theorem shows that these quantities are indeed equal.

\begin{theorem}\label{thm:span delta is dir x}
    For any convex set $\cC\subset \R^d$, we have $\Vect \Delta(\cX,\cC)=\dir{\dualpoly{\cX}{\cC}}$.
\end{theorem}

The converse inclusion proven in this theorem is not immediate. In fact, for a general polyhedron $\cX$ and $\cC$ (see \cref{fig: optimality cones} with $\cC'$ for eg.), $\Delta(\cX,\cC)$ is much smaller than the set of differences of elements in $\dualpoly{\cX}{\cC}$ but their spans are equal. The proof (Appendix \ref{proof thm:span delta is dir x}) exploits a \emph{C-strong connectivity} property (\cref{def:c-strong connectivity}): for any $x,x'\in \dualpoly{\cX}{\cC}\cap \cX^\angle$, there exists a sequence of extreme points $x_1,\dots,x_h \in \dualpoly{\cX}{\cC}$ such that $x_1=x$ and $x_h=x'$ and for any $i\in [h-1]$, $x_{i+1}-x_i\in \Delta(\cX,\cC)$. In other words, $x_{i+1},x_i$ are neighbors and are both optimal for some $c_i\in \cC$. This implies that every element in \(\dir{\dualpoly{\mathcal{X}}{\mathcal{C}}}\) can be written as a finite linear combination of elements in \(\Delta(\mathcal{X}, \mathcal{C})\), completing the equality. Relating again to \cref{prop:suff:vectorspace} of the case $\cC = \Re^d$, careful linear algebra shows that indeed $\dir{\cX^\star(\Re^d)} = \dir{\cX} = \Ker A \cap F_0$.
\cref{thm: relatively open characterization} implies that to construct a \sufficient\, it suffices to find
a basis of $\dir{\dualpoly{\cX}{\cC}}$ rather than $\Vect \Delta(\cX,\cC)$, which is a much simpler task. The following result will indeed be the basis of our algorithm in the next section.

\begin{theorem}\label{cor:dualpoly-carac}
    Let $\cC$ be a relatively open convex set. $\dataset$ is a \sufficient \;for $\cC$ if and only if $\dir{\dualpoly{\cX}{\cC}} \subset \dir{\cC}^\perp + \Vect \mathcal D$.
\end{theorem}

The characterization of \cref{cor:dualpoly-carac} has further important implications on the generality of our result. In fact, our characterization generalizes seamlessly to Mixed Integer Programs (MIPs).
 
\begin{corollary} \label{cor:MIP-sufficiency}
For any set $\cC\subset \R^d$, a set $\cD\subset \R^d$ is a \sufficient{} for uncertainty set $\cC\subset \R^d$ and decision set $\cX\cap \{0,1\}^d$ if and only if $\dir{\dualpoly{(\cX \cap\{0,1\}^d)}{\cC}}\subset \dir{\cC}^\perp + \Vect \cD$.   
\end{corollary}

We must point out, however, that computing $\dir{\dualpoly{(\cX \cap\{0,1\}^d)}{\cC}}$ is challenging compared to computing $\dir{\dualpoly{\cX}{\cC}}$, as the tools used in \ref{sec:algorithm:dir} do not apply to data collection for solving MIPs.

\section{Data Collection}\label{sec:algorithm}

\subsection{A Data Collection Algorithm: Finding Minimal Sufficient data set}
\label{sec:algorithm:dir}
We now turn to the practical problem of selecting a minimal---i.e., smallest or least costly---data set $\cD$ that enables generalization from prior contextual knowledge (captured by \(c \in \mathcal{C}\)) to a specific decision-making task (defined by \(\mathcal{X}\)). 

In many practical settings, data collection is subject to constraints on what can be queried. We model this by restricting the data set to lie in a predefined query set \(\mathcal{Q} \subset \mathbb{R}^d\), so that \(\mathcal{D} \subseteq \mathcal{Q}\). For example, in the subway line problem mentioned in \cref{sec:intro}, $\cQ$ can be the set of canonical basis vectors---a data point is revealing the cost of building in a specific road in the city. \cref{cor:dualpoly-carac} implies that to find a minimal \suff\ data set $\cD$, we need to address two distinct challenges: compute $\dir{\dualpoly{\cX}{\cC}}$, then construct a minimal $\cD \subset \cQ$ verifying the spanning condition $\dir{\dualpoly{\cX}{\cC}} \subset \dir{\cC}^\perp + \Vect \mathcal D$. Intuitively, $\dir{\dualpoly{\cX}{\cC}}$ captures precisely the information needed to solve the problem $\min_{x\in \cX}c^\top x$ given the prior information $c \in \cC$. The second step, which can be done separately and uses the output of step 1 as input, accounts for the data collection constraints $\cQ$ and will be addressed in detail in \cref{sec:query constraints}.

In the remainder of this section, we will focus on the first step, i.e. to compute $\dir{\dualpoly{\cX}{\cC}}$ and construct a basis for it. 
We can write 
\(
    \dir{\dualpoly{\cX}{\cC}}=\Vect \{x_0 - x,\; x\in \dualpoly{\cX}{\cC}\}
\)
for some $x_0 \in \dualpoly{\cX}{\cC}$. Hence, to compute $\dir{\dualpoly{\cX}{\cC}}$, we can iteratively add elements of it while ensuring we increase the dimension at every step. This is formalized in Algorithm \ref{alg:meta_alg_dir}.

\begin{algorithm}[h]
\caption{Conceptual Algorithm Computing $\dir{\dualpoly{\cX}{\cC}}$}
\label{alg:meta_alg_dir}

    \KwIn{Decision set $\cX$, Uncertainty set $\cC$.}
    \KwOut{A basis $\cD\subset \R^d$ of $\dir{\dualpoly{\cX}{\cC}}$.}
    Initialize $\cD$ to $\varnothing$.
    
    Set $x_0\in \arg\min_{x\in \cX}c_0^\top x$ for some $c_0 \in \cC$.
    
    \textbf{while} there exists $c\in \cC$, $x^\star \in \argmin_{x\in\cX} c^\top x$ such that $x^\star-x_0 \not \in \Vect{\cD}$.

    \quad $\cD \leftarrow \cD \cup \{x^\star-x_0\}$.
    
    \textbf{return} $\cD$
\end{algorithm}

The main step in Algorithm \ref{alg:meta_alg_dir} (condition of the while loop) can be seen as verifying whether the optimization problem
{\setlength{\abovedisplayskip}{2pt}
 \setlength{\belowdisplayskip}{2pt}
 \setlength{\abovedisplayshortskip}{2pt}
 \setlength{\belowdisplayshortskip}{2pt}
\begin{align}
    \sup \{\ \|\proj{(\Vect \cD)^\perp}{x^\star - x_0}\|  \; : \; c \in \cC, \; x^\star \in \argmin_{x \in \cX} c^\top x\},\label{eq:alg_optimization_pb}
\end{align} }
where $\proj{(\Vect \cD)^\perp}{\cdot}$ is the projection map onto $(\Vect \cD)^\perp$, has a solution with a non-zero objective. This optimization problem has two main challenges: first, it entails the inherently difficult task of maximizing a convex objective, and second, it has a bilinear, bi-level constraint $x^\star \in \argmin_{x \in \cX} c^\top x$ as both $c$ and $x^\star$ are variables and $x^\star$ must be an optimal solution to a linear program parameterized by c.

\textbf{Linearizing the objective.}
Remark that if $\alpha$ is a randomly generated Gaussian vector, then any vector $v$, with $\|v\|>0$, satisfies $\Prob(\alpha^\top v = 0) = 0$.
Hence, if Problem \eqref{eq:alg_optimization_pb} admits a solution $\bar{x}$ verifying $\|\proj{(\Vect \cD)^\perp}{\bar{x} - x_0}\|> 0$, then $\alpha^\top \proj{(\Vect \cD)^\perp}{\bar{x} - x_0} \neq 0$ with probability $1$, and therefore
either maximizing or minimizing $\alpha^\top \proj{(\Vect \cD)^\perp}{x^\star - x_0}$ must lead a non-zero objective with probability $1$. This is a linear objective as the projection onto a subspace is linear.


\textbf{Linearizing the bilinear, bilevel constraint.}
To address the second challenge, we use complementary slackness conditions, which characterize the optimal solutions of linear programs. We replace $x^\star \in \argmin_{x \in \cX} c^\top x, \; c \in \cC$ by 
{\setlength{\abovedisplayskip}{1pt}
 \setlength{\belowdisplayskip}{1pt}
 \setlength{\abovedisplayshortskip}{1pt}
 \setlength{\belowdisplayshortskip}{1pt}
\begin{align*}
    &x^\star \geq 0, \; s \geq 0, \; \lambda \in \Re^m, \; c \in \cC, \\
    &Ax^\star = b, \; A^\top \lambda + s = c, \; x^\star_is_i = 0, \; \forall i \in [d] 
\end{align*}
}
The bilinear constraint $x^\star_is_i = 0$ can be linearized by introducing a binary variable $\tau_i \in \{0,1\}$ and adding the constraint $1- \epsilon s_i \geq \tau_i \geq \epsilon x^\star_i$ with $\epsilon>0$ a small constant. When \(\mathcal{C}\) is a polyhedron, the resulting formulation is a mixed-integer linear program (MILP) with linear constraints and objectives.

Putting everything together gives Algorithm \ref{alg:LP_case} for linear programs. The algorithm will terminate in exactly $\dim \dir{\dualpoly{\cX}{\cC}}$  (\cref{thm:alg_termination}). When $\cC$ is a polyhedron, each iteration involves solving a mixed integer program with $O(d+m)$ variables and $O(d+m + \mathrm{constr}(\cC))$ constraints where $\mathrm{constr}(\cC)$ is the number of constraints defining $\cC$.   Even though solving MILPs is NP-hard in general, modern MILP solvers can handle the required computations effectively for problems of moderate size such as the examples in \cref{sec:applications}.

\begin{algorithm}[h]
    \caption{Computing $\dir{\dualpoly{\cX}{\cC}}$}
     \label{algorithm to compute delta}
    \KwIn{Polyhedron $\cX = \{x \geq 0 \; : \; Ax=b\}$, Uncertainty set $\cC$.}
    \KwOut{A basis of $\dir{\dualpoly{\cX}{\cC}}$.}
    Initialize $\cD$ to $\varnothing$.
    
    Set $x_0\in \arg\min_{x\in \cX}c_0^\top x$ for some $c_0 \in \cC$.

    Sample $\alpha \sim \mathcal{N}(0,Id)$.
    
    \textbf{while} either of the problems 
{\setlength{\abovedisplayskip}{1pt}
 \setlength{\belowdisplayskip}{1pt}
 \setlength{\abovedisplayshortskip}{1pt}
 \setlength{\belowdisplayshortskip}{1pt}
    \begin{align*}
        \min/ \max &\;\alpha^\top \proj{(\Vect \cD)^\perp}{x_0-x} \\
        \text{s.t.}&\;x \geq 0,\; \lambda \in \R^m,\; s\in \R_+^d,\; c\in \cC \\
        &Ax=b, \; A^\top \lambda +s=c,\\ &1- \epsilon s_i \geq \tau_i \geq \epsilon x_i, \; \tau_i \in \{0,1\}, \; \forall i
    \end{align*}
}
    \quad has a solution $x^\star$ with non-zero optimal value,
    
    \quad $\cD \leftarrow \cD \cup \{x^\star - x_0\}$.
    
    \quad resample $\alpha \sim \mathcal{N}(0,Id)$. \Comment{\textit{To maintain randomization independence across iterations}}
    
    \textbf{return} $\cD$
\label{alg:LP_case}
\end{algorithm}

\begin{theorem}[Correctness]\label{thm:alg_termination}
   Algorithm \ref{algorithm to compute delta} terminates with probability $1$ after $\dim \dir{\dualpoly{\cX}{\cC}} \leq d$ steps and outputs a basis of $\dir{\dualpoly{\cX}{\cC}}$.
\end{theorem}

\subsection{Finding Minimal data sets under Data Collection Constraints}\label{sec:query constraints}
\cref{sec:algorithm:dir} provides an algorithm that computes a basis for $\spc{\dualpoly{\cX}{\cC}}$. The next step is to find a smallest data set $\cD \subset \cQ$ that is sufficient, i.e. $\spc{\dualpoly{\cX}{\cC}}\subset \spc{\cC}^\perp + \Vect \cD$ (\cref{cor:dualpoly-carac}) given some data collection constraint set $\cQ \ \subset \R^d$. Real applications impose diverse constraints on what can be queried. For example, if our task is to build a subway line at minimal cost (a shortest path problem), a natural setting is where we can only query building costs along individual edges of the graph. This corresponds to the case where $\cQ$ is the canonical basis of $\R^d$. Another example is where we can query the cost of feasible decisions. This is $\cQ = \cX$. If our problem is a combinatorial problem that can be relaxed into an LP, then feasible solutions are extreme points of the polyhedron and $\cQ = \cX^\angle$. Another relevant case is when we can query the costs of feasible solutions in a another polyhedron $\mathcal Y$ at data collection time, with the goal of informing the optimal solution of a different polyhedron $\cX$ at decision time. This would be the case of $\cQ = \mathcal Y$, a convex set.

In what follows, we study the problem of finding the smallest sufficient data set $\cD\subset \cQ$ (i.e. that satisfies the inclusion in \cref{cor:dualpoly-carac}), which we formally define below.

\begin{problem}[Space coverage problem] \label{problem: space coverage problem}
    $\dualpoly{\cX}{\cC}\subset \spc{\cC}^\perp + \Vect \cQ$. Given a basis of $\text{dir}\left(\cX^\star(\cC)\right)$, find a smallest set $\cD\subset \cQ$ such that $\spc{\dualpoly{\cX}{\cC}}\subset \spc{\cC}^\perp+\Vect \cD$, i.e. a smallest sufficient data set.
    \end{problem}
 Note that the inclusion $\dualpoly{\cX}{\cC}\subset \spc{\cC}^\perp + \Vect \cQ$ is necessary for the existence of a sufficient data set.
 This subsection will be dedicated to studying this problem and introducing algorithms to solve it. We first introduce two key lemmas that are essential to our understanding of the problem.

\begin{lemma} \label{lemma:problem equivalence}
    The first inclusions in \cref{problem: space coverage problem} is to equivalent to $\proj{\spc{\cC}}{\spc{\dualpoly{\cX}{\cC}}} \subset \proj{\spc{\cC}}{\Vect \cQ}$. Similarly, the second inclusion is equivalent to $\proj{\spc{\cC}}{\spc{\dualpoly{\cX}{\cC}}} \subset \proj{\spc{\cC}}{\Vect \cD}$.
\end{lemma}

The lemma above provides intuitive insight: for a data set $\cD$ to be sufficient, its projection on the space of directions along which the cost vector is uncertain should cover the set of decision-relevant directions projected on that same space, i.e. ignoring directions in $\spc{\cC}^\perp$ along which we already have full information on the cost vector. 
Notice that we can immediately deduce from the lemma above that the size of a sufficient data set $\cD$ is at least the dimension of $\proj{\spc{\cC}}{\spc{\dualpoly{\cX}{\cC}}}$. We formally define this quantity below and formalize this property in \cref{sufficient data set lowerbound}.


\begin{definition}[Missing Information]\label{def: Missing Information}
    We define the quantity of missing information as 
    {\setlength{\abovedisplayskip}{1pt}
 \setlength{\belowdisplayskip}{1pt}
 \setlength{\abovedisplayshortskip}{1pt}
 \setlength{\belowdisplayshortskip}{1pt}
    \begin{align*}
        r(\cX,\cC):=\dim \proj{\spc{\cC}}{\spc{\dualpoly{\cX}{\cC}}}=\dim \spc{\dualpoly{\cX}{\cC}} - \dim (\spc{\dualpoly{\cX}{\cC}} \cap \spc{\cC}^\perp). 
    \end{align*}}
\end{definition}

The quantity defined above is equal to the amount of information needed to find the optimal decision ($\dim \dualpoly{\cX}{\cC}$) minus the quantity of overlap between the information we need to know and the information we already know ($\dim (\spc{\dualpoly{\cX}{\cC}} \cap \spc{\cC}^\perp)$):
{\setlength{\abovedisplayskip}{3pt}
 \setlength{\belowdisplayskip}{3pt}
 \setlength{\abovedisplayshortskip}{3pt}
 \setlength{\belowdisplayshortskip}{3pt}
\begin{align*}
    \underbrace{r(\cX,\cC)}_{\text{quantity of missing information}}=\underbrace{\dim \spc{\dualpoly{\cX}{\cC}}}_{\text{quantity of necessary information}} - \underbrace{\dim (\spc{\dualpoly{\cX}{\cC}} \cap \spc{\cC}^\perp)}_{\text{quantity of relevant information already available}}.
\end{align*}
}

This natural quantity turns out to be a lower bound on the size of a sufficient data set. It is a tight lower bound, as we will see later.

\begin{lemma} \label{sufficient data set lowerbound}
   The size of a sufficient data set is at least the amount of missing information $r(\cX,\cC)$. 
\end{lemma}

\noindent
\cref{problem: space coverage problem} is NP-hard in general, as the following proposition shows.


\begin{proposition} \label{hardness result}
    \cref{problem: space coverage problem} is NP-hard.
\end{proposition}

To understand why the above result holds, we analyze a simple instance. Let $v\in\R^d\setminus\{0\}$, $\cX=\conv\{0,v\}$, and $\cC=\R^d$. Then $\spc{\cC}^\perp={0}$ and $\spc{\dualpoly{\cX}{\cC}}=\spc{\cX}\Vect v$, so \cref{problem: space coverage problem} reduces to finding the smallest $\cD \subset \cQ$ such that $\Vect{v}\subset \cD$. When $\cQ$ is finite, this is equivalent to finding the sparsest solution $x$ to $Mx=v$, where the columns of $M$ are the elements of $\cQ$. This problem is precisely the $\ell_0$-minimization problem, which is NP-hard (\cite{natarajan1995sparse}). 

This simple instance further intuitively indicates how the structure of $\cQ$ shapes the hardness of \cref{problem: space coverage problem}.
When $\cQ$ is a basis of $\R^d$, then $M$ is invertible, and the problem becomes computationally tractable. 
However, as we will discuss in \cref{subsubsec: Q discrete}, this insight only holds when $\cC$ is open, and the NP-hardness remains when $\cC$ is relatively open.

When $\cQ$ is a vector space, the instance we considered becomes trivial, as by assumption we have $v\in \cQ$, so we can simply take $\cD=\{v\}$ as a minimal sufficient data set. More generally, we will prove in \cref{subsubsec: Q continuous} that \cref{problem: space coverage problem} can be solved in polynomial time by solving a linear equation when $\cQ$ is a vector space.

Further special structures of $\cQ$ will admit polynomial-time algorithms.
\cref{tab:hardness results} summarizes our complexity results and algorithmic contributions for \cref{problem: space coverage problem}. We believe that these results cover most of the relevant cases in practice. When $\cQ$ is a vector space, \cref{problem: space coverage problem} is tractable and the size of a smallest sufficient data set (assuming that it exists) depends only on the task structure (i.e. $\cX$ and $\cC$) and not on the specific choice of $\cQ$ (see \cref{subsubsec: Q continuous}). Interestingly, \cref{tab:hardness results} shows that this extends to sets that locally behave like a vector space: the problem is tractable when $\cQ$ has continuous structure (e.g. vector spaces, relatively open sets, or convex sets), but NP-hard for general discrete query sets—even when restricted to basis vectors—except for very special structured families such as the set of extreme points of the feasible set.


\begin{table}[]
\centering
\begin{center}
\begin{tabular}{|>{\centering\arraybackslash}m{2.7cm}|
                >{\centering\arraybackslash}m{2.7cm}|
                >{\centering\arraybackslash}m{5.5cm}|
                >{\centering\arraybackslash}m{4.3cm}|}
\hline
$\cQ$ & $\cC$ & \textbf{Complexity} & \textbf{Minimal data set size} \\
\hline
\hline
Finite & Relatively open & NP-hard (\cref{hardness result}) & $\geq r(\cX,\cC)$ \\
\hline
Basis of $\mathbb{R}^d$ & Relatively open & NP-hard (Alg \ref{prop: hardness Q basis}) & $\geq r(\cX,\cC)$ \\
\hline
Basis of $\mathbb{R}^d$ & Open & Polynomial time (Alg \ref{alg:data collection algorithm}) & $\geq r(\cX,\cC)$ \\
\hline
Vector space & Relatively open & Polynomial time (Alg \ref{algorithm space coverage problem vector space}) & $ r(\cX,\cC)$ \\
\hline
Relatively open & Relatively open & Polynomial time (Alg \ref{alg:relatively open}) & $\in \{r(\cX,\cC),r(\cX,\cC)+1\}$ \\
\hline
Convex & Relatively open & Polynomial time$^\star$ (Alg \ref{alg:relatively open}) & $\in \{r(\cX,\cC),r(\cX,\cC)+1\}$ \\
\hline
$\cX^\angle$ & Relatively open & Polynomial time$^\star$ (Alg \ref{alg: sufficient datatset for extreme points}) & $\in \{r(\cX,\cC),r(\cX,\cC)+1\}$ \\
\hline
\end{tabular}
\end{center}

\caption{Hardness of \cref{problem: space coverage problem} for different values of $\cQ$ and $\cC$. Here, $\cC$ is assumed to be convex. When $\cQ$ is relatively open, \cref{alg:relatively open} outputs a minimal sufficient data set. $^\star$For both cases when $\cQ$ is convex or the set of extreme points of $\cX$, the algorithms provided are finding sufficient data sets of cardinality differing by at most one from a smallest sufficient data set. The quantity $r(\cX,\cC)$ is defined in \cref{def: Missing Information}.}
\label{tab:hardness results}
\end{table}


We now study each case of $\cQ$ structure and introduce tractable algorithms and hardness results.

\subsubsection{$\cQ$ discrete} \label{subsubsec: Q discrete} 

While the discrete nature makes the problem inherently hard, we study two important special cases for which we provide tractable algorithms, namely $\cQ$ basis and $\cQ=\cX^\angle$.

\textit{$\cQ$ basis of $\R^d$.} We start with the nominal case of $\cQ$ being the canonical basis vectors $\{e_i \; : \; i \in [d]\}$ where $(e_i)_j = \mathbf{1}(i=j)$ for all $i,j \in [d]$. That is the highly practical case where we can query coordinates of the unknown cost vector (eg, individual edges of a graph, see also \cref{sec:applications} for more applications). Despite \cref{problem: space coverage problem} being NP-hard, it will admit a simple explicit solution for this special structure of $\cQ$ when $\cC$ is open. 
In fact, given $v_1,\ldots,v_k$ a basis of $\spc{\dualpoly{\cX}{\cC}}$, the smallest sufficient data set verifying the spanning condition of \cref{cor:dualpoly-carac}, is exactly $\cD = \{ e_i \; : \; i \in [d], \; \exists j \in [k], \; v_j^\top e_i \neq 0 \}$. This is all the non-zero coordinates of basis vectors of $\dir{\dualpoly{\cX}{\cC}}$. This case can be generalized in a straightforward manner when $\cQ$ is any basis of $\Re^d$; see \cref{alg:data collection algorithm}.

Despite the fact that \cref{problem: space coverage problem} can be solved in polynomial time when $\cC$ is open, and $\cQ$ is a basis of $\R^d$, it becomes NP-hard when $\cC$ is relatively open rather than open.
\begin{proposition} \label{prop: hardness Q basis}
    When $\cQ$ is a basis of $\R^d$ and $\cC$ is relatively open, \cref{problem: space coverage problem} is NP-hard.
\end{proposition}

Let us provide a proof sketch of this result. When $\cC$ is relatively open, \cref{lemma:problem equivalence} allows us to reduce the problem to a modified version of the one discussed after \cref{hardness result}. The resulting problem is to minimize $\norm{x}_0$ subject to $Mx = v$, except that the columns of $M$ are given by the projections of the elements of $\cQ$ onto $\spc{\cC}$. Because $\cQ$ is a basis of $\R^d$, the matrix $M$ has rank $\dim \spc{\cC}$. If $\cC$ is open, that is, if $\spc{\cC} = \R^d$, then $M$ is invertible, and the problem is tractable.
If $\cC$ is not open but only relatively open, then $\dim \spc{\cC} \le d - 1$. In this case, $M$ is not invertible, and the problem becomes NP-hard.


 \textit{$\cQ$ equal to the set of extreme points of $\cX$.} \cref{algorithm to compute delta} returns a basis of $\spc{\dualpoly{\cX}{\cC}}$ of the form $\cD_0 := \{x^\star_1 - x_0,\dots,x^\star_r - x_0\}$ where $x_0$ and $\{x_i^\star\}_{i \leq r}$ are 
 elements of $\dualpoly{\cX}{\cC}$. We can modify this procedure (see \cref{alg: modified algorithm to compute delta}, Appendix \ref{appendix: modified alg to compute delta}) to guarantee that $\{x_0, x^\star_1, \dots, x^\star_r\}$ are all extreme points and instead return $\cD=\{x_0, x^\star_1, \dots, x^\star_r\}\subset \cQ$. Notice that when $\cC$ is open, $\cD_0$ is a sufficient data set (since $\spc{\dualpoly{\cX}{\cC}}\subset \Vect \cD_0$) with minimal cardinality $r(\cX,\cC)$ (because $\cD_0$ is a basis of $\spc{\dualpoly{\cX}{\cC}}$, i.e. $\abs{\cD_0}=\dim \spc{\dualpoly{\cX}{\cC}}=r(\cX,\cC)$), but does not satisfy $\cD_0 \subset \cQ$. Since $\Vect \cD_0 \subset \Vect \cD $, $\cD$ is also a sufficient data set of cardinality $r(\cX,\cC)+1$ and satisfies $\cD\subset \cQ$. In the more general case when $\cC$ is relatively open instead of open, \cref{alg: sufficient datatset for extreme points} generalizes this procedure by applying gaussian elimination to the projection of $\cD_0$ on $\spc{\cC}$ (see Appendix \ref{appendix: justification of algorithm for Q extreme points} for correctness proof).

 \begin{algorithm}[h]
    \caption{Computing a nearly smallest sufficient dataset when $\cQ=\cX^\angle$ and $\cC$ is relatively open.}
     
    \KwIn{$\mathcal U:=\{x_0,x^\star_1,\dots,x^\star_r\}$, where $\mathcal U$ is the output of  the modified version of Alg \ref{algorithm to compute delta} (Alg \ref{alg: modified algorithm to compute delta}).}
    \KwOut{A \sufficient\ $\cD\subset \cX^\angle$ such that $\abs{\cD}=r(\cX,\cC)+1$.}
    Run gaussian elimination on $\{\proj{\spc{\cC}}{x_0-x^\star_1},\dots,\proj{\spc{\cC}}{x_0-x^\star_r}\}$ to obtain a linearly independent set of vectors $\{\proj{\spc{\cC}}{x_0-x^\star_{i_1}},\dots,\proj{\spc{\cC}}{x_0-x^\star_{i_p}}\}$.

    \textbf{return} $\{x_0,x^\star_{i_1},\dots,x^\star_{i_{p}}\}$
\label{alg: sufficient datatset for extreme points}
\end{algorithm}

\subsubsection{$\cQ$ continuous} \label{subsubsec: Q continuous} When $\cQ$ exhibits a continuous structure, the problem becomes tractable in most cases as we will see below.

\textit{$\cQ$ vector space. } From \cref{lemma:problem equivalence}, \cref{problem: space coverage problem} is equivalent to finding a smallest set $\cD$ such that $\proj{\spc{\cC}}{\spc{\dualpoly{\cX}{\cC}}} \subset \proj{\spc{\cC}}{\Vect{\cD}}$ under the assumption that $\proj{\spc{\cC}}{\spc{\dualpoly{\cX}{\cC}}} \subset \proj{\spc{\cC}}{\Vect \cQ}$. When $\cQ$ is a vector space, this assumption can be rewritten as $\proj{\spc{\cC}}{\spc{\dualpoly{\cX}{\cC}}} \subset \proj{\spc{\cC}}{\cQ}$. Hence, in order to find a minimal sufficient data set, it suffices to find a basis of $\proj{\spc{\cC}}{\spc{\dualpoly{\cX}{\cC}}}$ of the form $\proj{\spc{\cC}}{q_1},\dots,\proj{\spc{\cC}}{q_r}$ such that $q_1,\dots,q_r\in \cQ$ (exists because of the inclusion we have just mentioned). Consequently, we can immediately see that $\cD=\{q_1,\dots,q_r\}\subset \cQ$ is a sufficient data set, and it is of cardinality $r(\cX,\cC)$. The following algorithm formalises this idea and provides a polynomial time procedure to solve Problem \ref{problem: space coverage problem} when $\cQ$ is a vector space.

\begin{algorithm}[h]
    \caption{Computing a smallest \sufficient\ when $\cQ$ is a vector space}
     \label{algorithm space coverage problem vector space}
    \KwIn{Matrices $M_{\spc{\dualpoly{\cX}{\cC}}}\in \R^{d\times \dim\spc{\dualpoly{\cX}{\cC}}},\;M_{\spc{\cC}}\in \R^{d\times \dim \spc{\cC}},M_{\mathcal Q}\in \R^{d\times \dim \cQ}$ whose columns represent respectively basis for the spaces $\spc{\dualpoly{\cX}{\cC}},\spc{\cC},\mathcal Q$.}
    \KwOut{A minimal sufficient data set $\cD\subset \cQ$.}
    \begin{itemize}
        \item Compute projection matrix $P_{\spc{\cC}}:=M_{\spc{\cC}}\left(M_{\spc{\cC}}^\top M_{\spc{\cC}}\right)^{-1}M_{\spc{\cC}}^\top\in \R^{d\times d}$.
        \item Obtain a basis of $\proj{\spc{\cC}}{\spc{\dualpoly{\cX}{\cC}}}$ by doing gaussian elimination on the columns of $P_{\spc{\cC}} M_{\spc{\dualpoly{\cX}{\cC}}}$. Consider $\Tilde{M}_{\spc{\dualpoly{\cX}{\cC}}} \in \R^{d\times r(\cX,\cC)}$ the matrix whose columns are the obtained basis.
        \item Solve the linear system $\Tilde{M}_{\spc{\dualpoly{\cX}{\cC}}}=P_{\spc{\cC}}M_{\cQ}S$ for $S\in \R^{\dim \cQ \times r(\cX,\cC)}$.
    \end{itemize}

    \textbf{return} set of columns of $M_{\cQ}S$.
\end{algorithm}

It is easy to see that Algorithm \ref{algorithm space coverage problem vector space} involves only polynomial-time operations. The following proposition shows that it indeed solves \cref{problem: space coverage problem}. Furthermore, it proves that the lower bound $r(\cX,\cC)$ (\cref{def: Missing Information}) is tight: it is exactly the smallest cardinality of a sufficient decision data set in this case. We can see here that as long as $\cQ$ is a vector space and a sufficient data set exists, the size of a smallest sufficient data set remains the same, because by assumption $\proj{\spc{\cC}}{\cQ}$ will always contain a basis of $\proj{\spc{\cC}}{\spc{\dualpoly{\cX}{\cC}}}$. 

\begin{proposition} \label{prop: covspace vector space corectness}
 When $\cC$ is a relatively open set, $\cQ$ is a vector space and $\spc{\dualpoly{\cX}{\cC}}\subset \spc{\cC}^\perp + \cQ$, Algorithm \ref{algorithm space coverage problem vector space} outputs a set $\cD\subset \cQ$ with minimal cardinality such that $\spc{\dualpoly{\cX}{\cC}}\subset \spc{\cC}^\perp +\Vect \cD$. Furthermore, its cardinality is $r(\cX,\cC)$.    
\end{proposition}

As we will see below, \cref{algorithm space coverage problem vector space} serves as a fundamental building block for designing algorithms for broader classes of sets $\cQ$.

\textit{$\cQ$ relatively open.} The following lemma will be key to designing an algorithm in this case.

\begin{lemma} \label{equivalent data set open}
    Assume that $\cQ$ is a relatively open set. If $\cD$ is a finite subset of $\Vect \cQ$ and contains one element of $\cQ$, then there exists a finite subset $\cD'$ of $\cQ$ such that $\abs{\cD}=\abs{\cD'}$ and $\Vect \cD = \Vect \cD'$.
\end{lemma}
A direct consequence of this lemma is that for any relatively open query set $\cQ$ we can always construct a \sufficient\ $\cD\subset \cQ$ of cardinality $r(\cX,\cC)+1$. This is because we can apply \cref{algorithm space coverage problem vector space} to obtain a sufficient data set $\cD\subset \Vect \cQ$ such that $\abs{\cD}=r(\cX,\cC)$, and then use the idea in \cref{equivalent data set open} to obtain from $\cD \cup \{q_0\}$ (where $q_0\in \cQ$) a sufficient data set $\cD'\subset \cQ$ such that $\abs{\cD'}=\abs{\cD \cup \{q_0\}}=r(\cX,\cC)+1$. Hence, the following corollary follows from \cref{sufficient data set lowerbound}.
\begin{corollary} \label{cardinality suff data set relatively open set}
    When $\cQ$ is a relatively open set, the smallest cardinality of a sufficient data set $\cD\subset \cQ$ is either $r(\cX,\cC)$ or $r(\cX,\cC)+1$.
\end{corollary}

The approach described above already provides a set whose cardinality differs by at most $1$ from the smallest cardinality possible. We provide in the following algorithm a way to find the smallest cardinality possible of a sufficient decision data set that is a subset of any relatively open set $\cQ$.

\begin{algorithm}[h] 
    \caption{Computing a smallest \sufficient\ when $\cQ$ is a relatively open set}
    \label{alg:relatively open}
    \KwIn{Matrices $M_{\spc{\dualpoly{\cX}{\cC}}}\in \R^{d\times \dim \spc{\dualpoly{\cX}{\cC}}},\;M_{\spc{\cC}}\in \R^{d\times \dim \spc{\cC}},M_{\mathcal Q}\in \R^{d\times \dim \cQ}$ whose columns represent respectively basis for the spaces $\spc{\dualpoly{\cX}{\cC}},\spc{\cC},\Vect{\mathcal Q}$.}
    \KwOut{A minimal sufficient data set $\cD\subset \cQ$.}
    
    \uIf{$\cQ \cap (\spc{\dualpoly{\cX}{\cC}}+\spc{\cC}^\perp)\setminus \spc{\cC}^\perp\neq \varnothing$}{Consider $q\in \cQ \cap (\spc{\dualpoly{\cX}{\cC}}+\spc{\cC}^\perp) \setminus \spc{\cC}^\perp$
    
    Apply Algorithm \ref{algorithm space coverage problem vector space} to find a minimal $\cD\subset \Vect \cQ$ such that $\spc{\dualpoly{\cX}{\cC}}\subset \spc{\cC}^\perp+\Vect \{q\}+\Vect \cD$.

    Apply the method in Lemma \ref{equivalent data set open} to obtain a set $\cD'$ such that $\Vect \cD'=\Vect (\cD\cup \{q\})$, $\abs{\cD'}=\abs{\cD \cup \{q\}}$ and $\cD'\subset \cQ$.
    
    \textbf{return} $\cD'$} \Else{
    Apply Algorithm \ref{algorithm space coverage problem vector space} to find a minimal $\cD\subset \Vect \cQ$ such that $\spc{\dualpoly{\cX}{\cC}}\subset \spc{\cC}^\perp+\Vect \cD$.

    $\cD \leftarrow \cD \cup \{q\}$ for an arbitrary element $q\in \cQ$
    
    Use the method in Lemma \ref{equivalent data set open} to find a set $\cD'$ such that $\abs{\cD'}=\abs{\cD}$, $\Vect \cD= \Vect \cD'$, $\cD'\subset \cQ$.

    \textbf{return} $\cD'$
    
    }
\end{algorithm}

\begin{proposition} \label{covspace vector space corectness relatively open set}
 When $\cC$ is a relatively open set, $\cQ$ is a relatively open set and $\spc{\dualpoly{\cX}{\cC}}\subset \spc{\cC}^\perp + \Vect \cQ$, Algorithm \ref{alg:relatively open} outputs a \sufficient\ $\cD\subset \cQ$ with minimal cardinality.
\end{proposition}

\textit{$\cQ$ convex.} \cref{alg:relatively open} can also be used for general convex sets. Indeed, since we have for any convex set $\cQ$, $\Vect\cQ = \Vect \relint \cQ $ (see Appendix \ref{appendix: proof of span Q equals span relint Q}), we can apply \cref{alg:relatively open} to $\relint \cQ$, which will provide a sufficient data set containing at most $r(\cX,\cC)+1$ elements. Since a sufficient data set cannot have less than $r(\cX,\cC)$ elements (see \cref{sufficient data set lowerbound}), then the sufficient data set obtained by \cref{alg:relatively open} will have at most one more element compared to the cardinality of a smallest \sufficient\ that is a subset of $\cQ$.

\section{Applications}\label{sec:applications}

\subsection{Minimal Cost Subway Line Design}\label{sec: subway example}

We now apply our framework and algorithm to the subway design problem introduced in \cref{sec:intro}. The objective is to construct a minimum-cost subway line connecting a specified source and destination. Construction costs on the network’s edges are uncertain, and the planner must decide on which edges to dispatch engineers to conduct field studies and obtain precise cost estimates.

Formally, this problem is a shortest path problem with uncertain edge costs. 
The city is represented by a graph \(G := (V, E)\), where \(V\) denotes the set of vertices and 
\(E \subseteq V^2\) the set of edges. The origin and destination of the subway line are given by 
\(s, t \in V\). Each edge \((i, j) \in E\) has an unknown construction cost \(c_{ij}\). Prior 
information about these costs is encoded through an uncertainty set 
\(\mathcal{C} \subset \mathbb{R}^{|E|}\), specifying that the true cost vector satisfies 
\(c \in \mathcal{C}\).

The decision maker’s goal is to determine the smallest set of edge costs that must be observed—%
i.e., the minimal set of coordinates \(c_{ij}\) to query—in order to identify the optimal 
\(s\)-\(t\) path. This is exactly the problem of finding a minimal sufficient data set when 
queries are restricted to canonical basis vectors,
\(
    q \in \mathcal{Q} := \{ e_{ij} : (i,j) \in E \},
\)
each of which reveals the cost of a single edge.

\paragraph{Shortest Path Problem as an LP.} 
The shortest path problem can be cast as the following min-flow problem:
{\setlength{\abovedisplayskip}{3pt}
 \setlength{\belowdisplayskip}{3pt}
 \setlength{\abovedisplayshortskip}{3pt}
 \setlength{\belowdisplayshortskip}{3pt}
\begin{align*}
    \min_{x\in \{0,1\}^E} &\sum_{(i,j) \in E} c_{ij} x_{ij} \quad \quad \text{s.t.} \quad \quad \forall v \in V,\; \sum_{j : (v,j) \in E} x_{vj} - \sum_{i : (i,v) \in E} x_{iv} = \tau_v
\end{align*}}
Here, $\tau_v$ is equal to $1$ if $v=s$, $-1$ if $v=t$ and $0$ otherwise. $x_{i,j}$ is a binary decision variable determining whether the chosen path from the source to the destination goes through the edge $(i,j)$. Since the constraints are totally unimodular, the constraint $x\in \{0,1\}^E$ can be relaxed to $x\in [0,1]^E$ (see \cite{wolsey2020integer}). We have therefore an LP and our theory applies.

\paragraph{Uncertainty Set.} 
In practice, the uncertainty set can capture several aspects of the problem. For example, it can impose positivity of all edge costs, encode lower and upper bounds, or incorporate structural relationships such as linear correlations with observable features (as we will discuss in the application of \cref{sec: hiring experiments}).

In our experiments, we assume that the decision maker has a nominal estimate $c_0$ of the cost vector, derived from historical data and preliminary expert assessments. This estimate is uncertain by some margin $\varepsilon$.
This is captured by the following box uncertainty set
$$\cC_{\varepsilon}:=\{c\in \R^E,\; (1-\varepsilon)c_0 \leq c \leq (1+\varepsilon)c_0\},$$
where $\varepsilon\in (0,1)$ measures the level of uncertainty in the nominal estimate. The assumption is that the true cost vector satisfies $c \in \cC_{\varepsilon}$.
When $\varepsilon=0$, the nominal estimation $c_0$ is assumed to be perfectly accurate. As $\varepsilon$ increases, the estimates become more uncertain.

We run Algorithms \ref{algorithm to compute delta} and \ref{alg:data collection algorithm} on the graph of a real neighborhood in the United States (see \cref{fig: cambridge map and shortest path}), using progressively larger values of $\varepsilon$. This allows us to illustrate how the minimal sufficient data set—that is, the smallest set of edges whose costs must be observed—depends jointly on the structure of the task (the underlying graph) and on the uncertainty set.

\begin{figure}[h] 
    \centering
    \begin{minipage}{0.5\textwidth} 
        \centering
        \includegraphics[scale=0.35]{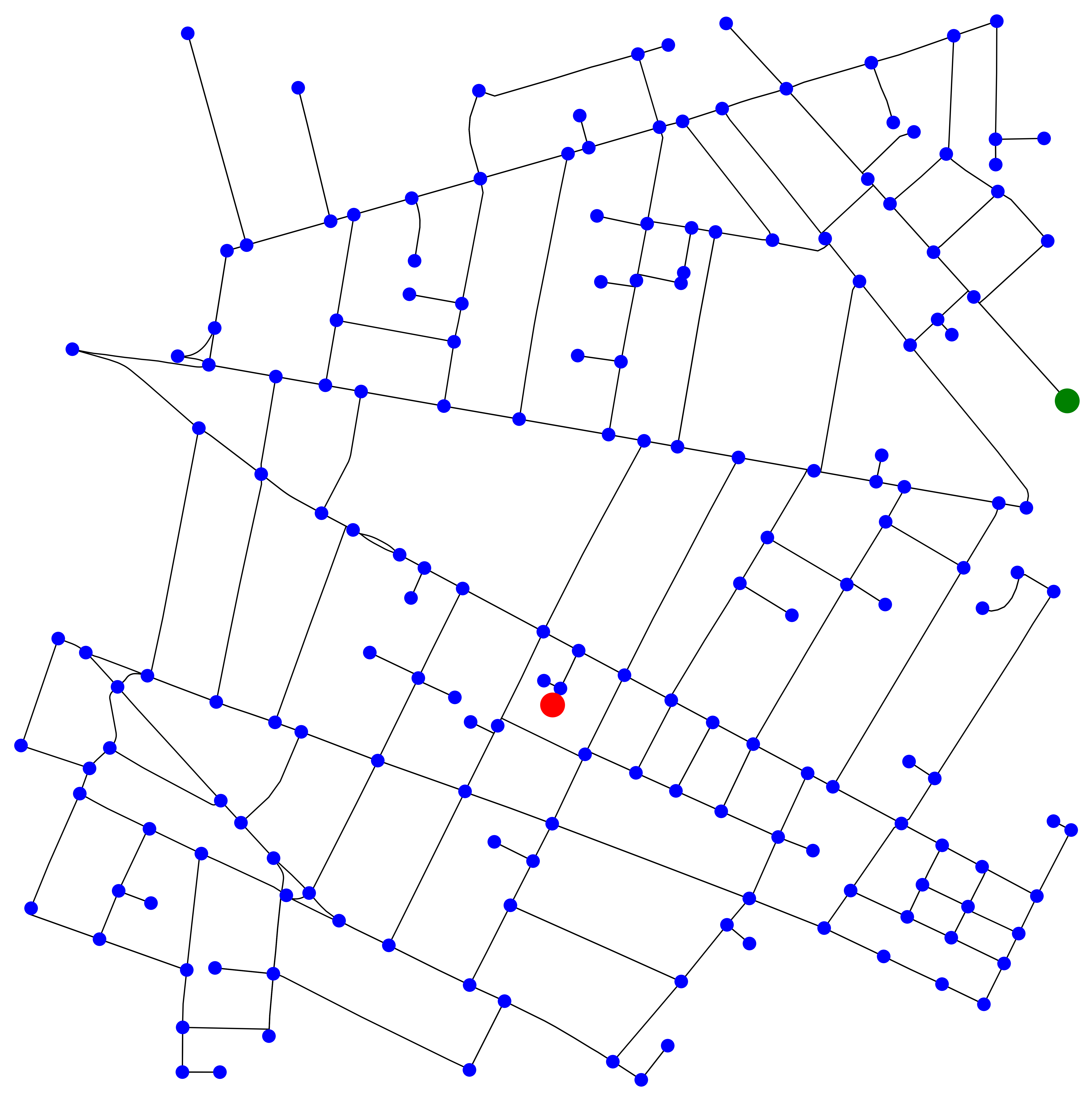}
        \end{minipage}\hfill
    \begin{minipage}{0.5\textwidth} 
        \centering
        \includegraphics[scale=0.35]{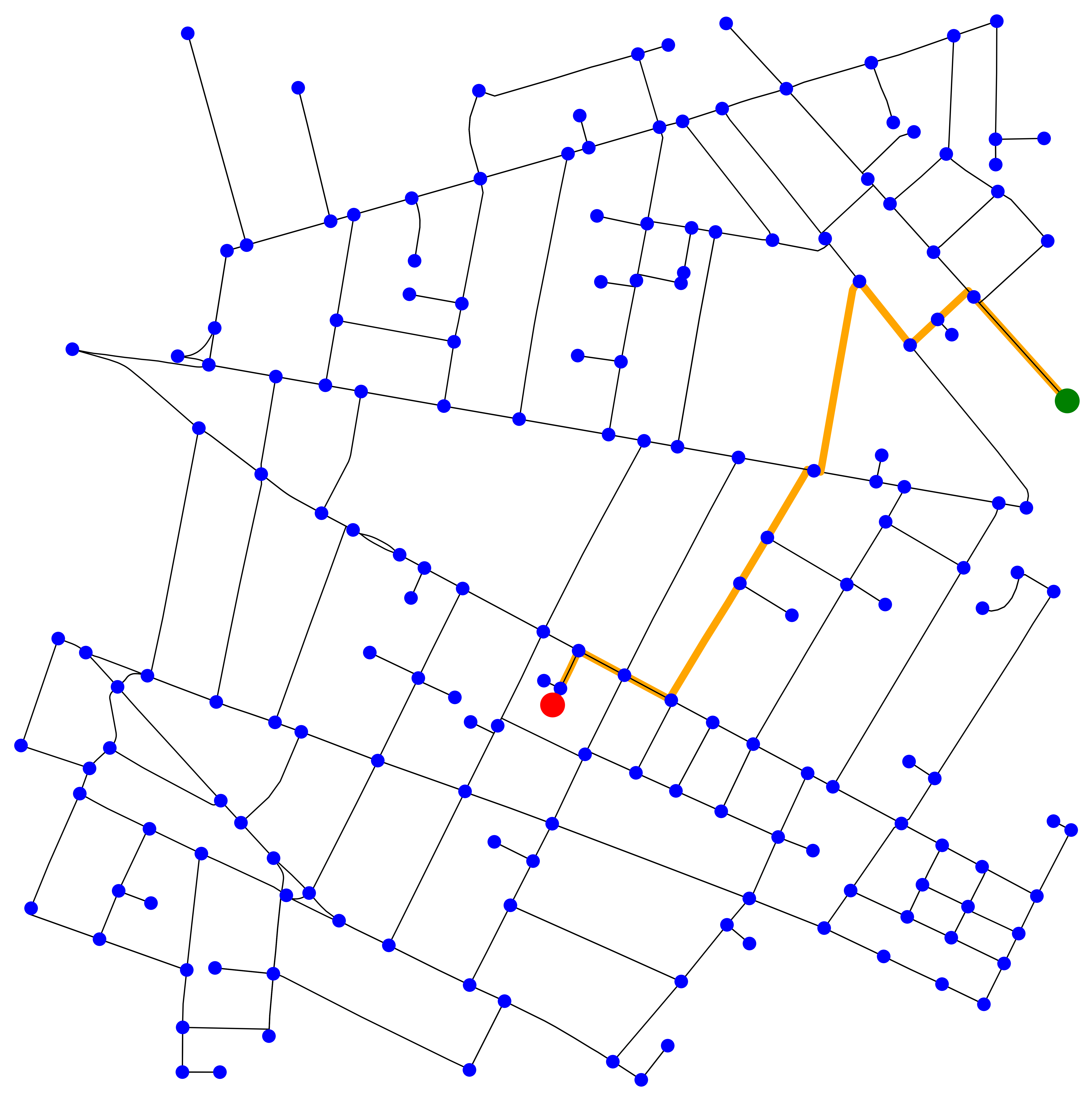}
        \end{minipage}
  \caption{US neighborhood graph map with origin in green and destination in red (left). $c_0$ is taken as the edge lengths, and the corresponding shortest path is in orange (right).}
    \label{fig: cambridge map and shortest path}
\end{figure}

\begin{figure}
    \centering
    \begin{minipage}{0.32\textwidth}
        \centering
        \includegraphics[width=\linewidth]{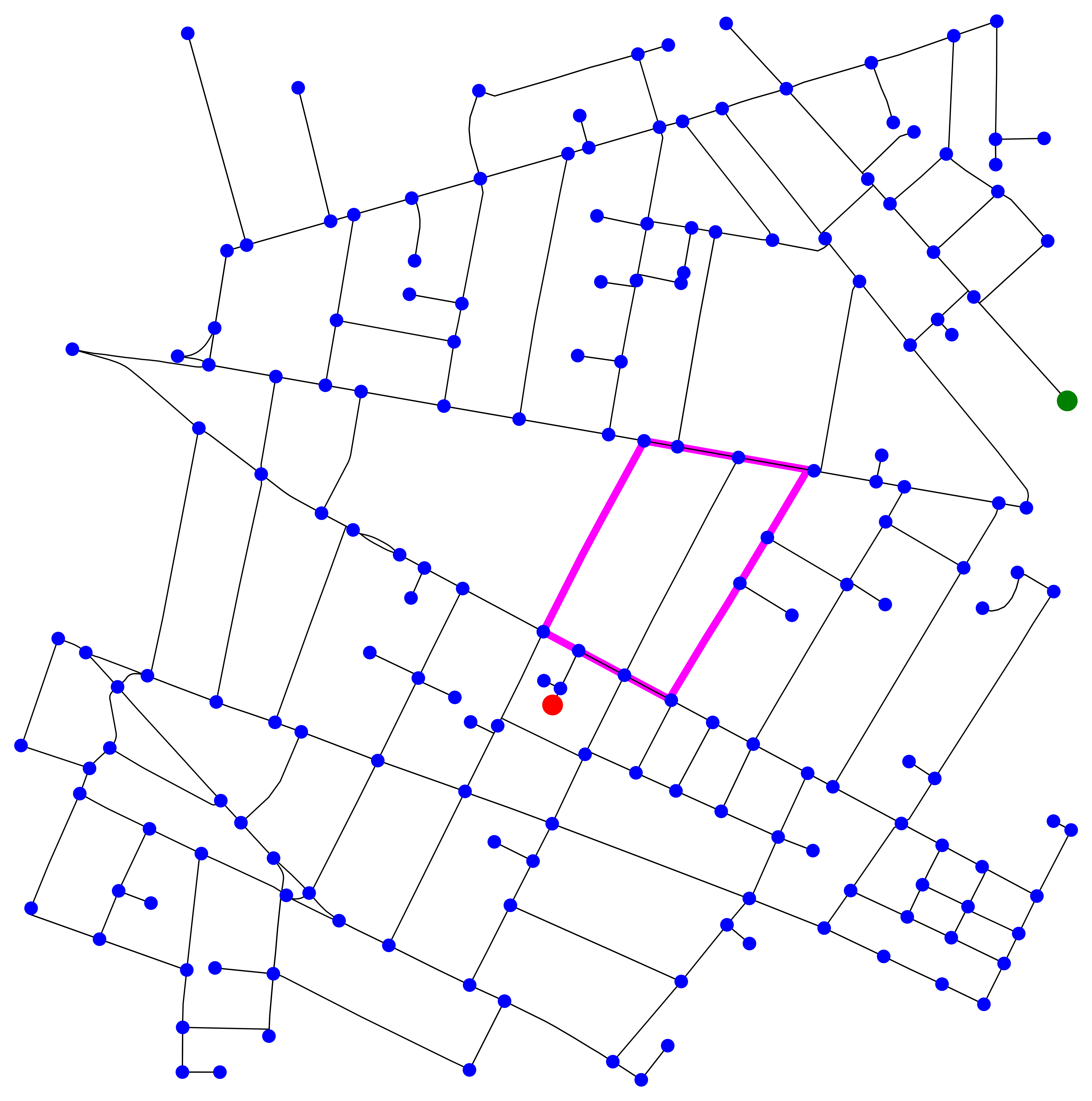}
    \end{minipage}\hfill
    \begin{minipage}{0.32\textwidth}
        \centering
        \includegraphics[width=\linewidth]{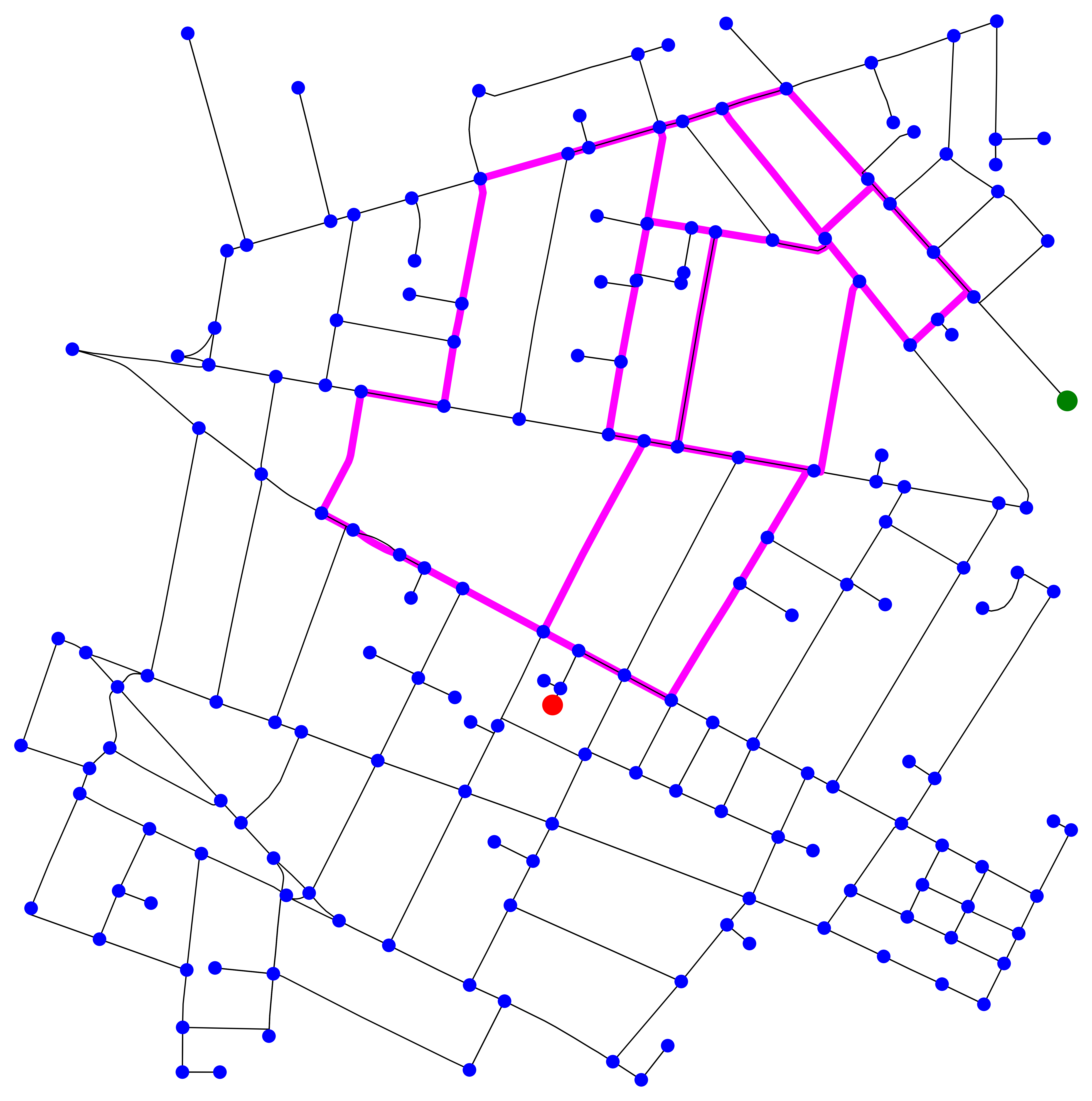}
    \end{minipage}\hfill
    \begin{minipage}{0.32\textwidth}
        \centering
        \includegraphics[width=\linewidth]{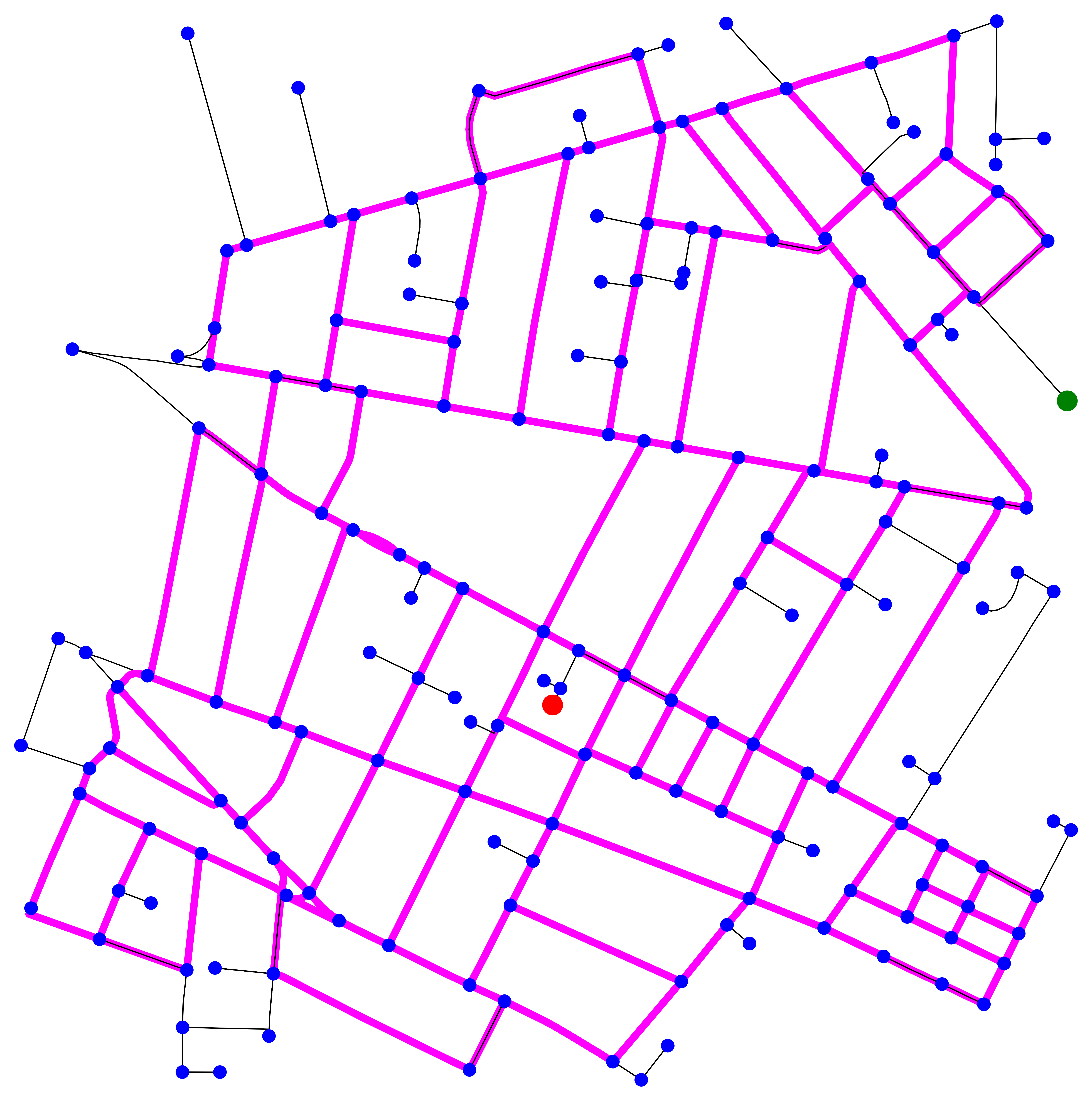}
    \end{minipage}

    \caption{Minimal edges to observe (in magenta) to find an optimal solution for values of $\epsilon=7\%, 30\%,\; 99\%$ (left to right).}
    \label{fig:progressive epsilon experiments}
\end{figure}

\paragraph{Insights}
\cref{fig:progressive epsilon experiments} shows the output of our algorithm for various values of \(\varepsilon\): the minimal set of edges whose costs must be queried to determine the optimal path. Clearly the task structure (the graph here) strongly influences which set of edges are sufficient. As intuition 
suggests, the selected edges tend to be those positioned strategically between the source and destination, where their costs can influence which path is ultimately optimal.

The geometry of the uncertainty set also plays a critical role. As \(\varepsilon\) increases---that is, as the box \(\mathcal{C}_{\varepsilon}\) becomes larger---the 
number of edges that must be observed naturally increases. In particular, when 
\(\varepsilon = 99\%\) (effectively equivalent to 
\(\mathcal{C}_{\varepsilon} = \mathbb{R}_{>0}^{E}\)), our algorithm recommends observing all edges except those that are provably irrelevant: edges connected 
only to isolated nodes, edges that are inaccessible, or edges that must be taken in every feasible \(s\)--\(t\) path. Naturally, such edges do not affect which path is optimal, and our algorithm correctly omits them. This highlights once more the nontrivial interplay between task structure and the informational value of data.

At $\varepsilon =7\%$, a careful reader may notice that one edge located inside the selected region does \emph{not} appear in the minimal sufficient set. This edge is a one-way, oriented 
northeast, and therefore is too costly to use in a path from \(s\) to \(t\). 
Given the uncertainty set, it is never optimal and thus irrelevant for the decision. This again illustrates the subtle ways in which task structure and uncertainty jointly determine which data points are informative.

The reader might also notice that the minimal sufficient data set here is a collection of cycles. For example, when $\varepsilon =7\%$, it is the difference between the optimal path for $c_0$ and an alternative route that can potentially be faster under plausible perturbations. This brings practical sense to our theory. Indeed, by \cref{thm: relatively open characterization}, the selected 
queries must span \(\Delta(\mathcal{X}, \mathcal{C})\), the set of relevant extreme directions. In flow problems, extreme directions correspond to cycles, which explains why a cycle structure emerges naturally in this flow setting. 
Cycles are precisely the most informative objects to measure, as they allow the decision maker to compare competing path alternatives.

\subsection{Hiring Interviews}  \label{sec: hiring experiments}

Our second application is a hiring problem. Here, a decision-maker is given a large pool of candidates along with their resumes and must decide which subset of candidates to interview in order to make the optimal hiring decision.
This problem has been studied in various settings \citep{purohitHiringUncertainty2019, epsteinSelectionOrdering2024}, including within the popular Secretary Problem \citep{kleinbergMultiplechoicesecretary2005,arlottoUniformlyBounded2019, brayDoesMultisecretary2019}. Prior work typically assumes a sequential, adaptive model, where interviews and hiring decisions occur in an online fashion. However, in many real-world scenarios---such as hiring PhD students or faculty---the set of candidates to interview must be chosen in advance, with hiring decisions made afterward based on all interview outcomes. This latter \textit{non-adaptive} setting is a natural instance of the data informativeness problem.
The smallest sufficient data set here is the smallest subset of candidates to interview to recover the optimal hiring decision.

\textit{Non-adaptive Secretary Problem as an LP.}
Formally, hiring from $d$ candidates is a decision-making problem where a decision consists of a binary vector $x \in  \cX \subset \{0,1\}^d$ indicating which candidates to hire. 
The feasible set $\cX$ encodes organizational constraints, such as a maximum number of hires $\sum_{i=1}^d x_i \leq k$, or maximum expertise-based quotas $\sum_{i \in I_j} x_i \leq k_j$ for subsets $I_j \subseteq [d]$, to name a few. 
Each candidate has an unknown value $c_i$, with $c \in \cC \subset \Re^d$ modeling prior information on these values. The objective function is $c^\top x$, the total value of the hired candidates.
A data set $\cD \subset \cQ := \{e_i \; :\; i\in [d]\}$ is a subset of candidates to interview, and each interview $q \in \cD$, $q_j=1$, reveals the value $c_j$ of a given candidate. The goal in this application is to select the smallest subset of candidates to interview to recover the optimal hiring decision: that is, the smallest sufficient data set. The constraints of this problem are totally unimodular, so $x \in \{0,1\}^d$ can be relaxed to $x \in [0,1]^d$ and the problem formulated as an LP \citep[Chapter 3]{wolsey2020integer}. Hence, our theory applies to this setting.

\textit{Task Structure and Uncertainty Set.}
In this example, the goal is to hire $20$ candidates from a pool of $d=100$ candidates. Each candidate is associated with two features: GPA and years of experience.
We study two settings: vanilla hiring, with only a total hire cap, and experience-constrained hiring, which also limits hires per seniority group. The decision sets are 
$\cX_{\text{vanilla}} := \{x \in \{0,1\}^d : \sum_{i=1}^d x_i \leq 20\}$ and $\cX_{\text{experience}} := \{x \in \cX_{\text{vanilla}} : \forall j \in [4], \sum_{i \in I_j} x_i \leq 8\},$ where $I_j$ is the set of candidates with $j$ years of experience. 
 We assume a misspecified linear model, i.e. candidate values belong to
{\setlength{\abovedisplayskip}{1pt}
 \setlength{\belowdisplayskip}{1pt}
 \setlength{\abovedisplayshortskip}{1pt}
 \setlength{\belowdisplayshortskip}{1pt}
\begin{align*}
\cC := \{c \in \mathbb{R}^d : \exists \alpha \in \mathbb{R}^2,\; \exists \varepsilon \in [-\eta, \eta],\; \ell \leq \alpha \leq u,\; c = \alpha^\top\phi  + \varepsilon\},
\end{align*}}
where $\eta \geq 0$ controls the misspecification level, and $\ell=(4,4), u =(5,5)$. $\phi$ is a feature matrix whose rows are GPAs and years of experience of candidates. 

The GPAs of candidates are generated using a uniform distribution in the interval $[2,4]$, and the level of experience is also uniform in $\{1,2,3,4,5\}$. 
\cref{fig:hiring experiment} is the output of \cref{alg:data collection algorithm}: the smallest set of candidates to interview to enable an optimal hiring decision.

\paragraph{Impact of $\cC$.} As misspecification increases (the uncertainty set $\cC$ grows larger, from left to right), so does the number of required interviews: more uncertainty requires more data points. 

\paragraph{Impact of $\cX$.} In the first row of \cref{fig:hiring experiment}, candidates fall into three groups: low potential value (never hired), high potential value (always hired), and mid uncertain value (interviewed)---an intuitive pattern given the task, automatically recovered by our algorithm.

When adding group hiring constraints---second row of \cref{fig:hiring experiment}---a similar pattern arise, but now across experience groups rather than the entire population: low misspecification yields separate treatment between experience groups, as values don’t overlap across experience levels; high misspecification leads to cross-experience group comparisons and mixing---again, an intuitive pattern given the new constraints. 

One would naively expect that since $\cX_{\text{experience}}$ has more constraints than $\cX_{\text{vanilla}}$, more data would be needed to make an optimal decision in the vanilla setting, but that is not necessarily true. In \cref{fig:hiring experiment}, we see that in the high misspecification regime, more data is needed for the experience-constrained setting than the vanilla setting. In reality, the data needed depends on the geometry of the decision set $\cX$ relative to the uncertainty set $\cC$, as can be seen from \cref{thm: relatively open characterization} and \cref{cor:dualpoly-carac}.

\begin{figure}
    \centering
    \includegraphics[width=1\linewidth]{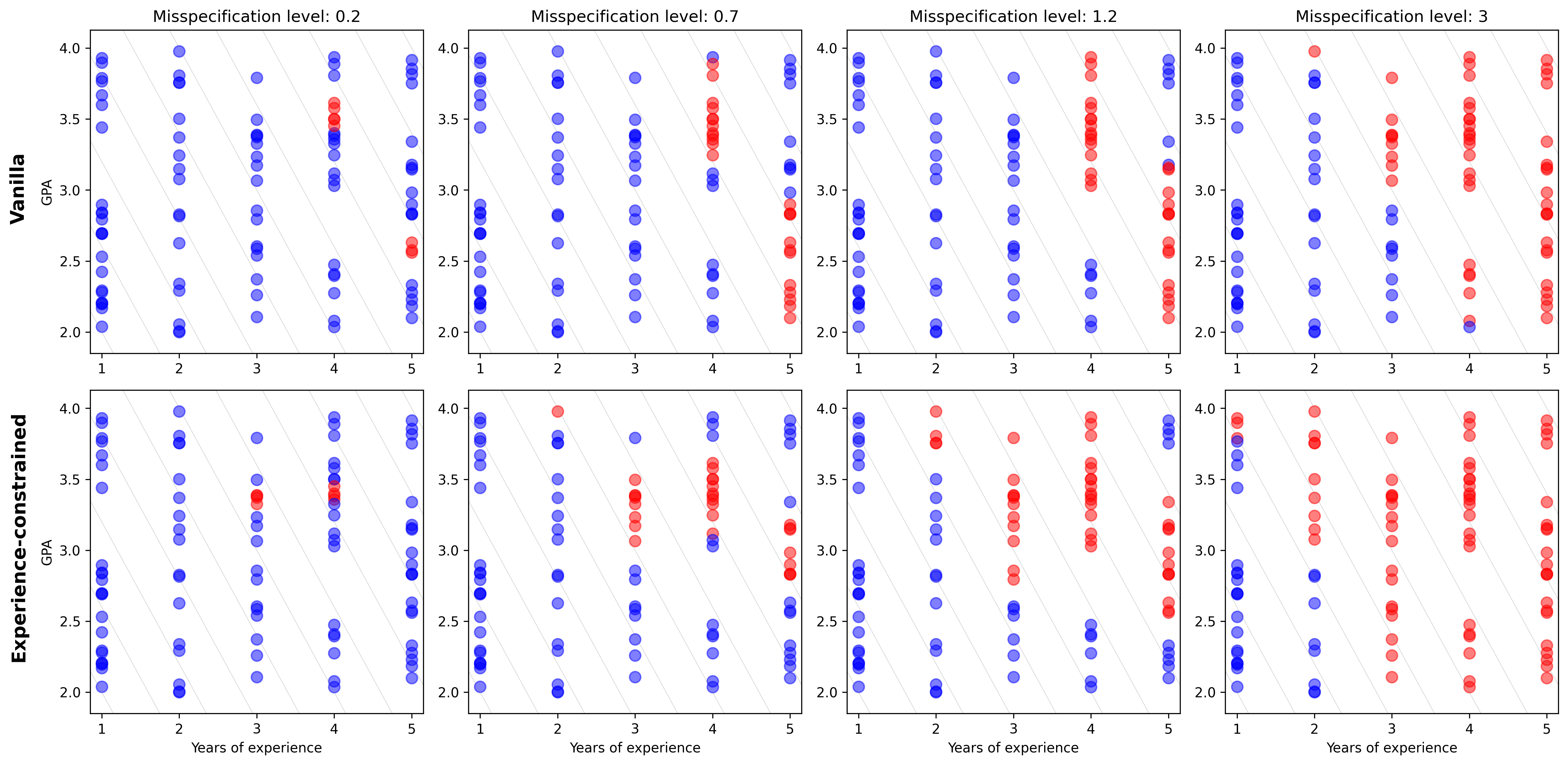}
    \caption{Candidates to be interviewed (in red) to make an optimal hiring decision. Number of candidates to interview from left to right for top and bottom row respectively: $8,24,31,52$ and $8,28,43,70$. }
    \label{fig:hiring experiment}
\end{figure}

\section{Conclusion: Limitations, Extensions and Open Problems}\label{sec:conclusion}

This paper introduces a framework for quantifying the informativeness of data sets in decision-making tasks. While our analysis yields sharp results, several natural extensions and open problems remain, of which we mention a few:

\begin{enumerate}
    \item \textit{Beyond LPs:} Our study focuses on linear optimization. Extending the framework to broader classes of problems—such as mixed-integer or general convex programs—would enable task-focused data collection algorithms for a wider range of applications. \label{ext: obj}
    \item \textit{Hardness:} We have proven that once we have a basis of $\spc{\dualpoly{\cX}{\cC}}$, the problem of finding a minimal sufficient data set remains NP-hard. However, it is still unclear whether computing a basis of $\spc{\dualpoly{\cX}{\cC}}$ can be done in polynomial time, as our algorithm uses MIPs.
    \item \textit{$\epsilon$-optimality:} In many practical applications, decision-makers are satisfied with near-optimal rather than exactly optimal solutions. Studying the notion of sufficiency in the sense of recovering an $\epsilon$ optimal solution would provide insights into the trade-off between optimality and data size.
    \label{ext: eps}
    
    \item \textit{General observation structure:} Data in our model provides partial information about the unknown parameter in the form of linear projections. Studying more general observation structures would further broaden the scope of the framework. \label{ext: obs}

    \item \textit{Uncertainty in $A$:} The uncertainty in our analysis lies in the cost vector. Yet in many important applications, uncertainty appears in the constraints (e.g., in the matrix $A$).

\end{enumerate}
These extensions can be captured through a generalized notion of sufficiency. Let $\{f_c,\: c\in \cC\}$ define a parametric family of objective functions, and $\mathbb{O}(\cdot,\cdot)$ denote an observation structure, where $\mathbb{O}(q,c)$ is the observation obtained from query $q$ when the true parameter is $c$.

\begin{definition}[Generalized Sufficiency] \label{def:gen-sufficient}
    A data set (or set of experiments) $\cD=\{q_1,\ldots,q_N\}$ is a $\epsilon-$\suff\ for an uncertainty set $\cC\subset \R^d$ and decision set $\cX$ if there exists a mapping $\hat x:\R^{N} \longrightarrow \cal \cX$ such that
    {\setlength{\abovedisplayskip}{0pt}
 \setlength{\belowdisplayskip}{0pt}
 \setlength{\abovedisplayshortskip}{0pt}
 \setlength{\belowdisplayshortskip}{0pt}
    \begin{align*}
        \forall c \in \cC, \quad f_c\left(\hat x\left(\mathbb{O}(q_1,c),\ldots,\mathbb{O}(q_N,c)\right)\right) \leq \min_{x\in \cX}f_c(x) + \epsilon.
    \end{align*}}
\end{definition}

This formulation incorporates extension \ref{ext: obj} through the general class of parametric objective $f_c$, and recovers our setting with the class $f_c:x \to c^\top x$. Extension \ref{ext: eps} is built in through requiring $\epsilon-$optimality of which our original framework is the special case $\epsilon=0$. The general observation structure $\mathbb{O}$ models extension \ref{ext: obs}, with our case corresponding to $\mathbb{O}(q,c) = c^\top q$.

A further important extension arises when observations are stochastic rather than deterministic.

This can be modeled by considering $\mathbb{O}_q(\cdot|c)$, the distribution of possible observations under query $q$ when the true parameter is $c$. A data set $\cD=\{q_1,\ldots,q_N\}$ is then $\epsilon-$\suff\ when there exists a mapping $\hat x:\R^{N} \longrightarrow \cal \cX$ verifying
{\setlength{\abovedisplayskip}{3pt}
 \setlength{\belowdisplayskip}{3pt}
 \setlength{\abovedisplayshortskip}{3pt}
 \setlength{\belowdisplayshortskip}{3pt}
    \begin{align*}
         \forall c \in \cC, \quad \Eb_{o_1 \sim \mathbb{O}_{q_1}(\cdot|c),\ldots,o_N \sim \mathbb{O}_{q_N}(\cdot|c)}\left[f_c\left(\hat x\left(o_1,\ldots,o_N\right)\right)\right] \leq \min_{x\in \cX}f_c(x) + \epsilon.
    \end{align*}
    }


This definition aligns closely with Blackwell’s informativeness framework. A common instance of such observation structure is noisy observation of the sort $\mathbb{O}_{q}(\cdot|c) = {\cal N}(c^\top q, \sigma)$, that is, observations are $c^\top q$ with an additive Gaussian noise. This would model classical experimentation settings such as surveys or randomized trials. Studying this setting would reveal not only which components of the decision problem must be queried, but also which precision (e.g., number of repeated experiments) per component is required to recover near-optimal solutions. Intuitively, some part of the problem need more refined estimation than others for the purpose of decision-making.

\bibliography{references.bib}

\newpage

\begin{APPENDICES}

\crefalias{section}{appsec}

\section{Proofs for \cref{sec:problem-formulation}}
\subsection{Proof of Proposition \ref{prop:equivalence-argmin}}\label{prop:equivalence-argmin proof}
Before proving the proposition, we need to introduce the following lemma.
\begin{lemma}\label{existence of unique argmin}
    For any $x^\star\in \cX^\angle$, there exists $c\in \R^d$ such that $\arg\min_{x\in \cX}c^\top x=\{x^\star\}$, i.e. for all $\delta \in D(x^\star)$, $c^\top \delta >0$.
\end{lemma}
\begin{proof}{Proof. }
    Let $x^\star \in \cX^\angle$. Assume that such a $c\in \R^d$ does not exists. We first show that there exists $\delta^\star \in D(x^\star)$ such that $\Lambda(x^\star)\perp \delta^\star$. Suppose no such $\delta^\star$ exists, then for any $\delta \in D(x^\star)$, there would exist $v(\delta)\in \Lambda(x^\star)$ such that $v(\delta)^\top \delta >0$. Consequently, we have for any $\delta \in D(x^\star)$,
    \begin{align*}
        \left(\sum_{\delta'\in D(x^\star)}^{}v(\delta')\right)^\top \delta >0,
    \end{align*}
    which contradicts our initial assumption. 
    
    Let $N\in \mathbb N$ and $\delta_1,\dots,\delta_N$ such that $D(x^\star)=\{\delta_1,\dots,\delta_N\}$. Assume without loss of generality that $ \Lambda (x^\star)\perp \delta_N$. Consequently, we have for all $c \in \Re^d$
    \begin{align*}
        \left(\forall i \in [N-1],\; c^\top \delta_i\geq 0\right)\implies c^\top \delta_N\leq 0.
    \end{align*}
    We show that this implies that $-\delta_N$ belongs to the cone spanned by $\delta_1,\dots,\delta_{N-1}$, i.e. there exists $\mu_1,\dots,\mu_{N-1}\in \R^+$ such that $\sum_{i=1}^{N-1}\mu_i\delta_i=-\delta_N$. Assume that this is not true. Let $K$ be the cone spanned by $\delta_1,\dots,\delta_{N-1}$. Since $-\delta_N\not\in K$, then (by the separation lemma), there exists $u\in \R^d$ such that for all $h\in K$, we have $u^\top h\geq 0$ and $-u^\top \delta_N <0$. In particular, we have for all $i\in [N-1],$ $u^\top \delta_i\geq 0$ and $u^\top \delta_N>0$, a contradiction. Hence, there exists $\alpha_1,\dots,\alpha_{N-1}\in \R^+$ such that 
    \begin{align*}
        -\delta_N=\sum_{i=1}^{N-1}\alpha_i \delta_i.
    \end{align*}

    Consequently, both $\delta_N$ and $-\delta_N$ are feasible directions from $x^\star$ in $\cX$, which contradicts the fact that $x^\star$ is an extreme point.
\end{proof}

We now prove Proposition \ref{prop:equivalence-argmin}.
\begin{proof}{Proof. }    
    It is easy to see that \ref{item:fullargmin} implies \ref{item:singlesol}. We now prove that \ref{item:singlesol} implies \ref{item:fullargmin}. Assume that \ref{item:singlesol} is verified but not \ref{item:fullargmin}, that is $\cD$ is not a \sufficient{}. From Theorem \ref{thm: relatively open characterization}, there exists $\delta\in \Delta(\cX,\cC)$ such that $\delta \not\perp \dir{\cC}\cap (\Vect \cD)^\perp$. By definition, there exists $x^\star \in \cX^\angle$ and $c\in \cC$ such that $c\in F(x^\star,\delta)$. From Lemma \ref{existence of unique argmin}, there exists $v\in \R^d$ such that for all $\delta \in D (x^\star)$, $v^\top \delta >0$. Let $\varepsilon>0$ such that $B(c,\varepsilon)\subset \cC$. Let $\eta>0$ small enough such that $c+\eta v\in B(c,\varepsilon)$, and $\eta '$ small enough such that $c_{\eta,\eta'}:=c+\eta v - \eta '\delta_{\dir{\cC}\cap (\Vect \cD)^\perp}\in B(c,\varepsilon)$. For any $\delta' \in D(x^\star)$, we have
    \begin{align*}
        (c+\eta v)^\top \delta' = \underbrace{c^\top \delta'}_{\geq 0} +\underbrace{\eta v^\top \delta'}_{>0} >0.
    \end{align*}
    This means that $\arg\min_{x\in \cX}(c+\eta v)^\top x=\{x^\star\}$. Furthermore, we have 
    \begin{align*}
        c_{\eta,\eta'}^\top \delta = \underbrace{c^\top \delta}_{=0, \text{ as }c\in F(x^\star,\delta)} +\eta v^\top \delta - \eta ' \underbrace{\norm{\delta_{\dir{\cC}\cap (\Vect(\cD))^\perp}\textbf{}}^2}_{\neq 0, \text{ as }\delta \not\perp \dir{\cC}\cap (\Vect \cD)^\perp }.
    \end{align*}
    
    Consequently, when $\eta$ is small enough compared to $\eta'$, we have $c_{\eta,\eta'}^\top \delta<0$, i.e. $c_{\eta,\eta'} \not\in \Lambda(x^\star)$. This means that $x^\star \not\in \arg\min_{x\in \cX}c_{\eta,\eta'}^\top x $. Assume that a mapping $\hat{x}$ satisfying condition \ref{item:singlesol} of the proposition. We have $c+\eta v-c_{\eta,\eta'}=\eta'\delta_{\dir{\cC}\cap (\Vect(\cD))^\perp}\in (\Vect(\cD))^\perp$. This means that for all $i\in [N]$, we have $(c+\eta v)^\top q_i=c_{\eta,\eta'}^\top q_i$, and hence 
    \begin{align*}
        \hat{x}((c+\eta v)^\top q_1,\dots,(c+\eta v)^\top q_N)=\hat{x}(c_{\eta,\eta'}^\top q_1,\dots,c_{\eta,\eta'}^\top q_N),
    \end{align*}
    which implies that
    \begin{align*}
        \hat{x}(c_{\eta,\eta'}^\top q_1,\dots,c_{\eta,\eta'}^\top q_N)\in \left(\arg\min_{x\in \cX}c_{\eta,\eta'}^\top x\right)\cap \left(\arg\min_{\cX}(c+\eta v)^\top x\right)=\varnothing,
    \end{align*}
    which is impossible.

\end{proof}
\subsection{Proof of \cref{prop:sufficient:projections}}
\begin{proof}{Proof. } \;
    \begin{itemize}
        \item
        $(\Rightarrow)$ Assume that $\cD$ is a \sufficient{}. Let $c,c'\in \cC$ such that $c_{\Vect \cD}=c'_{\Vect \cD}$. We have for any $q\in \cD$, $c^\top q=c'^\top q$. Let $\hat{X}$ given by Definition \ref{def:sufficient}. We have $\hat X\left(c^\top q_1,\dots,c^\top q_N\right)=\hat X\left(c'^\top q_1,\dots,c'^\top q_N\right)$ i.e. $\arg\min_{x\in \cX}c^\top x=\arg\min_{x\in \cX}c'^\top x$.
        
        \item $(\Leftarrow)$ Assume that $\cD$ satisfies the property of the proposition. Since for any $c,c'\in \cC$ we have 
        $
            c_{\Vect \cD}=c'_{\Vect \cD}\Longleftrightarrow (c^\top q)_{q\in \cD}=(c'^\top q)_{q\in \cD},
        $
        then for any $c\in \cC$, we define $\hat X\left(c^\top q_1,\dots,c^\top q_N\right)$ to be equal to $\arg\min_{x\in \cX}c'^\top x$ for any $c'$ such that $c'_{\Vect \cD}=c_{\Vect \cD}$. This mapping is well-defined and verifies the desired property.
    \end{itemize}
\end{proof}
\section{Proof for \cref{sec:characterizing-sufficient-data sets}}

        \subsection{Proof of Proposition \ref{prop:suff:vectorspace}} \label{proof prop:suff:vectorspace}
        \begin{proof}{Proof. }
        We denote $F:=\Vect \cD$. The condition $F_0\cap \Ker A \subset \cC^\perp + \Vect \cD$ is equivalent to $F_0\cap \Ker A\perp \cC\cap F^\perp$, so in order to prove the equivalence with the 3rd proposition, we will prove the equivalence with $F_0\cap \Ker A\perp \cC\cap F^\perp$.
    \begin{itemize}
        \item $\eqref{item:prop:Cvector-space:KerA} \Rightarrow \eqref{item:prop:Cvector-space:suff}$ Assume that $\cC\cap F^\perp\perp F_0 \cap \Ker A$.
        Let $c,c'\in \cC$ such that $c_F=c'_F$. We will show that they have the same $\argmin$, which proves \suff{} as a result of \cref{prop:sufficient:projections}.
        Let $\cC' = (\cC\cap F^\perp)^\perp \cap \cC$. We have
        \begin{align*}
            c=c_{\cC\cap F^\perp} +c_{\cC'},\; c'=c'_{\cC\cap F^\perp} +c'_{\cC'}.
        \end{align*}
       
        Hence, we have
        \begin{align*}
            &\proj{F}{c}=\proj{F}{c'}\\
            &\Longrightarrow \proj{F}{c_{\cC\cap F^\perp} +c_{\cC'}}=\proj{F}{c'_{\cC\cap F^\perp} +c'_{\cC'}}
            \\&\Longrightarrow \underbrace{\proj{F}{c_{\cC\cap F^\perp}}}_{=0}+\proj{F}{c_{\cC'}}=\underbrace{\proj{F}{c'_{\cC\cap F^\perp}}}_{=0}+\proj{F}{c'_{\cC'}}
            \\ &\Longrightarrow \proj{F}{c_{\cC'}}=\proj{F}{c'_{\cC'}}
            \\ & \Longrightarrow c_{\cC'} - c'_{\cC'} \in F^\perp
        \end{align*}
       We have $c_{\cC'} - c'_{\cC'} \in F^\perp \cap \cC'\subset F^\perp \cap \cC$ and $\cC' \subset (F^\perp \cap \cC)^\perp$. Hence, $c_{\cC'} = c'_{\cC'}$.

        Using this result, we can write $c - c' = c_{\cC\cap F^\perp}-c'_{\cC\cap F^\perp} \in \cC\cap F^\perp$. We show that the mapping $x \in \cX \to (c-c')^\top x$ is constant. In fact, for $x,x' \in \cX$, we have
        \begin{align*} 
            (c-c')^\top x - (c-c')^\top x' = \underbrace{(c-c')^\top}_{\in \cC\cap F^\perp} \underbrace{(x-x')}_{\in F_0 \cap \Ker A}=0,
        \end{align*}
        by the assumption $\cC\cap F^\perp\perp F_0 \cap \Ker A$.
Hence, the mappings $x\longmapsto c^\top x$ and $x\longmapsto c'^\top x$ are identical in $\cX$, within a constant. Consequently, we have $\arg \min_{x\in \cX}c^\top x = \arg\min_{x\in \cX}c'^\top x$.
        \item $\eqref{item:prop:Cvector-space:suff}\Rightarrow \eqref{item:prop:Cvector-space:one-solution}$ This implication follows directly from \cref{prop:sufficient:projections}.

        \item $\eqref{item:prop:Cvector-space:one-solution}\Rightarrow \eqref{item:prop:Cvector-space:KerA}$ Assume that $(3)$ is not true, that is $\cC\cap F^\perp\not\perp F_0 \cap \Ker A $.
        Let $c\in \cC$.
        We would like to show that there exists $c'\in \cC$ such that $c_F=c_F'=a$ and $\arg\min_{x\in \cX}c^\top x \neq \arg\min_{x\in \cX}c'^\top x$. 
      Let $x^\star(c)\in \arg\min_{x\in \cX}c^\top x$. There exists a set of feasible directions for $x^\star(c)$, $V=\{\delta_1,\dots,\delta_r\} \subset \FD(x^\star(c))$, that spans $F_0 \cap \Ker A$ (see \cref{lemma:feasible directions in general polyhedron}). Since $V$ spans $F_0\cap \Ker A$, and $\cC\cap F^\perp\not\perp F_0 \cap \Ker A $, then there exists $\delta\in V$ such that $\proj{\cC\cap F^\perp}{\delta}\neq 0$. Let $M$ be a positive constant and define $c' = c -M \proj{\cC\cap F^\perp}{\delta}$. As $\cC$ is a vector space, and each term in the definition of $c'$ is in $\cC$, we have $c' \in \cC$. Moreover, we have $c'_F = c_F$. For all $\alpha >0$ such that $x^\star(c)+\alpha \delta \in \cX$, we have
    \begin{align*}
                c'^\top (x^\star(c)+\alpha \delta) &= {c'}^\top x^\star(c) + \alpha c^\top \delta - \alpha M\proj{\cC\cap F^\perp}{\delta}^\top \delta\\
                &=c'^\top x^\star(c) + \alpha c^\top \delta - \alpha M \norm{\proj{\cC\cap F^\perp}{\delta}}^2.
            \end{align*}
            When $M$ is set to be large enough, we can see that we have $c'^\top (x^\star(c)+\alpha \delta )<c'^\top x^\star(c)$, which means that $x^\star (c)\not\in \arg\min_{x\in \cX}c'^\top x$.
  
  \end{itemize}

  

Let us now prove the final part of the proposition. Let $K>0$ and $\hat{x}:\R^N \longrightarrow \cX$. We can show that $FD\left(\hat{x}\left(c^\top q_1,\dots,c^\top q_N\right)\right)$ (see Proposition \ref{feasible directions polyhedral cone} for the definition of $FD\left(\hat{x}\left(c^\top q_1,\dots,c^\top q_N\right)\right)$) spans $F_0\cap \Ker A$ (see Lemma \ref{lemma:feasible directions in general polyhedron}). Furthermore, we have $\cC\cap F^\perp \not \perp F_0 \cap \Ker A$. Hence, there exists $\delta\in FD\left(\hat{x}\left(c^\top q_1,\dots,c^\top q_N\right)\right)$ such that $\delta_{\cC\cap F^\perp}\neq 0$. The condition $c\in \cC$ allows $c_{\cC\cap F^\perp }$ to take any value in $\cC \cap F^\perp $ without changing the values of $c^\top q_1,\dots, c^\top q_N$. Consequently, we set $c_{\cC\cap F^\perp }=-M\delta_{\cC\cap F^\perp }$ where $M$ is a nonnegative number that we will set later. Hence, letting $\alpha >0$ such that $\hat{x}\left(c^\top q_1,\dots,c^\top q_N\right)+\alpha \delta \in \cX$, we have 
\begin{align*}
    c^\top (\hat{x}\left(c^\top q_1,\dots,c^\top q_N\right)+\alpha \delta)&=c^\top \hat{x}\left(c^\top q_1,\dots,c^\top q_N\right) + \alpha c_F^\top \delta + \alpha c_{\cC\cap F^\perp }^\top  \delta\\
    &=c^\top \hat{x}\left(c^\top q_1,\dots,c^\top q_N\right) + \alpha c_F^\top \delta  - M \alpha \norm{\delta_{\cC\cap F^\perp }}^2
\end{align*}
Hence, we have $c^\top \hat{x}\left(c^\top q_1,\dots,c^\top q_N\right) + \alpha c_F^\top \delta  - M \alpha \norm{\delta_{\cC\cap F^\perp }}^2 \geq \min_{x\in \cX}c^\top x$, i.e. 
\begin{align*}
    c^\top \hat{x}\left(c^\top q_1,\dots,c^\top q_N\right) \geq - \alpha c_F^\top \delta  + M \alpha \norm{\delta_{\cC\cap F^\perp }}^2 +\min_{x\in \cX}c^\top x.
\end{align*}
Taking $M\geq \frac{K +\alpha c_F^\top \delta }{\alpha \norm{\delta_{\cC\cap F^\perp }}^2 }$, we indeed get 
\begin{align*}
    c^\top \hat{x}\left(c^\top q_1,\dots,c^\top q_N\right) \geq K +\min_{x\in \cX}c^\top x.
\end{align*}
 \end{proof}

\subsection{Proof of Proposition \ref{feasible directions polyhedral cone}} \label{feasible directions polyhedral cone proof}
\begin{proof}{Proof. }
    Let $x^\star \in \cX^\angle$. We denote $J=\{i\in [d], x^\star_i = 0\}$ and $I_0=\{i\in [d],\; \exists x\in \cX,\; x_i\neq 0\}$. For every $\delta \in \R^d$, we have 
    \begin{align*}
        \delta \in \FD(x^\star) \Longleftrightarrow \exists \varepsilon>0,\; x^\star +\varepsilon \delta \geq 0\text{ and }A\delta =0
        \Longleftrightarrow A\delta = 0 \text{ and }\delta_j \geq 0 \text{ for every }j\in J.
    \end{align*}
    This means that $\FD(x^\star)$ is a polyhedral cone, and $\FD(x^\star)\subset \Ker A$. Furthermore, since $[d]\setminus I_0 \subset J$, we also have $\FD(x^\star) \subset F_0$ which yields $\FD(x^\star)\subset F_0 \cap \Ker A$.
\end{proof}

\subsection{Proof of Proposition \ref{optimality cone def}} \label{proof optimality cone def}
\begin{proof}{Proof. }
Let $x^\star \in \cX^\angle$. For every $c\in \R^d$, we have
$$x^\star \in \arg\min_{x\in \cX}c^\top x \Longleftrightarrow \forall \delta \in \FD(x^\star),\; c^\top \delta \geq 0 \Longleftrightarrow \forall \delta \in D(x^\star),\; c^\top \delta \geq 0. $$
\end{proof}

\subsection{Proof of \cref{thm: relatively open characterization}}
\begin{proof}{Proof. }
     We denote $F=\Vect \mathcal D$. Notice that we have $\Delta(\cX,\cC) \subset \dir{\cC}^\perp + F \Longleftrightarrow \Delta(\cX,\cC) \perp \dir{\cC} \cap F^\perp.$ We will now prove that  $\dataset$ is a \sufficient \;for $\cC$ if and only if $ \Delta(\cX,\cC) \perp \dir{\cC} \cap F^\perp$.
    \begin{itemize}
        \item $(\Leftarrow)$ Suppose $\Delta(\cX,\cC) \not \perp \dir{\cC} \cap F^\perp$. There exists $\delta \in \Delta(\cX,\cC)$ such that $\delta \not\perp \dir{\cC}\cap F^\perp$. By definition, there exists $x\in \cX^\angle$ such that $\delta \in D(x)$ and $F(x,\delta) \cap \cC \neq \varnothing$. Let $v\in F(x,\delta) \cap \cC$. Let $\delta_0\in \dir{\cC}\cap F^\perp$ such that $\delta_0 ^\top \delta <0$ ($\delta_0$ exists because $\delta \not \perp  \dir{\cC} \cap F^\perp$). As $\cC$ is relatively open, $v \in \cC = \relint \cC$ and we can assume without loss of generality that $v+\delta_0\in \relint{\cC}$ by rescaling $\delta_0$. This is by definition of $\relint{\cC}$: since there exists $\varepsilon >0$ such that $B(v,\varepsilon)\cap \aff \cC \subset \relint \cC$, and $v\in \cC \subset \aff{\cC}$, $\delta_0\in \dir{\cC}$, then $\delta_0$ can be multiplied by $\frac{\varepsilon}{2\norm{\delta_0}}$ to satisfy $v+\frac{\varepsilon}{2\norm{\delta_0}} \delta_0\in B(v,\varepsilon) \cap \aff \cC \subset \relint \cC$. We know that $v\in F(x,\delta)\subset \Lambda(x)$, and $(v+\delta_0)^\top \delta = \delta_0^\top \delta <0$ which implies that $v+\delta_0\not\in \Lambda(x)$. Finally, since we have $\delta_0 \in F^\perp$, we have $(v+\delta_0)_F=v_F+\delta_{0,F}=v_F$. However, 
        $v \in \Lambda(x)$ and $v+\delta \not \in \Lambda(x)$ which implies $x\in \arg\min_{x'\in \cX}v^\top x'$ and $x\not\in\arg\min_{x'\in \cX} (v+\delta_0)^\top x'$, meaning that $\arg\min_{x'\in \cX} (v+\delta_0)^\top x'\neq \arg\min_{x'\in \cX} v^\top x'$. This
        implies that $\cD$ is not a \suff~
   from \cref{prop:sufficient:projections}.

        \item $(\Rightarrow)$ Suppose $\cD$ is not \suff. From \cref{prop:sufficient:projections}, there exists $c,c'\in \cC$ such that $c_F=c'_F$ and $\arg\min_{x\in \cX}c^\top x \neq \arg\min_{x\in \cX}c'^\top x $. It follows from the definition of the optimality cones $\Lambda$ (\cref{optimality cone def}) that there exists $x\in \cX^\angle$ such that $c\in \Lambda (x)$ and $c'\not\in \Lambda (x)$ (see also \cref{cone equivalence}). For any $\alpha\in[0,1]$, we denote $c_\alpha:=(1-\alpha) c + \alpha c'$
        
        \begin{align*}
        \alpha^\star :&=\sup\left\{\alpha \in [0,1]\; : \; c_\alpha \in \Lambda(x) \right\} \\
        &= \sup\left\{\alpha \in [0,1]\; : \; c_\alpha^\top\delta \geq 0, \; \forall \delta \in D(x) \right\}.
        \end{align*}
        
        Since $\cC$ is convex, we have $c_{\alpha^\star} \in \cC$.  Since $\Lambda(x)$ is a closed set, we have $c_{\alpha^\star} \in \Lambda(x)$ and hence we have $c_{\alpha^\star} \neq c'$ i.e. $\alpha^\star<1$. Let $\varepsilon\in(0,1-\alpha^\star)$ small enough such that for any $\delta \in D(x)$ such that ${c_{\alpha^\star}}^\top \delta>0$, we have $c_{\alpha^\star+\varepsilon}^\top \delta >0$. As $c_{\alpha^\star + \epsilon} \not \in \Lambda(x)$, there exists $\delta \in D(x)$ for which ${c_{\alpha^\star + \epsilon}}^\top \delta<0$. Such $\delta$ must verify ${c_{\alpha^\star}}^\top \delta = 0$ given the condition defining $\epsilon$. Hence, $c_{\alpha^\star} \in \Lambda(x) \cap \{\delta\}^\perp = F(x,\delta)$, and $c_{\alpha^\star} \in \cC$, which implies $F(x,\delta) \cap \cC \neq \emptyset$ and therefore $\delta \in \Delta(\cX,\cC)$. Moreover, we have $\underbrace{(c_{\alpha^\star+\varepsilon}-c_{\alpha^\star})^\top}_{=\varepsilon(c'-c)\in \dir \cC \cap F^\perp} \delta =  c_{\alpha^\star+\varepsilon}^\top \delta \neq 0$, i.e. $\delta \not \perp \dir \cC \cap F^\perp$, and consequently we have $\Delta(\cX,\cC) \not\perp \dir \cC \cap F^\perp$.
        \end{itemize}
        \end{proof}
  \subsection{Proof of Theorem \ref{thm:span delta is dir x}} \label{proof thm:span delta is dir x}
 Before proving the theorem, we will have to introduce a few lemmas and definition.
\begin{lemma} \label{pseudo continuity of argmin}
    For any $c\in \R^d$, there exists $\varepsilon>0$ such that for any $c'$ satisfying $\norm{c-c'}<\varepsilon$, $\arg\min_{x\in \cX}c'^\top x\subset \arg\min_{x\in \cX}c^\top x$.
\end{lemma}
\begin{proof}{Proof. }
    Assume that there exists $c\in \R^d$ such that for all $\varepsilon>0$, there exists $c'$ satisfying $\norm{c-c'}<\varepsilon$ and $\arg\min_{x\in \cX}c^\top x \not\subset \arg\min_{x\in \cX}c'^\top x$. There exists a sequence $(c'_n)_{n \in \integ}$ that converges to $c$ such that for all $n\in \mathbb N$, there exists $x\in \cX^\angle\setminus \arg\min_{x\in \cX}c^\top x$ such that $x\in \arg\min_{x\in \cX}c_n'^\top x$. Since there is a finite number of extreme points, there exists a subsequence $(c'_{\varphi(n)})_{n \in \integ}$ and $x\in \cX^\angle\setminus \arg\min_{y\in \cX}c^\top y$ such that for all $n\in \mathbb N$, we have $x\in \arg\min_{y\in \cX}c_{\varphi(n)}'^\top y$, i.e. $c'_{\varphi(n)}\in \Lambda(x)$. Hence, since $\Lambda(x)$ is closed, we have $c\in \Lambda(x)$ and $x\not\in \arg\min_{y\in \cX}c^\top y$ which is not possible.
\end{proof}

\begin{definition}[Extreme Point Neighbors]
    Let $\cC\subset \R^d$. For any two extreme points $x_1,x_2\in X^\angle$, we say that $x_1$ and $x_2$ are neighbors in $\cX$ if there exists an extreme direction $\delta \in D(x_1)$ such that $x_2=x_1+\delta$. We say that they are $\cC-$strong neighbors in $\cX$ if furthermore there exists $c\in \cC$ such that $x,x'\in \arg\min_{y\in \cX}c^\top y$.
\end{definition}

\begin{definition}[Connected and $\cC-$strongly Connected Points] \label{def:c-strong connectivity}
    For any subset $\cal Y\subset \cX^\angle$ and $\cC\subset \R^d$, for any pair of elements $x,x'\in \cal Y$, we say that $x,x'$ are connected by neighboring extreme points in $\cal Y$ if there exist $h\in \mathbb N$ and a sequence $x_1,\dots,x_h\in  \cal Y$ such that for all $i\in [h-1]$, $x_i$ and $x_{i+1}$ are neighbors in $\cX$ and $x_1=x$ and $x_h=x'$. We say that they are $\cC-$strongly connected, when $x_i$ and $x_{i+1}$ are $\cC-$strong neighbors.
    When there is no ambiguity, we say that $x$ and $x'$ are (strongly) connected. 
    
    We say that the set $\cal Y$ is ($\cC-$strongly) connected by neighboring extreme points if this property holds for any pair of extreme points in $\cal Y$. When there is no ambiguity, we say that $\cal Y$ is (strongly) connected. For any element $x$ of $\cal Y$, we call the ($\cC-$strong) connection class of $x$ the set of points in $\cal Y$ that are ($\cC-$strongly) connected by neighboring extreme points to $x$.
\end{definition}
\begin{lemma} \label{strong connection in argmin}
    For any $c\in \cC$, $\cX^\angle \cap \arg\min_{x\in \cX}c^\top x$ is $\cC-$strongly connected by neighboring extreme points in $\cX$.
\end{lemma}
\begin{proof}{Proof. }
    Let $c\in \cC$. Every extreme point in $\arg\min_{x\in \cX}c^\top x$ is also an extreme point in $\cX$ (see \cref{cone equivalence}), and every extreme direction in $\arg\min_{x\in \cX}c^\top x$ is also an extreme direction in $\cX$. Hence, since $\arg\min_{x\in \cX}c^\top x$ is a bounded polyhedron, $\cX^\angle \cap \arg\min_{x\in \cX}c^\top x$ is connected by neighboring extreme points in $\cX$. Furthermore, by definition, since $\cX^\angle \cap \arg\min_{x\in \cX}c^\top x\subset \arg\min_{x\in \cX}c^\top x$, then $\cX^\angle \cap \arg\min_{x\in \cX}c^\top x$ is $\cC-$strongly connected.
\end{proof}
\begin{lemma} \label{dual poly is strongly connected}
     When $\cC$ is convex, $\dualpoly{\cX}{\cC}\cap \cX^\angle$ is $\cC-$strongly connected by neighboring extreme points.
\end{lemma}
\begin{proof}{Proof. }
    Assume that there exist $x,x'\in  \dualpoly{\cX}{\cC}\cap \cX^\angle$ that are not strongly connected. Let $c,c'\in \cC$ such that $x\in \arg\min_{y\in \cX}c^\top y$, $x'\in \arg\min_{y\in \cX}c'^\top y$. For any $\alpha \in [0,1]$, we denote $c_\alpha:=(1-\alpha)c+\alpha c'$. Let 
    \begin{align*}
        U:=\{x^\star \in \cX^\angle,\; \exists \alpha\in[0,1],\; x^\star \in \arg\min_{y\in \cX}c_\alpha^\top y\}\subset \dualpoly{\cX}{\cC}\cap \cX^\angle.
    \end{align*}
    
    Let $K$ be the intersection of $U$ and the connection class of $x$. We have $x'\not\in K$. Let
    
    \begin{align*}
        \alpha^\star =\max\left\{\alpha \in [0,1],\;  K \cap \arg\min_{y\in \cX}c_\alpha^\top y \neq \varnothing\right\}.
    \end{align*}
    
    If $\alpha^\star=1$, then there exists $v\in \arg\min_{y\in \cX}c'^\top y$ such that $v\in K$. From Lemma \ref{strong connection in argmin}, $\cX^\angle \cap \arg\min_{y\in \cX}c'^\top y$ is $\cC-$strongly connected and $v\in K\cap \cX^\angle \cap \arg\min_{y\in \cX}c'^\top y$ and consequently $x'\in K$, and therefore is connected to $x$ which contradicts our assumption. Hence, we necessarily have $\alpha^\star<1$. 
    Furthermore, from Lemma \ref{pseudo continuity of argmin}, there exists $\varepsilon\in (0,1-\alpha^\star)$ such that 

    \begin{equation}\label{eq:intersec_argmins}
    \arg\min_{y\in \cX}c_{\alpha^\star + \varepsilon}^\top y \cap \arg\min_{y\in \cX}c_{\alpha^\star}^\top y \neq \varnothing.
    \end{equation}

    As $K\cap \arg\min_{y\in \cX}c_{\alpha^\star}^\top y \neq \varnothing$ and $\arg\min_{y\in \cX}c_{\alpha^\star}^\top y $ is $\cC-$strongly connected from \cref{strong connection in argmin}, we have $\arg\min_{y\in \cX}c_{\alpha^\star}^\top y \subset K$. Combined with \eqref{eq:intersec_argmins}, it implies that $K\cap \arg\min_{y\in \cX}c_{\alpha^\star+\varepsilon}^\top y \neq \varnothing$. This contradicts the supremum definition of $\alpha^\star$.
\end{proof}

We have now enough tools to prove the theorem.

\begin{proof}{Proof. } Proof of \cref{thm:span delta is dir x}

We have
    \begin{align} 
        \Vect \Delta(\cX,\cC)& \underset{(1)}{=} \Vect \{x_1-x_2,\; x_1,x_2\in \cX^\angle\cap \dualpoly{\cX}{\cC},\; x_1 \text{ and }x_2 \text{ are $\cC-$strong neighbors}\} \label{set 1}
        \\& \underset{(2)}{=} \Vect \{x_1-x_2,\; x_1,x_2\in \cX^\angle\cap \dualpoly{\cX}{\cC}\} \label{set 2}
        \\& \underset{(3)}{=}\dir{\cX^\angle\cap \dualpoly{\cX}{\cC}} \label{set 3}
        \\&\underset{(4)}{=}\dir{\dualpoly{\cX}{\cC}}.\label{set 5}
    \end{align}
    Let's justify each of the equalities above.
    \begin{itemize}
        \item (1) Let $\delta \in \Delta(\cX,\cC)$. There exists $c\in \cC$ and $x\in \cX^\angle$ such that $c\in F(x,\delta)$. This means that $x\in \arg\min_{y\in \cX}c^\top y$, $\delta$ is an extreme direction for $x$ in $\cX$, and $c^\top \delta=0$. Consequently, there exists $\eta>0$ such that $x':=x+\eta \delta $ is an extreme point
        , that is a neighbor of $x$ by definition. Also, we have $x'\in \arg\min_{y\in \cX}c^\top y$. Hence, $\delta =\frac{1}{\eta}(x'-x)$, which proves
        $$\Delta(\cX,\cC)\subset \Vect\{x_1-x_2,\; x_1,x_2\in \cX^\angle\cap \dualpoly{\cX}{\cC},\; x_1 \text{ and }x_2 \text{ are $\cC-$strong neighbors}\}.$$
        Conversely, if $x_1,x_2\in \cX^\angle \cap \dualpoly{\cX}{\cC}$ are $\cC-$strong neighbors, then there exists an extreme direction $\delta$ for $x_1$ such that $x_2=x_1 + \delta$ and $c\in \cC$ such that $x_1,x_2\in \arg\min_{y\in \cX}c^\top y$. Hence, we have $c^\top \delta=0$, and consequently $c\in F(x_1,\delta)$, which means that $\delta \in \Delta(\cX,\cC)$, i.e. $x_1-x_2 \in \Delta (\cX,\cC)$. This proves the desired equality.
        \item (2) Set \cref{set 1} is clearly a subset of set \cref{set 2}. Let's prove the converse inclusion. Let $x,x'\in \cX^\angle \cap \dualpoly{\cX}{\cC}$. According to Lemma \ref{dual poly is strongly connected}, there exists $h\in \mathbb N$ and a sequence $x_1,\dots,x_h\in \cX^\angle \cap \dualpoly{\cX}{\cC}$ such that $x_1=x$ and $x_h=x'$ and for all $i\in [h-1]$, $x_i,x_{i+1}$ are $\cC-$strongly connected. Hence, we have 
        \begin{align*}
            x-x'=\sum_{i=1}^{h-1}x_{i+1}-x_i.
        \end{align*}
        All of the terms in the sum above are in set $\eqref{set 1}$ and therefore their sum as well, by linearity. Hence, we indeed have the inclusion.
        \item (3) This equality is immediate since for any $x_1,x_2 \in \cX^\angle\cap \dualpoly{\cX}{\cC}$, $x_1-x_2=x_1 - x_0 -(x_2-x_0)$ for any $x_0\in  \cX^\angle\cap \dualpoly{\cX}{\cC}$ and consequently $x_1-x_2\in \dir{ \cX^\angle\cap \dualpoly{\cX}{\cC}}$. 
        \item (4) In order to prove this equality, we prove that $\dualpoly{\cX}{\cC}\subset \conv(\cX^\angle \cap \dualpoly{\cX}{\cC})$. Let $x\in \dualpoly{\cX}{\cC}$ and $c\in \cC$ such that $x\in \arg\min_{y\in \cX}c^\top y$. There exists $\alpha_1,\dots,\alpha_k\in (0,1]$ such that $\alpha_1+\dots+\alpha_k=1$ and $x_1,\dots,x_k\in \cX^\angle$ such that $x=\sum_{i=1}^{k}\alpha_kx_k$. We have 
        \begin{align*}
            \min_{y\in \cX}c^\top y \geq \sum_{i=1}^{k}\alpha_kc^\top x_k \text{ i.e. }\sum_{i=1}^{k}\alpha_k(c^\top x_k-\min_{y\in \cX}c^\top y)\leq 0.
        \end{align*}
        All of the terms in the sum are positive, and are consequently equal to $0$. Hence we have $x_1,\dots,x_k\in \arg\min_{y\in \cX}c^\top y \subset \dualpoly{\cX}{\cC}$. Consequently, we have $x\in \conv(\cX^\angle \cap \dualpoly{\cX}{\cC})$. Hence, we have 
        \begin{align*}
            \dir{\dualpoly{\cX}{\cC}}&\subset \dir{\conv(\cX^\angle \cap \dualpoly{\cX}{\cC})}\\&=\dir{X^\angle \cap \dualpoly{\cX}{\cC}} \\&\subset \dir{\dualpoly{\cX}{\cC}}.
        \end{align*}
        This proves the desired equality.
    \end{itemize}
\end{proof}

\subsection{Proof of \cref{cor:MIP-sufficiency}}
\begin{proof}{Proof. }
    Since minimizing $c^\top x$ when $x$ belongs to $\cX\cap\{0,1\}^d$ is the same as minimizing over its convex hull $\conv (\cX\cap\{0,1\}^d)$ which is a polyhedron, we have $\dir{\dualpoly{\conv(\cX \cap\{0,1\}^d)}{\cC}} = \dir{\dualpoly{(\cX \cap\{0,1\}^d)}{\cC}}$. The result then follows directly from \cref{cor:dualpoly-carac}.

\end{proof}

\section{Proofs for \cref{sec:algorithm}}
 
\subsection{Proof of \cref{thm:alg_termination}}
\begin{proof}{Proof. }\;
\begin{itemize}
    \item We first show that when the algorithm terminates, i.e., the condition of the while loop is no longer satisfied, then with probability 1, $\dir{\dualpoly{\cX}{\cC}}\subset\Vect \cD$.
    Notice that the constraints in the minimization and maximization problems in Algorithm  \ref{algorithm to compute delta} encode complimentary slackness and, therefore are equivalent to $$\min / \max \{ \alpha^\top \proj{(\Vect \cD)^\perp}{x^\star - x_0}  \; : \; c \in \cC, \; x^\star \in \argmin_{x \in \cX} c^\top x\}.$$
    By definition of $\dualpoly{\cX}{\cC}$, this equivalent to $$\min / \max \{ \alpha^\top \proj{(\Vect \cD)^\perp}{x^\star - x_0}  \; : \; x^\star \in \dualpoly{\cX}{\cC}\}.$$
    
     If the two problems have an optimal value equal to $0$, then $\proj{(\Vect \cD)^\perp}{\dir{\dualpoly{\cX}{\cC}}}\perp \alpha$ i.e. $\alpha \in \proj{(\Vect \cD)^\perp}{\dir{\dualpoly{\cX}{\cC}}}^\perp$. Unless $\proj{(\Vect \cD)^\perp}{\dir{\dualpoly{\cX}{\cC}}}^\perp=\R^d$, this set is of empty interior and its Lebesgue measure is equal to $0$, and consequently the probability of having $\proj{(\Vect \cD)^\perp}{\dir{\dualpoly{\cX}{\cC}}}\perp \alpha$ is zero since $\alpha$ has a continuous distribution. Hence, with probability $1$, we have $\proj{(\Vect \cD)^\perp}{\dir{\dualpoly{\cX}{\cC}}}=\{0\}$ i.e. $\dir{\dualpoly{\cX}{\cC}}\subset\Vect \cD$.

    \item  We now show that at every step of the algorithm, the dimension of the span of $\cD$ increases by $1$, and that it remains a linearly independent set, as well as satisfies $\Vect \cD \subset \dir{\dualpoly{\cX}{\cC}}$. Indeed, initially, $\cD$ is empty and is hence a linearly independent set and satisfies $\Vect \cD \subset \dir{\dualpoly{\cX}{\cC}}$. Assuming that $\cD$ is a linearly independent set and that $\Vect \cD \subset \dir{\dualpoly{\cX}{\cC}}$, if there exists $x\in \dualpoly{\cX}{\cC}$ such that $\alpha^\top \proj{(\Vect \cD)^\perp}{x_0-x}\neq 0$, then $\proj{(\Vect \cD)^\perp}{x_0-x}\neq 0$ with probability $1$ and consequently $x_0-x\in \dir{\dualpoly{\cX}{\cC}}\setminus \Vect \cD$. Hence, we have $\dim (\Vect (\cD\cup\{x_0-x\}))=\dim (\Vect \cD)+1$ and $\cD\cup\{x_0-x\}$ is a linearly independent set and satisfies $\Vect (\cD\cup\{x_0-x\}) \subset \dir{\dualpoly{\cX}{\cC}}$, which proves the desired result.
    \item Finally, combining the two results above, when the algorithm terminates, $\cD$ is a linearly independent set, and $\Vect \cD=\dir{\dualpoly{\cX}{\cC}}$ i.e. $\cD$ is a basis of $\dir{\dualpoly{\cX}{\cC}}$ with probability 1. Furthermore, the analysis above show that the algorithm indeed terminates after $\dim \dir{\dualpoly{\cX}{\cC}}$ iterations of the while loop.
\end{itemize}

\end{proof}

\subsection{Proof of \cref{sufficient data set lowerbound}}
\begin{proof}{Proof.}
    Assuming $\cD$ is sufficient data set, we have

    \begin{align*}
        \abs{\cD}&\geq \dim \Vect \cD \geq \dim \proj{\spc{\cC}}{\Vect \cD}\geq \dim \proj{\spc{\cC}}{A}\\&= \dim \spc{\dualpoly{\cX}{\cC}} - \dim \spc{\dualpoly{\cX}{\cC}} \cap \spc{\cC}^\perp =r(\cX,\cC).
    \end{align*}
\end{proof}

\subsection{Proof of \cref{lemma:problem equivalence}}
\begin{proof}{Proof.} \;
    \begin{itemize}
        \item $(\Rightarrow)$ Since $\spc{\dualpoly{\cX}{\cC}}\subset \spc{\cC}^\perp+\Vect \cD$, then we have $\proj{\spc{\cC}}{\spc{\dualpoly{\cX}{\cC}}}\subset \proj{\spc{\cC}}{\spc{\cC}^\perp+\Vect \cD}=\proj{\spc{\cC}}{\Vect \cD}$.
        \item $(\Leftarrow)$ Assume that $\proj{\spc{\cC}}{\spc{\dualpoly{\cX}{\cC}}}\subset \proj{\spc{\cC}}{\Vect \cD}$. Let $a\in \spc{\dualpoly{\cX}{\cC}}$. There exists $v\in \Vect \cD$ such that $\proj{\spc{\cC}}{a}=\proj{\spc{\cC}}{\Vect \cD}$. Hence, we have $a=v + \proj{\spc{\cC}^\perp}{a} - \proj{\spc{\cC}^\perp}{v}\in \spc{\cC}^\perp+\Vect \cD,$ which gives the desired result.
    \end{itemize}
\end{proof}

\subsection{Proof of \cref{hardness result}} \label{hardness result proof}
\begin{proof}{Proof. }


We will show that \cref{problem: space coverage problem} is equivalent to the sparse regression problem. Let $v\in \R^d$, $\cX=\conv\{0,v\}$ and $\cC=\R^d$. Let $\cQ$ be a finite set in $\R^d$ and $M$ a matrix whose columns are the elements of $\cQ$. In this case, \cref{problem: space coverage problem} is equivalent to finding the smallest subset $\cD\subset \cQ$ such that $v\in  \Vect{\cD}$. This shows that our problem is equivalent to finding the smallest subset (i.e. the sparsest linear combination of vectors) of $\cQ$ that spans $v$. This is equivalent to finding the sparsest vector $\phi\in \R^{\abs{\cQ}}$ such that $M\phi = v$ where $M$ is the matrix whose columns are the elements of $\cQ$, i.e. solve the problem
\begin{align*}
    \min_{}\norm{\phi}_0 \text{ s.t. }M\phi=v,
\end{align*}
where $\norm{\phi}_0$ is the number of nonzero coordinates of $\phi$. This problem is known as the sparse representation problem and is NP-hard \citep{natarajan1995sparse}.
\end{proof}
\subsection{Algorithm to compute a minimal sufficient data set in the basis vectors case}

\begin{algorithm}[h] 
    \caption{Computing a smallest \sufficient\ when $\cQ$ is a basis of $\R^d$}
    \label{alg:data collection algorithm}
    \KwIn{Polyhedron $\cX$, Uncertainty Set $\cC$, Query Set $\cQ = \{q_1,\ldots,q_d \}$ (basis of $\R^d$)}
    \KwOut{A minimal \suff\ data set under constraint $\cD \subset \cQ$.}
    Find $\{v_1,\dots,v_k\}$ a basis of $\spc{\dualpoly{\cX}{\cC}}$ using \cref{alg:LP_case}
    
    $Q \leftarrow [q_1,\ldots,q_d]$
    
    \textbf{return} $\cD:=\{q_i \; : \; i\in [d] \; \text{s.t.} \;\exists j\in [k],\; (Q^{-1}v_j)^\top e_i\neq 0\}$.
\end{algorithm}

\subsection{Proof of \cref{prop: hardness Q basis}}
\begin{lemma}
    Any matrix $H\in \R^{d\times d}$ in $\R^{d\times d}$ can be written as $PM$ where $P\in \R^{d\times d}$ is the orthogonal projection matrix on $\Vect H$, and $M\in \R^{d\times d}$ is an invertible matrix.
\end{lemma}
\begin{proof}{Proof.}
Let $H$ be a matrix in $\R^{d\times d}$ such that $\text{rank}\; H=k\in \mathbb N$. Let $U$ an orthogonal matrix whose first $k$ columns are a basis of $\Vect H$. Since the last $d-k$ columns of $U$ are orthogonal to all of the columns of $H$, the last $d-k$ rows of $U^\top H$ are equal to zero. Furthermore the nonzero rows of $U^\top H$ iare linearly independent, since there are $k$ of then and $\text{rank }U^\top H=\text{rank }H=k$. Hence, we can construct an invertible matrix $V$ such that its first $k$ rows are equal to the rows of $U^\top H$. Hence, we have
$$U^\top H=\begin{pmatrix}
        I_k & 0 \\ 0 & 0
    \end{pmatrix} V.$$
Hence, we have 
\begin{align*}
    H=\underbrace{U\begin{pmatrix}
        I_k & 0 \\ 0 & 0
    \end{pmatrix} U^\top}_{:=P} \underbrace{UV}_{:=M}=PM.
\end{align*}
    $M$ is invertible and $P$ is the othogonal projection matrix on $\Vect H$, which proves the desired result.
\end{proof}

\begin{proof}{Proof of the proposition.}
Let $M_\cQ$ be a matrix whose columns are elements of $\cQ$. Assume that $\spc{\dualpoly{\cX}{\cC}}$ is of dimension equal to $1$. Let $\delta\in \spc{\dualpoly{\cX}{\cC}}\setminus \{0\}$. \cref{problem: space coverage problem} is equivalent to finding the sparsest vector $S\in \R^d$ such that $$P_{\spc{\cC}}\delta=P_{\spc{\cC}}M_\cQ S,$$
where $P_{\spc{\cC}}$ is the orthogonal projection matrix on $\spc{\cC}$. Since $\spc{\cC}$ can be taken equal to any space, $M_\cQ$ can be taken equal to any invertible matrix, the left hand side of the equation can be taken equal to any element of $\spc{\cC}$, and from the lemma above, any matrix $H$ can be written as the product of an orthogonal projection and an invertible matrix for any matrix $H\in \R^{d\times d}$ and any vector $Y\in \Vect H$, the problem of finding the sparsest $S\in \R^d$ such that $HS=Y$ is an instance of \cref{problem: space coverage problem} when $\cC$ is relatively open and $\cQ$ is a basis of $\R^d$. This problem is well known to be NP-hard \citep{natarajan1995sparse}.
    
\end{proof}

\subsection{Modified version of \cref{algorithm to compute delta}} \label{appendix: modified alg to compute delta}
\cref{alg:modified-algorithm} is an alternative version of \cref{algorithm to compute delta} that returns the extreme points involved in computing a basis for $\spc{\dualpoly{\cX}{\cC}}$. The output of this algorithm is used to compute a nearly-smallest sufficient data set in \cref{alg: sufficient datatset for extreme points}.
 \begin{algorithm}[h] 
    \caption{Modified version of \cref{algorithm to compute delta}} \label{alg:modified-algorithm}
    \KwIn{Polyhedron $\cX = \{x \geq 0 \; : \; Ax=b\}$, Uncertainty set $\cC$.}
    \KwOut{A \sufficient\ $\cD$ such that $\abs{\cD}=r(\cX,\cC)+1$.}
    Initialize $\mathcal U$ and $\cD$ to $\varnothing$.
    
    Set $x^\star_0\in \arg\min_{x\in \cX}c_0^\top x$ an extreme point for some $c_0 \in \cC$.

    \quad $\mathcal U \leftarrow \mathcal U \cup \{x^\star_0\}$.
    
    Sample $\alpha \sim \mathcal{N}(0,Id)$.
    
    \textbf{while} either of the problems
    \begin{align*}
        \min&\;\alpha^\top \proj{(\Vect \cD)^\perp}{x_0-x} & \max&\;\alpha^\top \proj{(\Vect \cD)^\perp}{x_0-x}\\
        \text{s.t.}&\;x \geq 0,\; \lambda \in \R^m,\; s\in \R_+^d,\; c\in \cC& \text{s.t.} &\;x\geq 0,\; \lambda \in \R^m,\; s\in \R_+^d,\; c\in \cC \\
        &Ax=b, \; A^\top \lambda +s=c, & & Ax=b, \; A^\top \lambda +s=c \\
        & 1- \epsilon s_i \geq \tau_i \geq \epsilon x_i, \; \tau_i \in \{0,1\}, \; \forall i & & 1- \epsilon s_i \geq \tau_i \geq \epsilon x_i, \; \tau_i \in \{0,1\}, \; \forall i
    \end{align*}

    \quad has a solution with non-zero optimal value,

    \quad consider the corresponding optimal cost vector $c^\star$, and solve the following linear program using the simplex algorithm
    \begin{align*}
        \min&\;\alpha^\top \proj{(\Vect \cD)^\perp}{x_0-x} & \max&\;\alpha^\top \proj{(\Vect \cD)^\perp}{x_0-x}\\
        \text{s.t.}&\;x \geq 0,\; Ax=b,\;c^{\star\top}x\leq \min_{x'\in \cX}c^{\star\top}x' & \text{s.t.}&\;x \geq 0,\; Ax=b,\;c^{\star\top}x\leq \min_{x'\in \cX}c^{\star\top}x' 
    \end{align*}
    
    \quad Consider $x^\star$ the extreme points that is solution to any of the two problems above that has non-zero optimal value
    
    \quad $\mathcal U \leftarrow \mathcal U \cup \{x^\star\}$.

    \quad $\cD \leftarrow \cD \cup \{x_0^\star -x^\star\}$
    
    \quad resample $\alpha \sim \mathcal{N}(0,Id)$.

    \textbf{end while}

    \textbf{return} $\mathcal U$
\label{alg: modified algorithm to compute delta}
\end{algorithm}

\subsection{Justification of correctness of \cref{alg: sufficient datatset for extreme points}} \label{appendix: justification of algorithm for Q extreme points}
Since $\{x_0-x_1^\star,\dots,x_0-x_r^\star\}$ is basis of $\spc{\dualpoly{\cX}{\cC}}$, it is clear that for all $i\in \{1,\dots,r\}$, we have $\proj{\spc{\cC}}{x_i^\star - x_0}\in \proj{\spc{\cC}}{\spc{\dualpoly{\cX}{\cC}}}$ and $\Vect \cD_0 = \proj{\spc{\cC}}{\spc{\dualpoly{\cX}{\cC}}}$ where $\cD_0:= \{\proj{\spc{\cC}}{x_0-x_1^\star},\dots,\proj{\spc{\cC}}{x_0-x_r^\star}\}$. Hence, we can use gaussian elimination to extract from $\cD_0$ a basis $\cD_0':=\{\proj{\spc{\cC}}{x_0-x^\star_{i_1}},\dots,\proj{\spc{\cC}}{x_0-x^\star_{i_p}}\}$ of $\proj{\spc{\cC}}{\spc{\dualpoly{\cX}{\cC}}}$. This immediately implies that $\abs{\cD_0'}=\dim \proj{\spc{\cC}}{\spc{\dualpoly{\cX}{\cC}}}=r(\cX,\cC)$ and $\Vect \cD_0' = \proj{\spc{\cC}}{\spc{\dualpoly{\cX}{\cC}}}$. Finally, by considering $\cD=\{x_0,x_{i_1},\dots,x_{i_p}\}$, we clearly have $\proj{\spc{\cC}}{\spc{\dualpoly{\cX}{\cC}}}=\Vect \cD_0'\subset \Vect \proj{\spc{\cC}}{\Vect \cD}$, i.e. $\cD$ is a sufficient data set and $\abs{\cD}=\abs{\cD_0'}+1=r(\cX,\cC)+1$.
\subsection{Proof of Proposition \ref{prop: covspace vector space corectness}} \label{proof covspace vector space corectness}

Let $A:=\spc{\dualpoly{\cX}{\cC}}$ and $B:=\spc{\cC}^\perp$. Since we have $\proj{B^\perp}{A}\subset \proj{B^\perp}{\cQ}$, then any basis of $\proj{B^\perp}{A}$ is of cardinality $r:=r(\cX,\cC)$ and can be written as $\proj{B^\perp}{q_1},\dots \proj{B^\perp}{q_r}$ (as constructed by \cref{algorithm space coverage problem vector space}) where $q_1\dots,q_r\in \cQ$, consequently $\cD=\{q_1\dots,q_r\}$ is a smallest cardinality set verifying the desired property.
\subsection{Proof of \cref{cardinality suff data set relatively open set}}
\begin{proof}{Proof.}
From \cref{sufficient data set lowerbound}, we know that the cardinality of any \sufficient\ is lower bounded by $r(\cX,\cC)$, and the discussion above shows that the smallest cardinality of a sufficient data set is upper bounded by $r(\cX,\cC)+1$.
    
\end{proof}

\subsection{Proof of \cref{equivalent data set open}}
\begin{proof}{Proof. }
    Let $q_0\in \R^d$ and $V:=\spc{\cQ}$ a subspace of $\R^d$ such that $\aff{\cQ}=q_0+V$. We have $\Vect \cQ = \Vect (q_0+V)$ where $q_0\in \cQ$. Let $\cD$ a finite subspace of $\Vect \cQ$ containing an element of $\cQ$. Assume without loss of generality that this element is $q_0$. We denote without loss of generality $\cD:=\{q_0,\lambda_1q_0+v_1,\dots,\lambda_t q_0+v_t\}$ where $t\in \mathbb N$, $\lambda_1,\dots,\lambda_t\in \R$, $v_1,\dots,v_t\in V$. Let $M_1,\dots,M_t\in \R^\star$. Let $\cD'=\{q_0,q_0+\frac{1}{M_1}v_1,\dots,q_0+\frac{1}{M_t}v_t\}$. since $q_0\in \cQ$ and $\cQ$ is relatively open, then for $M_1,\dots,M_t$ large enough, we have $\cD'\subset \cQ$. Furthermore, we have  $\cD'\subset \cQ$, $\abs{\cD'}=\abs{\cD}$ and $\Vect \cD=\Vect \cD'$ since $\cD'$ can be obtained by applying invertible linear operations on $\cD$.
\end{proof}

\subsection{Proof of Proposition \ref{covspace vector space corectness relatively open set}} \label{appendix: relatively open alg correctness}
From \cref{cardinality suff data set relatively open set}, we know that the smallest cardinality of a sufficient data set is either $r(\cX,\cC)$ or $r(\cX,\cC)+1$. Assume that $\cQ \cap (\spc{\cC}^\perp +\spc{ \dualpoly{\cX}{\cC}})\setminus \spc{\cC}^\perp = \varnothing.$  We will show that the smallest cardinality of a sufficient data set in this case is $r(\cX,\cC)+1$. We denote $A:=\spc{\dualpoly{\cX}{\cC}}$ and $B=\spc{\cC}^\perp$. Let $\cD\subset \cQ$ such that $\abs{\cD}=r(\cX,\cC)$. Let $q\in \cD$. Let $a\in A\cap (B+\Vect \{q\})$. There exists $\lambda \in \R$ and $b\in B$ such that $a=b+\lambda q$. Assume that $\lambda \neq 0$. We have $q=\frac{1}{\lambda}a-\frac{1}{\lambda}b\in A+B$. Hence, we have $q\in \cQ \cap (A+B)$, consequently since $\cQ \cap (A+B)\setminus B=\varnothing$, then we have $q\in B$, and hence $a\in B$. If $\lambda =0$, then we also we have $a\in B$. Hence, we have $A\cap (B+\Vect \{q\})\subset A \cap B$, i.e. $A\cap (B+\Vect \{q\})= A \cap B$. Let $\cD_0=\cD \setminus \{q\}$. From Lemma \ref{dim intersection inequality}, we have
\begin{align*}
    \dim (A\cap (B+\Vect \cD)) &= \dim (A\cap (B+\Vect \{q\}+\Vect \cD_0))\\
    &\leq \dim (A\cap (B +\Vect \{q\})) + \dim \cD_0\\
    &=\dim (A\cap B) +\dim \cD_0\leq \dim (A\cap B) +r(\cX,\cC)-1\\
    &=\dim A -1.
\end{align*}
Hence, we have $A\not\subset B+\Vect \cD$, which contradicts our initial assumption. Consequently, the cardinality of a minimal \sufficient{} is indeed lower bounded by $1+r(\cX,\cC)$. Furthermore, since $A\subset B+\Vect \cQ$, Algorithm \ref{algorithm space coverage problem vector space} is guaranteed to provide a minimal set $\cD\subset \Vect \cQ$ such that $A\subset B+\Vect \cD$ satisfying $\abs{\cD}=r$. Let $q\in \cQ$. From Lemma \ref{equivalent data set open}, we can construct a data set $\cD'\subset \cQ$ such that $\Vect \cD'=\Vect (\cD \cup \{q\})$,  $\abs{\cD'}=\abs{\cD \cup \{q\}}=r+1$. From the analysis above, $\cD'$ is indeed a minimal \sufficient{}.

Assume now that $\cQ \cap (A+B)\setminus B\neq \varnothing$. Let $q\in \cQ \cap (A+B)\setminus B$. We first prove that $\dim (A\cap (B+\Vect \{q\}))=1+\dim A\cap B$. From Lemma \ref{dim intersection inequality}, we have $\dim (A\cap (B+\Vect \{q\}))\leq 1+\dim (A\cap B)$. Let $a\in A$, $b\in B$ such that $q=a+b$. We have $a=q-b \in A\cap (B+\Vect \{q\})$. Assume that $a\in A\cap B$. There exists $b'\in B$ such that $b'=a=q-b$. Hence, we have $q=b+b'\in B$ which is impossible. Hence, we have $A\cap (B+\Vect \{q\})\not\subset A\cap B$ and consequently $\dim (A\cap (B+\Vect \{q\}))\geq 1+\dim  A\cap B$. This indeed gives $\dim (A\cap (B+\Vect \{q\}))=1+\dim A\cap B$. Let $B'=B+\Vect \{q\}$. We have $\dim A- \dim A\cap B'=r(\cX,\cC)-1$ and $A\subset B' + \Vect \cQ$. Consequently, Algorithm \ref{algorithm space coverage problem vector space} provides a minimal set $\cD\subset \Vect \cQ$ such that $A\subset B' + \Vect \cD$ and $\abs{\cD}=r(\cX,\cC)-1$. Using Lemma \ref{equivalent data set open}, we can construct a data set $\cD'\subset \cQ$ such that $\Vect \cD'=\Vect (\cD \cup \{q\})$,  $\abs{\cD'}=\abs{\cD \cup \{q\}}=r(\cX,\cC)$. From the analysis above, $\cD'$ is indeed a minimal \sufficient{}.

\subsection{Proof of $\Vect \cQ=\Vect \mathrm{relint}~ (\cQ)$ when $\cQ$ is convex} \label{appendix: proof of span Q equals span relint Q}
Since $\relint \cQ \subset \cQ$, we have $\Vect \relint \cQ \subset \Vect \cQ$. Hence, we only need to prove the converse inclusion. Let $q_1,\dots,q_r\in \cQ$ vectors that span $\cQ$. let $q_0 \in \relint \cQ$ and $\varepsilon>0$ such that $\aff{\cQ}\cap B(q_0,\varepsilon)\subset \cQ$. We would like to show that for all $i\in \{1,\dots,r\}$, we have $\frac{1}{2}(q_0+q_i)\in \relint \cQ$. To do that, we show that $\aff{\cQ}\cap B(\frac{1}{2}(q_0+q_i),\frac{1}{2}\varepsilon)\subset \cQ$. let $q\in \aff{\cQ}\cap B(\frac{1}{2}(q_0+q_i),\frac{1}{2}\varepsilon)$, we have 
\begin{align*}
    q=\frac{1}{2}(q_0+q_i) + \left(q-\frac{1}{2}(q_0+q_i)\right)=\frac{1}{2}\underbrace{\left(q_0 + 2 \left(q-\frac{1}{2}(q_0+q_i)\right)\right)}_{\in B(q_0,\varepsilon)\cap \aff{\cQ}\subset \cQ} + \frac{1}{2}\underbrace{q_i}_{\in \cQ}\in \cQ.
\end{align*}
This proves the desired inclusion.

\section{Noisy measurements} \label{sec:noisy-measurements}
In this section, we would like to evaluate the effect of noise on our algorithm, i.e. how do our guarantees change when the evaluations of the queries of some \sufficient{} are observed with noise. In other words, given $\cD=\{q_1,\dots,q_r\}$ with $r\in \mathbb N$ a \sufficient{}, and $\ctrue\in \cC$, if we observe for all $i\in \mathbb N$ the real numbers $o_i:=\ctrue^\top q_i + \varepsilon_i$, where $\varepsilon_i \in \R$, can we quantify how sensitive is the notion of data sufficiency to measurement noise? We show in this section that our algorithm can obtain near optimality guarantees. Let $o=\begin{pmatrix}
    o_1 & \dots & o_r
\end{pmatrix}^\top$ and $Q$ a matrix whose rows are the elements of $\cD$, a \sufficient{}. We consider the following cost vector predictor
\begin{align*}
    \chat(o_1,\dots,o_r):=\arg\min_{c\in \cC}\norm{Qc-o}.\end{align*}

Given the cost vector prediction above, we would like to evaluate the optimality gap of the decision prescribed by a sufficient data set:
\begin{align*}
    \frac{\ctrue^\top \hat{x} - \min_{ x\in X}\ctrue^\top x}{\norm{\ctrue}\diam{\cX}}.
\end{align*}
where $\hat{x}\in \arg\min_{x\in \cX}\chat(o_1,\dots,o_r)^\top x$ and $\diam{\cX} = \sup_{x,y\in \cX}\norm{x-y}$. The quantity above is the optimality gap induced by our decision $\hat{x}$ normalized by the norm of the true cost vector and the size of $\cX$. This normalization is important because $\ctrue^\top \hat{x} - \min_{ x\in X}\ctrue^\top x$ scales with the norm of $\ctrue$ and with any multiplication of $\cX$ by a positive scalar. We show that our approach provides near optimality guarantees in the noisy case. In particular, when the noise $\varepsilon$ is small enough, we can prove that solving $\min_{x\in \cX}\chat(o_1,\dots,o_r)^\top x$ still recovers an exact solution to $\min_{x\in \cX}\ctrue^\top x$, and an approximate solution in general with error that scales with $\norm{\varepsilon}$. 

\begin{proposition} \label{prop:bound on noisy case}
    Given $\cC\subset \R^d$, $\ctrue \in \cC$ $D:=\{q_1,\dots,q_r\}$ a sufficient decision data set, $Q$ a matrix whose rows are the elements of $D$, $\varepsilon\in \R^d$, observations $o:=Q\ctrue + \varepsilon$,  $\hat{c}(o_1,\dots,o_r):=\arg\min_{c\in \cC}\norm{Qc-o}$, and $\hat{x}$ an element of $\arg\min_{x\in \cX}\hat{c}(o_1,\dots,o_r)^\top x$, there exists $\kappa>0$ such that
    \begin{align*}
          \frac{\ctrue^\top \hat{x} - \min_{ x\in X}\ctrue^\top x}{\norm{\ctrue}\diam{\cX}}\leq \begin{cases}
          0 & \text{ if }\norm{\varepsilon}<\kappa\\
              \frac{2}{\lambda_{\mathrm{min}}(D) \norm{\ctrue}}\norm{\varepsilon}& \text{ else, }
          \end{cases}
    \end{align*}
    where $\diam{\cX}$ is the diameter of $\cX$ and $\lambda_{\mathrm{min}}(D)$ is the smallest singular value of $Q$.
\end{proposition}
Here, $\lambda_{\mathrm{min}}(D)$ quantify how much redundancy there is in the data set: if all the data points in a data set are nearly collinear, $\lambda_{\mathrm{min}}(D)$ will be small, and hence the error will be large. This is intuitive because for a fixed value of the noise $\varepsilon$, every additional data point's additional information is inversely proportional to how redundant it is with the rest of the data set (i.e. $\lambda_{\mathrm{min}}(D)$) and the noise level.

\section{Proofs for \cref{sec:noisy-measurements}}

\subsection{Proof of Proposition \ref{prop:bound on noisy case}}
\begin{proof}{Proof. }
    Let $\eta := Q\hat{c}(o_1,\dots,o_r) - Q \ctrue$. Since $\hat{c}(o_1,\dots,o_r)=\arg\min_{c\in \cC}\norm{Qc-o}$, then we have $\norm{Q\hat{c}(o_1,\dots,o_r) -Q\ctrue - \varepsilon}\leq \norm{\varepsilon}$. Hence, we have  $\norm{Q\hat{c}(o_1,\dots,o_r) -Q\ctrue}\leq 2\norm{\varepsilon}$ and consequently $\norm{\eta}\leq 2 \norm{\varepsilon}$. 
    We would like to show that the distance between the projections of $\ctrue$ and $\hat{c}(o_1,\dots,o_r)$ in the span of $\cD$ is upper bounded by $O(\norm{\varepsilon})$. Consider $\alpha_{\text{true}},\hat{\alpha}\in \R^r$ such that $\hat{c}(o_1,\dots,o_r)_{\Vect \cD}=Q^\top \hat{\alpha}$ and $c_{\text{true},\Vect \cD}=Q^\top \alpha_{\text{true}}$. We have 
    \begin{align*}
        Q\hat{c}(o_1,\dots,o_r)-Q\ctrue=\eta &\implies Q(\hat{c}(o_1,\dots,o_r)_{\Vect \cD - c_{\text{true},\Vect \cD}})\\
        &\implies QQ^\top (\hat{\alpha}-\alpha_{\text{true}})=\eta\\
        &\implies \hat{\alpha}-\alpha_{\text{true}}=(QQ^\top)^{-1}\eta\\
        &\implies \hat{c}(o_1,\dots,o_r)_{\Vect \cD}-c_{\text{true},\Vect \cD}=Q^\top (QQ^\top)^{-1}\eta.
    \end{align*}
    Let $U\in \R^{r\times r}$, $V\in \R^{d\times d}$ and $\Sigma\in \R^{r\times d}$ such that $U,V$ are orthogonal matrices and for all $(i,j)\in [r]\times [d]$
    \begin{align*}
        \Sigma_{ij}=\begin{cases}\text{$\sigma_i$ the $i-$th singular value of $Q$} & \text{if }i=j\\ 0 &\text{else,}\end{cases}
    \end{align*}
    and $Q=U\Sigma V^\top$. We have 
    \begin{align*}
        Q^\top (QQ^\top)^{-1}&=U\Sigma V^\top (U\Sigma V^\top V \Sigma^\top U^\top)^{-1}\\
        &=U\Sigma V^\top (U\Sigma \Sigma^\top U^\top)^{-1}\\
        &=V\Sigma^\top  U^\top U(\Sigma \Sigma^\top)^{-1} U^\top\\
        &=V\Sigma^\top (\Sigma \Sigma^\top)^{-1} U^\top=V\Sigma'U^\top,
    \end{align*}
    where $\Sigma'\in \R^{d\times r}$ satisfies
    \begin{align*}
        \Sigma'_{ij}=\begin{cases}\text{$\frac{1}{\sigma_i}$ the $i-$th singular value of $Q$} & \text{if }i=j\\ 0 &\text{else.}\end{cases}
    \end{align*}
    Let $\lambda_{\mathrm{min}}(D)$ the smallest singular value of $\cQ$. The calculations above gives, when $\norm{.}$ is the $L^2$ norm,
    \begin{align*}
        \norm{c_{\text{true},\Vect \cD}-\hat{c}(o_1,\dots,o_r)_{\Vect \cD}}=\norm{Q^\top (QQ^\top)^{-1}\eta}\leq \norm{Q^\top (QQ^\top)^{-1}}\cdot \norm{\eta}\leq \frac{2}{\lambda_{\mathrm{min}}(D)}\norm{\varepsilon}.
    \end{align*}
    
    Furthermore, since $\cD$ is a \sufficient{} we have from Theorem \ref{thm: relatively open characterization} that $ \spc{\dualpoly{\cX}{\cC}}\perp \spc{\cC}\cap (\Vect \cD)$ and $\ctrue,\hat{c}(o_1,\dots,o_r)\in \cC$, then we have for any $\xtrue\in \arg\min_{x\in \cX}\ctrue^\top x$ and $\xhat \in \arg\min_{x\in \cX}\hat{c}(o_1,\dots,o_r)^\top x$
    \begin{align*}
        \ctrue^\top (\xhat - \xtrue)&=(\ctrue-\hat{c}(o_1,\dots,o_r))^\top(\xhat - \xtrue) + \underbrace{\hat{c}(o_1,\dots,o_r)^\top(\xhat - \xtrue)}_{\leq 0}\\
        &\leq (c_{\text{true},\Vect \cD}-\hat{c}(o_1,\dots,o_r)_{\Vect \cD})^\top(\xhat - \xtrue) \\&
        +\underbrace{(c_{\text{true},(\Vect \cD)^\perp}-\hat{c}(o_1,\dots,o_r)_{(\Vect \cD)^\perp}}_{\in \spc{\cC}\cap (\Vect \cD)^\perp})^\top\underbrace{(\xhat - \xtrue)}_{\in \spc{\dualpoly{\cX}{\cC}}}\\
        &\leq \norm{c_{\text{true},\Vect \cD}-\hat{c}(o_1,\dots,o_r)_{\Vect \cD}}\cdot \norm{\xhat - \xtrue}\\
        &\leq \frac{2}{\lambda_{\mathrm{min}}(D)}\norm{\varepsilon} \times \diam{\cX}.
    \end{align*}
    Finally, this implies
    \begin{align*}
          \frac{\ctrue^\top \hat{x} - \min_{ x\in X}\ctrue^\top x}{\norm{\ctrue}\diam{\cX}}\leq \frac{2}{\lambda_{\mathrm{min}}(Q) \norm{\ctrue}}\norm{\varepsilon}.
    \end{align*}
    We now prove that for any $\ctrue \in \cC$, when $\norm{\varepsilon}$ is small enough, we have 
    \begin{align*}
        \frac{\ctrue^\top \hat{x} - \min_{ x\in X}\ctrue^\top x}{\norm{\ctrue}\diam{\cX}}=0.
    \end{align*}
    In order to do that, we prove the following lemma.
    \begin{lemma}
        Assume that $\cC$ is compact. Let $\cD$ a \sufficient{} for $\cC\subset \R^d$. Let $c\in \cC$. There exists $\mu>0$ such that for any $c'\in \cC$ such that $\norm{c_{\Vect \cD}-c'_{\Vect \cD}}<\mu$, we have $\arg\min_{x\in \cX}c'^\top x \subset \arg\min_{x\in \cX}c^\top x$.
    \end{lemma}
    \begin{proof}{Proof. }
        Assume that the result does not hold, i.e. there exists a sequence $c'_n\in \cC$ such that $c'_{n,\Vect D}$ converges to $c_{\Vect D}$, and for all $n\in \mathbb N$, there exists $x'\in \cX^\angle$ such that $x'\in \arg\min_{x\in \cX}c'^\top_n x \setminus \arg\min_{x\in \cX}c^\top x $. Since the number of extreme points in $\cX$ are finite, there exists $x'\in \cX^\angle$ and a strictly increasing map $\varphi:\mathbb N \longmapsto \mathbb N$ such that for all $n\in \mathbb N$, we have $x'\in \arg\min_{x\in \cX}c'^\top_{\varphi(n)} x \setminus \arg\min_{x\in \cX}c^\top x $. Since $\cC$ is compact, we can assume without loss of generality that the sequence $c'_{\varphi(n)}$ is convergent to some $c'\in \cC$ (it suffices to extract another time a converging sequence from $c'_{\varphi(n)}$). Consequently, since for all $n\in \mathbb N$, $c'_{\varphi(n)}\in \Lambda(x')$, and $\Lambda(x')$ is closed, then $c'\in \Lambda(x')$. Furthermore, we have $c\not\in \Lambda(x)$, and $c_{\Vect \cD}=c'_{\Vect \cD}$, which means that $\arg\min_{x\in \cX}c^\top x=\arg\min_{x\in \cX}c'^\top x$. This implies that $x'\in \arg\min_{x\in \cX}c^\top x$, i.e. $c\in \Lambda(x)$ which is impossible.
    \end{proof}
    
    Hence, when $\norm{\varepsilon}\leq \frac{\mu \lambda_{\mathrm{min}}(D)}{2}$, we have 
    \begin{align*}
        \norm{c_{\text{true},\Vect \cD}-\hat{c}(o_1,\dots,o_r)_{\Vect \cD}}<\mu,
    \end{align*}
    i.e. $\arg\min_{x\in \cX}\hat{c}(o_1,\dots,o_r)^\top \subset \arg\min_{x\in \cX}\ctrue^\top x$. This gives 
     \begin{align*}
        \frac{\ctrue^\top \hat{x} - \min_{ x\in X}\ctrue^\top x}{\norm{\ctrue}\diam{\cX}}=0.
    \end{align*}
\end{proof}

 \section{Some useful lemmas}

  \begin{lemma} \label{cone equivalence}
            Assume that $\cX$ is bounded. For every $c\in \R^d$, $\arg\min_{x\in \cX}c^\top x$ is a polyhedron, and all of its extreme points are extreme points in $\cX$. Recall the optimality cones $\Lambda(x^\star)$ of all $x^\star \in \cX^\angle$ defined in \cref{optimality cone def}.  For every $c,c'\in \R^d$, the following equivalence holds.
            \begin{align*}
                \arg\min_{x\in \cX}c^\top x=\arg\min_{x\in \cX}c'^\top x \Longleftrightarrow \forall x \in \cX^\angle,\; \left(c\in \Lambda (x)\Longleftrightarrow c'\in \Lambda(x)\right).
            \end{align*}
        \end{lemma}
        
        \begin{proof}{Proof. }
            We first show that for any $c\in \R^d$, any extreme point in $\arg\min_{x\in \cX}c^\top x$ is in $\cX^\angle$. Let $x\in \cX$ be an extreme point of $\arg\min_{x\in \cX}c^\top x$. Assume that $x$ is not an extreme point in $\cX$. Hence, there exists $u\in \R^d$ such that $x\pm u\in \cX$. If $c^\top u\neq 0$ we have $c^\top (x-u)<c^\top x$ or $c^\top (x+u)<c^\top x$. This means that $x\not\in \arg\min_{x\in \cX}c^\top x$ which is impossible.
            If $c^\top u=0$, then $c\pm u\in \arg\min_{x\in \cX}c^\top x$, which is also impossible since $x$ is an extreme point in $\arg\min_{x\in \cX}c^\top x$. Hence, since $\arg\min_{x\in \cX}c^\top x$ is convex, it's the convex hull of its extreme points.

            Consequently, for any $c,c'\in \R^d$, $\arg\min_{x\in \cX}c^\top x=\arg\min_{x\in \cX}c'^\top x$ if and only if these two sets have the same set of extreme points. Furthermore, for any $x\in \cX^\angle$, we have $c\in \Lambda(x)$ if on and only if $x\in \arg\min_{x\in \cX}c^\top x$. Hence, the desired result immediately follows:
            \begin{align*}
                \arg\min_{x\in \cX}c^\top x=\arg\min_{x\in \cX}c'^\top x&\Longleftrightarrow \left(\forall x\in \cX^\angle,\; x\in \arg\min_{x\in \cX}c^\top x \Longleftrightarrow x\in \arg\min_{x\in \cX}c'^\top x \right)\\
                & \Longleftrightarrow \left(\forall x \in \cX^\angle,\; c\in \Lambda(x)\Longleftrightarrow c'\in \Lambda(x)\right)
            \end{align*}
        \end{proof}
\begin{lemma} \label{lemma:feasible directions in general polyhedron}
            Let $r=\dim F_0 \cap \Ker A$. For any $x\in \cX$, there exists a set $V=\{\delta_1,\dots,\delta_r\}\subset \FD(x)$, such that $V$ is a basis of $F_0 \cap \Ker A$. In particular, the set of extreme directions of $\FD(x)$ spans $F_0\cap \Ker A$.
        \end{lemma}
\begin{proof}{Proof. }
            Let $x\in \cX$. Let $\cX^\angle$ be the set of extreme points of $\cX$, and $D^\angle$ be the set of extreme rays of $\cX$. For every $x^\angle \in \cX^\angle$ and $\delta^\angle \in D^\angle$, we have $x^\angle \geq 0$ and $\delta^\angle \geq 0$. Let $\{\alpha_{x^\angle}\}_{x^\angle \in \cX^\angle}\subset \R^*_+$ and $\{\alpha_{\delta^\angle}\}_{\delta^\angle \in D^\angle}\subset \R^*_+$ be a set of strictly positive numbers such that $\sum_{x^\angle \in \cX^\angle}^{}\alpha_{x^\angle}=1$. We define 
            \begin{align*}
                \overline{x}=\sum_{x^\angle \in \cX^\angle}^{}\alpha_{x^\angle}x^\angle + \sum_{\delta^\angle \in D^\angle}^{}\alpha_{\delta^\angle}\delta^\angle\in \cX.
            \end{align*}

        We have for any $i\in I_0$, $\overline{x}_i>0$. Indeed, for any $i\in I_0$, if $\overline{x}_i=0$, then for every $x^\angle \in \cX^\angle$ and $\delta^\angle \in D^\angle$, $x_i^\angle=\delta_i^\angle =0$ and hence for any $x'\in \cX$, $x'_i=0$ i.e. $i\not \in I_0$ which is impossible. Let $$\varepsilon:=\frac{1}{2}\min_{i\in I_0}\overline{x}_i>0,$$ for any $\delta \in F_0 \cap \Ker A$ such that $\norm{\delta}<\varepsilon$, we have $A(\overline{x}+\delta)=A\overline{x}=b$, and for every $i\in [d]$, if $i\in I_0$, then $\overline{x}_i + \delta _i >0$ and if $i\not\in I_0$, $\overline{x}_i=\delta_i=0$ and consequently $x+\delta \geq 0$, i.e. $\delta$ is a feasible direction for $\overline{x}$. Hence, every element of $B(0,\varepsilon)\cap F_0 \cap \Ker A$ is a feasible direction for $\overline{x}$, and consequently any element of $F_0 \cap \Ker A$. Let $x\in \cX$, and $v_1,\dots, v_r$ a basis of $F_0\cap \Ker A$ such that for every $i\in [r],\; \norm{v_i}=1$. Let $\eta \in \R^*_+$ small enough such that $\forall i \in [r]$, $\overline{x} + \eta \delta _i \in \cX$. We would like to show that for a well-chosen value of $\eta$, the following set of feasible directions for $x$, $\{\overline{x}+\eta v_1 - x,\dots,\overline{x}+\eta v_r - x,\}$ is a basis of $F_0 \cap \Ker A$. Since $\overline{x}-x\in F_0 \cap \Ker A$, we consider $\beta_1,\dots,\beta_r\in \R$ such that 
        \begin{align*}
            \overline{x}-x=\sum_{i=1}^{r}\beta_i v_i.
        \end{align*}
        Let $\alpha_1,\dots,\alpha_r\in \R$. We have 
        \begin{align*}
            \sum_{i=1}^{r}\alpha_i(\overline{x}+\eta v_i-x)=0 &\Longrightarrow \underbrace{\left(\sum_{i=1}^{r}\alpha_i\right)}_{:=A}(\overline{x}-x)+\eta \sum_{i=1}^{r}\alpha_i v_i=0\\
            &\Longrightarrow \sum_{i=1}^{r}(A\beta_i + \eta \alpha_i)v_i=0\\
            &\Longrightarrow \forall i \in [r],\; A\beta_i + \eta \alpha_i=0\\
            &\Longrightarrow \underbrace{ \forall i\in [r], \alpha_i=0 \text{ or }\left( A\neq 0 \text{ and } \forall i \in [r],\;\frac{\alpha_i}{A}=-\frac{\beta_i}{\eta}\right)}_{(*)}.
        \end{align*}
        By summing over $i$ in the last equality above, we get
        \begin{align*}
            (*)&\Longrightarrow  \forall i\in [r], \alpha_i=0 \text{ or }\left( A\neq 0 \text{ and } 1=-\frac{1}{\eta}\smash{\underbrace{\sum_{i=1}^{r}\beta_i}_{B}} \right)
        \end{align*}
        
        \vspace{5mm}
        
        The equality above is equivalent to $\eta = -B$. Hence, it suffices to take $\eta\neq -B$ and $\eta$ small enough to ensure that $\{\overline{x}+\eta v_1 - x,\dots,\overline{x}+\eta v_r - x,\}$ is linearly independent and is a set of feasible directions for $x$, and consequently since all of these vectors are elements of $F_0 \cap \Ker A$, and there are $r=\dim F_0 \cap \Ker A$ of them, $\{\overline{x}+\eta v_1 - x,\dots,\overline{x}+\eta v_r - x,\}$ is indeed a set of feasible directions for $x$ that is a basis of $F_0\cap \Ker A$.
        \end{proof}
    \begin{lemma} \label{dim intersection inequality}
        For any subspaces $A,B,C \subset \R^d$, we have $\dim A\cap(B+C)\leq \dim A\cap B + \dim C$.
    \end{lemma}
    \begin{proof}{Proof. }
        We have 
        \begin{align*}
            \dim (A\cap (B+C))&=\dim A + \dim (B+C) - \dim (A+B+C)\\
            &\leq \dim A + \dim B + \dim C - \dim (A+ B)\\
            &=\dim A \cap B + \dim C.
        \end{align*}
        
    \end{proof}
\section{Making a decision using measurements} \label{appendix: making a decision using measurements}

By combining our previous analysis and results, we derive two algorithms that enable optimal decision-making based on prior knowledge of $c$ (uncertainty set $\cC$) and a prescribed measurement of data points guided by our theory. Specifically, Algorithms \ref{algorithm space coverage problem vector space}, \ref{alg:relatively open}, \ref{alg: sufficient datatset for extreme points}, \ref{alg:data collection algorithm} identify which elements of $\cQ$ should be queried, using the output of \cref{alg:LP_case}, and returns a set $\cD$ containing these elements. Then, \cref{alg: decision-making} produces a decision based on the information obtained by evaluating the objective function at the elements of $\cD$.

\begin{algorithm}[h] 
    \caption{Decision-making with a \sufficient{}}
     \label{alg: decision-making}
    \KwIn{Decision set $\cX$, Uncertainty Set $\cC$, \Sufficient\ $\cD = \{q_1,\ldots,q_N\}$, Oracle $\pi$ such that for any $q\in \cQ$, $\pi(q)=c^\top q$ where $c$ is the ground truth.
}
    \KwOut{A decision $\hat{x}\in \arg\min_{x\in \cX}c^\top x$.}

    $o_1,\ldots,o_N \leftarrow \pi(q_1),\ldots,\pi(q_N)$ 

    Compute $\hat{c}\in \argmin \{\sum_{i=1}^{N}(c'^\top q_i - o_i)^2 \; : \; c' \in \cC\}$.

    \textbf{return} $\hat{x} \in \argmin_{x\in \cX}\hat{c}^\top x$.
\end{algorithm}
\end{APPENDICES}
\end{document}